\documentclass[reqno]{amsart}
\usepackage{amssymb}
\usepackage{graphicx}

\usepackage[usenames, dvipsnames]{color}
\usepackage{verbatim}
\usepackage{mathrsfs}
\usepackage{bm}
\usepackage{cite}

\numberwithin{equation}{section}

\newtheorem{theorem}{Theorem}[section]
\newtheorem{corollary}[theorem]{Corollary}
\newtheorem{lemma}[theorem]{Lemma}
\newtheorem{prop}[theorem]{Proposition}

\theoremstyle{definition}
\newtheorem{remark}[theorem]{Remark}

\theoremstyle{definition}

\theoremstyle{definition}

\makeatletter
\def\dashint{\operatorname%
{\,\,\text{\bf-}\kern-.98em\DOTSI\intop\ilimits@\!\!}}
\makeatother

\def\\det{\text{\det}}

\def\.5{\frac{1}{2}}

\newcommand{\RN}[1]{%
  \textup{\uppercase\expandafter{\romannumeral#1}}%
}

\renewcommand{\epsilon}{\varepsilon}

\newcounter{marnote}

\begin{document}

\title[The convexity of inclusions and gradient's concentration]{The convexity of inclusions and gradient's concentration for Lam\'{e} systems with partially infinite coefficients}
\author[Y.Y. Hou]{Yuanyuan Hou}
\address[Y.Y. Hou]{School of Mathematical Sciences, Beijing Normal University, Laboratory of Mathematics and Complex Systems, Ministry of Education, Beijing 100875, China.}
\email{yyhou@mail.bnu.edu.cn}

\author[H.J. Ju]{Hongjie Ju}
\address[H.J. Ju]{School of  Sciences, Beijing University of Posts  and Telecommunications,
Beijing 100876, China}
\email{hjju@bupt.edu.cn}
\thanks{H.J. Ju was partially supported by NSFC (11471050).}

\author[H.G. Li]{Haigang Li}
\address[H.G. Li]{School of Mathematical Sciences, Beijing Normal University, Laboratory of MathematiCs and Complex Systems, Ministry of Education, Beijing 100875, China.}
\email{hgli@bnu.edu.cn. Corresponding author.}
\thanks{H.G. Li was partially supported by  NSFC (11571042, 11631002), Fok Ying Tung Education Foundation (151003).}


\date{\today} 


\maketitle
\begin{abstract}
It is interesting to study the stress concentration between two adjacent stiff inclusions in composite materials, which can be modeled by the Lam\'{e} system with partially infinite coefficients. To overcome the difficulty from the lack of maximum principle for elliptic systems, we use the energy method and an iteration technique to study the gradient estimates of the solution. We first find a novel phenomenon that the gradient will not blow up any more once these two adjacent inclusions fail to be locally relatively strictly convex, namely, the top and bottom boundaries of the narrow region are partially ``flat". This is contrary to our expectation. In order to further explore the blow-up mechanism of the gradient, we next investigate two adjacent inclusions with relative convexity of order $m$ and finally reveal an underlying relationship between the blow-up rate of the stress and the order of the relative convexity of the subdomains in all dimensions.
\end{abstract}

\section{Introduction and main results}

The convexity plays a central role in many questions in analysis. The purpose of this paper is mainly to investigate the significant role of the relative convexity between two adjacent inclusions in the blow-up analysis of the stress in high-contrast fiber-reinforced composite materials, where the inclusions are frequently spaced very closely and even touching. This work is motivated by issue of material failure initiation, where it is well known that high concentration phenomenon of mechanical loads in the extreme loads will be amplified by the composite microstructure, for example, the narrow region between two adjacent inclusions. However, in this paper we first find a novel phenomenon, contrary to our expectation. Whenever the narrow region has certain partially ``flat" top and bottom boundaries (see Figure \ref{fig1}), we prove that the gradient of the solutions to Lam\'{e} systems with partially infinite coefficients, is bounded by some positive constant, independent of the distance between the inclusions, rather than blows up as one might expect. In order to further explore the blow-up mechanism of the stress, we next investigate two adjacent inclusions with relative convexity of order $m$, and finally reveal an underlying relationship between the blow-up rate of the stress and the order of the relative convexity of the subdomains in all dimensions.  This shows that the relative convexity between inclusions is critical for the stress concentration phenomenon in composite materials.

For strictly convex inclusions, especially for circular inclusions, there have been many important works on the gradient estimates for solution to a class of divergence form elliptic equations and systems with discontinuous coefficients, arising from the study of composite media. For two adjacent disks in dimension two with $\varepsilon$ apart, Keller \cite{k1} was the first to use analysis to estimate the effective properties of particle reinforced composites. In \cite{basl}, Babu\v{s}ka, Andersson, Smith, and Levin numerically analyzed the initiation and growth of damage in composite materials, where the Lam\'e system is assumed. Bonnetier and Vogelius \cite{bv} and Li and Vogelius \cite{lv} proved the uniform boundedness of $|\nabla{u}|$ regardless of $\varepsilon$ provided that the coefficients stay away from $0$ and $\infty$.
Li and Nirenberg \cite{ln} extended the results in \cite{lv} to general divergence form second order elliptic systems including systems of linear elasticity.

\begin{figure}[t]
\begin{minipage}[c]{0.9\linewidth}
\centering
\includegraphics[width=1.8in]{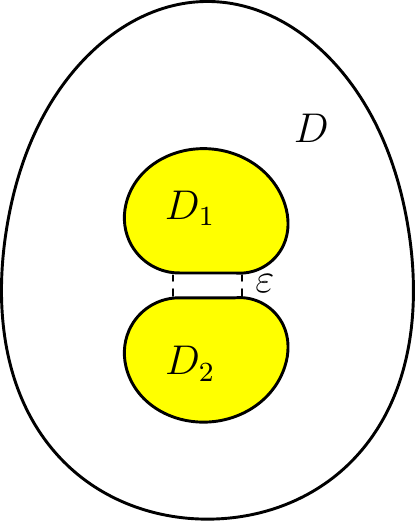}
\caption{\small Two adjacent inclusions with partially ``flat" boundaries.}
\label{fig1}
\end{minipage}
\end{figure}

On the other hand, in order to investigate the high-contrast conductivity problem and establish the relationship between $|\nabla u|$ and the distance $\varepsilon$, Ammari, Kang, and Lim \cite{akl} studied two close-to-touching  disks whose conductivity degenerate to $\infty$ or $0$, a lower bound on $|\nabla u|$ was constructed there showing blow-up of order $\varepsilon^{-1/2}$ in dimension two. Subsequently, it has been proved by many mathematicians that the generic blow-up rate of $|\nabla{u}|$ is $\varepsilon^{-1/2}$ in dimension $d=2$, $|\varepsilon\log\varepsilon|^{-1}$ in dimension $d=3$, and $\varepsilon^{-1}$ in dimensions $d\geq4$. See Ammari, Kang, Lee, Lee and Lim \cite{aklll}, Bao, Li and Yin \cite{bly1,bly2}, as well as Lim and Yun \cite{ly,ly2}, Yun \cite{y1,y2,y3}, Lim and Yu \cite{lyu}. The corresponding boundary estimates when one inclusion close to the boundary was established in \cite{aklll,LX}. Further, more detailed, characterizations of the singular behavior of gradient of $u$ have been obtained by Ammari, Ciraolo, Kang, Lee and Yun \cite{ackly}, Ammari, Kang, Lee, Lim and Zribi \cite{akllz}, Bonnetier and Triki \cite{bt1,bt2}, Gorb and Novikov \cite{gn} and Kang, Lim and Yun \cite{kly,kly2}. However, for the linear elasticity case, because of the lack of maximum principle, an essential tool to deal with the scalar case, there is no progress until Bao, Li and Li's work \cite{bll1, bll2}. They developed an iteration technique with respect to the energy estimate and obtained the pointwise upper bound of the gradient of solution to the Lam\'{e} system with partially infinite coefficients, and showed the same blow-up rate as the scalar case. The boundary estimates was studied in \cite{bjl} by Bao, Ju and Li. Recently, Kang and Yu \cite{ky} by using the layer potential techniques and the singular functions obtained a lower bound of the gradient of solution in dimension two showing that the blow-up rate obtained in \cite{bll1} is optimal. For more related work on elliptic and parabolic equations and systems from composites, see \cite{adkl,bc,dongli,dongzhang,gb1,gb2,kly0,kyun,kl,ll,llby,m} and the references therein.

As we have mentioned before, in all of the above known work, the strict convexity of the inclusions (or at least the strictly relative convexity of two adjacent inclusions) is assumed. Interestingly, when the inclusions are only convex but not strictly convex (see Figure \ref{fig1}), we prove that $|\nabla u|$ is uniformly bounded with respect to $\varepsilon$ once the area of the flat boundaries is positive, which implies that blow-up will not occur any more. To the best of our knowledge, this is a new phenomenon in the blow-up analysis of fiber-reinforced composite.  The corresponding result for perfect conductivity problem may refer to \cite{jlx}.

To describe the problem and results, we first fix our domain and notations. Let $D$ be a bounded open set in $\mathbb R^{d}$ $(d\geq2)$ that contains a pair of (touching) subdomains $D_{1}^{0}$ and $D_{2}^{0}$  with $C^{2,\alpha}$ $(0<\alpha<1)$ boundaries and far away from $\partial D$. We assume that $D_{1}^{0}$ and $D_{2}^{0}$ are convex but not strictly convex, and have a common flat boundary $\Sigma'$, such that
$$\partial D_{1}^{0}\cap\partial D_{2}^{0}=\Sigma'\subset\mathbb R^{d-1},$$
and
$$D_{1}^{0}\subset\{(x',x_{d})\in\mathbb R^{d}| x_{d}>0\},\quad D_{2}^{0}\subset\{(x',x_{d})\in\mathbb R^{d}| x_{d}<0\}.$$
Here we use superscript prime to denote the $(d-1)$-dimensional domains and variables, such as $\Sigma'$ and $x'$.
We also assume that $\Sigma'$ is a bounded convex domain in $\mathbb R^{d-1}$, which  can contain an $(d-1)$-dimensional ball. We set the center of the mass of $\Sigma'$ to be the origin. We also assume that  the $C^{2,\alpha}$ norms of $\partial{D}_{1}^{0}$, $\partial{D}_{2}^{0}$ and  $\partial{D}$ are bounded by some positive constant. By translating $D_{1}^{0}$ by a positive number $\varepsilon$ along the positive direction of $x_{d}$-axis, while $D_{2}^{0}$ is fixed, we obtain $D_{1}^{\varepsilon}$, that is,
$$D_{1}^{\varepsilon}:=D_{1}^{0}+(0',\varepsilon).$$
When there is no possibility of confusion, we drop superscripts and denote
$$D_{1}:=D_{1}^{\varepsilon},\qquad D_{2}:=D_{2}^{0}.$$
Set
$$\Sigma:=\Sigma'\times(0,\varepsilon),\quad \mbox{and}\quad~\Omega:=D\setminus\overline{D_1\cup D_2}.$$

We assume that $\Omega$ and $D_{1}\cup D_{2}$ are occupied, respectively, by two different isotropic and homogeneous materials with different Lam\'{e} constants $(\lambda, \mu)$ and $(\lambda_1, \mu_1)$. Then the elasticity tensors for the background and the inclusion can be written, respectively, as $\mathbb{C}^0$ and $\mathbb{C}^1$, with
$$C_{ijkl}^0=\lambda\delta_{ij}\delta_{kl} +\mu(\delta_{ik}\delta_{jl}+\delta_{il}\delta_{jk}),$$
and
$$C_{ijkl}^1=\lambda_1\delta_{ij}\delta_{kl} +\mu_1(\delta_{ik}\delta_{jl}+\delta_{il}\delta_{jk}),$$
where $i, j, k, l=1,2,\cdots,d$ and $\delta_{ij}$ is the kronecker symbol: $\delta_{ij}=0$ for $i\neq j$, $\delta_{ij}=1$ for $i=j$.

Let $u=(u^1, u^2,\cdots,u^{d})^T:D\rightarrow\mathbb{R}^{d}$ denote the displacement field. For a given vector valued function $\varphi=(\varphi^1,\varphi^2,\cdots,\varphi^{d})^{T}$, we consider the following Dirichlet problem for the Lam\'{e} system:
 \begin{align}\label{Lame}
\begin{cases}
\nabla\cdot \left((\chi_{\Omega}\mathbb{C}^0+\chi_{D_{1}\cup{D}_{2}}\mathbb{C}^1)e(u)\right)=0,&\hbox{in}~~D,\\
u=\varphi, &\hbox{on}~~\partial{D},
\end{cases}
\end{align}
where $\chi_{\Omega}$ is the characteristic function of $\Omega\subset \mathbb{R}^{d}$,
$$e(u)=\frac{1}{2}(\nabla u+(\nabla u)^T)$$
is the strain tensor. Assume that the standard ellipticity condition holds for (\ref{Lame}), that is,
\begin{align*}
\mu>0,\quad d\lambda+2\mu>0,\quad \mu_1>0,\quad d\lambda_1+2\mu_1>0.
\end{align*}
For $\varphi\in H^1(D; \mathbb{R}^{d})$, it is well known that there exists a unique solution $u\in H^1(D; \mathbb{R}^{d})$ to the Dirichlet problem (\ref{Lame}), which is also the minimizer of the energy functional
$$J_1[u]:=\frac{1}{2}\int_\Omega \left((\chi_{\Omega}\mathbb{C}^0+\chi_{D_{1}\cup{D}_{2}}\mathbb{C}^1)e(u), e(u)\right)dx $$
on
\begin{align*}
H^1_\varphi(D; \mathbb{R}^{d}):=\left\{u\in  H^1(D; \mathbb{R}^{d})~\big|~ u-\varphi\in  H^1_0(D; \mathbb{R}^{d})\right\}.
\end{align*}

Introduce the linear space of rigid displacement in $\mathbb{R}^{d}$:
$$\Psi:=\{\psi\in C^1(\mathbb{R}^{d}; \mathbb{R}^{d})\ |\ \nabla\psi+(\nabla\psi)^T=0\}.$$
Using $e_{1},\cdots,e_{d}$ to denote the standard basis of $\mathbb{R}^{d}$, then $$\left\{~e_{i},~x_{j}e_{k}-x_{k}e_{j}~\big|~1\leq\,i\leq\,d,~1\leq\,j<k\leq\,d~\right\}$$ is a basis of $\Psi$. Denote this basis of $\Psi$ as $\left\{\psi_{\alpha}~|~\alpha=1,2,\cdots,\frac{d(d+1)}{2}\right\}$. For fixed $\lambda$ and $\mu$ satisfying $\mu>0$ and $d\lambda+2\mu>0$, denote $u_{\lambda_1,\mu_1}$ as the solution of (\ref{Lame}). Then similarly as in the Appendix of \cite{bll1}, we also have
\begin{align*}
u_{\lambda_1,\mu_1}\rightarrow u\quad\hbox{in}\ H^1(D; \mathbb{R}^{d}),\quad \hbox{as}\ \min\{\mu_1, d\lambda_1+2\mu_1\}\rightarrow\infty,
\end{align*}
where $u$ is a $H^1(D; \mathbb{R}^{d})$ solution of
 \begin{align}\label{maineqn}
\begin{cases}
\mathcal{L}_{\lambda, \mu}u:=\nabla\cdot(\mathbb{C}^0e(u))=0,\quad&\hbox{in}\ \Omega,\\
u|_{+}=u|_{-},&\hbox{on}\ \partial{D}_{i},i=1,2,\\
e(u)=0,&\hbox{in}~~D_{i},i=1,2,\\
\int_{\partial{D}_{i}}\frac{\partial u}{\partial \nu_0}\Big|_{+}\cdot\psi_{\alpha}=0,&i=1,2,\alpha=1,2,\cdots,\frac{d(d+1)}{2},\\
u=\varphi,&\hbox{on}\ \partial{D},
\end{cases}
\end{align}
where $\varphi\in C^{2}(\partial{D}; \mathbb{R}^{d})$ is a given function,
\begin{align*}
\left(\mathcal{L_{\lambda, \mu}}u\right)_k=\mu\Delta u^{k}+(\lambda+\mu)\partial_{x_k}(\nabla\cdot u),\quad k=1, \cdots, d,
\end{align*}
and
\begin{align*}
\frac{\partial u}{\partial \nu_0}\Big|_{+}&:=(\mathbb{C}^0e(u))\vec{n}=\lambda(\nabla\cdot u)\vec{n}+\mu(\nabla u+(\nabla u)^T)\vec{n},
\end{align*}
and $\vec{n}$ is the unit outer normal of $D_{i}$, $i=1,2$. Here and throughout this paper the subscript $\pm$ indicates the limit from outside and inside the domain, respectively. The existence, uniqueness and regularity of weak solutions to (\ref{maineqn}) can be proved by the argument with a minor modification on assumption on the subdomain in the Appendix of \cite{bll1}. In particular, the $H^1$ weak solution to (\ref{maineqn}) is in $C^1(\overline{\Omega};\mathbb{R}^{d})\cap C^1(\overline{D}_{1}\cup\overline{D}_{2};\mathbb{R}^{d})$. The solution is also the unique function which has the least energy in appropriate functional spaces, characterized by
$$I_\infty[u]=\min_{v\in\mathcal{A}}I_\infty[v],\qquad\,I_\infty[v]:=\frac{1}{2}\int_{\Omega}(\mathbb{C}^0e(v), e(v))dx,$$
where
\begin{equation*}\label{def_A}
\mathcal{A}:=\left\{v\in H^1_\varphi(D;\mathbb{R}^{d}) ~\Big|~ e(v)=0\ \ \hbox{in}~~D_{1}\cup{D}_{2}\right\}.
\end{equation*}

Now we further assume that there exists a constant $R$, independent of $\varepsilon$, such that $B'_{2R}\supset\Sigma'$ and the top and bottom boundaries of the narrow region between $\partial{D}_{1}$ and $\partial{D}_{2}$ can be represented as follows. The corresponding partial boundaries of $\partial D_{1}$ and $\partial D_{2}$ are, respectively,
\begin{align}\label{h1h2'}x_{d}=\varepsilon+h_{1}(x')\quad\mbox{and}\quad x_{d}=h_{2}(x'),\quad\mbox{for}~x'\in B'_{2R},\end{align}
with
\begin{equation}\label{h1h2}
h_{1}(x')=h_{2}(x')\equiv0,\quad\mbox{for}~~x'\in\Sigma'.
\end{equation}
Moreover, in view of the assumptions of $\partial D_1$ and $\partial D_2$,  $h_{1}$ and $h_{2}$ satisfy
\begin{equation}\label{h1-h2}
\varepsilon+h_{1}(x')>h_{2}(x'),\quad\mbox{for}~~x'\in B'_{2R}\setminus\overline{\Sigma'},
\end{equation}
\begin{equation}\label{h1h20}
\nabla_{x'}h_{1}(x')=\nabla_{x'}h_{2}(x')=0,\quad\mbox{for}~~x'\in\partial\Sigma',
\end{equation}
\begin{equation}\label{h1-h21}
\nabla^{2}_{x'}(h_{1}-h_{2})(x')\geq\kappa_{0}I_{d-1},\quad\mbox{for}~~x'\in B'_{2R}\setminus\overline{\Sigma'},
\end{equation}
and
\begin{equation}\label{h1h3}
\|h_{1}\|_{C^{2,\alpha}(B'_{2R})}+\|h_{2}\|_{C^{2,\alpha}(B'_{2R})}\leq{\kappa_1},
\end{equation}
where $\kappa_{0},\ \kappa_1$ are positive constants, $I_{d-1}$ is the $(d-1)\times(d-1)$ identity matrix. Set the narrow region between $\partial{D}_{1}$ and $\partial{D}_{2}$ as
\begin{equation*}
\Omega_r:=\left\{(x',x_{d})\in \mathbb{R}^{d}~\big|~h_{2}(x')<x_{d}<\varepsilon+h_{1}(x'),~|x'|<r\right\}, \quad\mathrm{for}\ 0<r\leq\,2R.
\end{equation*}
We assume that for some $\delta_0>0$,
\begin{equation}\label{delta}
\delta_0\leq \mu, d\lambda+2\mu\leq\frac{1}{\delta_0}.
\end{equation}

Throughout the paper, unless otherwise stated, we use $C$ to denote some positive constant, whose values may vary from line to line, depending only on $d, \kappa_0, \kappa_1, R$, $\delta_0$, and an upper bound of the $C^{2,\alpha}$ norms of $\partial D_{1}, \partial D_{2}$ and $\partial D$, but not on $\varepsilon$. We call a constant having such dependence a {\it universal constant}.
Under the assumptions as above, we have the following gradient estimates in all dimensions. In order to present our idea clearly, with particular emphasis on the fact that $|\Sigma'|>0$ implies the boundedness of $|\nabla u|$, our first main result is restricted in dimension $d=2$.

\begin{theorem}\label{thm1}
Suppose that $D_{1}, D_{2}\subset{D}\subset\mathbb{R}^{2}$ be defined as above and  (\ref{h1h2'})--(\ref{h1h3}) hold. Let
$u\in{H}^{1}(D;\mathbb{R}^{2})\cap{C}^{1}(\overline{\Omega};\mathbb{R}^{2})$
be the solution to \eqref{maineqn}. Then for $x\in\Omega_{R}$, 
\begin{align}\label{thm1-1}
&|\nabla{u}(x_{1},x_{2})|\nonumber\\
&\leq C\left(\frac{\varepsilon}{|\Sigma'|+\sqrt{\varepsilon}}\frac{1}{\varepsilon+dist^2(x_{1},\Sigma')}+\frac{\varepsilon}{|\Sigma'|^{3}+\varepsilon}\frac{|x_{1}|}{\varepsilon+dist^2(x_{1},\Sigma')}\right)\|\varphi\|_{C^{2}(\partial D)},
\end{align}
and
\begin{align}\label{thm1-2}
\|\nabla{u}\|_{L^{\infty}(\Omega\setminus\Omega_{R})}\leq C\|\varphi\|_{C^{2}(\partial D)}.
\end{align}
\end{theorem}

\begin{remark}
From \eqref{thm1-1}, we can conclude that
\begin{itemize}
\item [(a)] If $|\Sigma'|>0$, where $|\Sigma'|$ denotes the area of $\Sigma'$, then for sufficiently small $\varepsilon$ (say, $<|\Sigma'|$), $|\nabla u|$ will be bounded in $\Omega_{R}$, that is,
$$|\nabla{u}(x_{1},x_{2})|\leq\,C\|\varphi\|_{C^{2}(\partial D)},\quad\mbox{for}~x\in\Omega_{R}.
$$
This implies that there is no blow-up to occur whenever $|\Sigma'|>0$. It means that this flat microstructure can not cause stress concentration. For smoothness of the linear laminates can refer to \cite{ckv,dong}.

\item [(b)] If $\Sigma'=\{0'\}$, then $|\Sigma'|=0$, \eqref{thm1-1} becomes, for $x\in\Omega_{R}$,
$$|\nabla{u}(x_{1},x_{2})|\leq C\left(\frac{\sqrt{\varepsilon}}{\varepsilon+|x_{1}|^{2}}+\frac{|x_{1}|}{\varepsilon+|x_{1}|^2}\right)\|\varphi\|_{C^{2}(\partial D)}\leq \frac{C}{\sqrt{\varepsilon}+|x_{1}|}\|\varphi\|_{C^{2}(\partial D)},
$$
which is exactly the main result in \cite{bll1} and shows that
this upper bound attain its maximum on the shortest segment $\overline{P_{1}P_{2}}$
$$|\nabla{u}(0,x_{2})|\leq \frac{C}{\sqrt{\varepsilon}}\|\varphi\|_{C^{2}(\partial D)},\quad\forall~-\varepsilon<x_{d}<\varepsilon.
$$
\end{itemize}
\end{remark}

Next, we extend Theorem \ref{thm1} to higher dimension $d\geq3$. In order to emphasize the role of $|\Sigma'|$ in the blow-up analysis and avoid the complicated calculation, we assume that $\Sigma'=B'_{R_{0}}(0')$, for some $R_{0}<R$. Actually, if we assume that $\Sigma'$ is symmetric about each $x_i$, $i=1,\cdots, d-1$, then a long remark is given after the proof of Theorem \ref{thmd}. More general cases are left to the interested readers.

\begin{theorem}\label{thmd}
Suppose that $D_{1}, D_{2}\subset{D}\subset\mathbb{R}^{d}$, $d\geq3$, be defined as above and  (\ref{h1h2'})--(\ref{h1h3}) hold. Let
$u\in{H}^{1}(D;\mathbb{R}^{d})\cap{C}^{1}(\overline{\Omega};\mathbb{R}^{d})$
be the solution to \eqref{maineqn}. Then
For $d\geq3$, under the additional assumption that $\Sigma'$ is symmetric about $x_i$ respectively, $i=1,\cdots, d-1$, then for $x\in\Omega_{R}$, if $d=3$,
\begin{align}\label{thmd-1}
&|\nabla{u}(x',x_{3})|\nonumber\\
&\leq C\left(\frac{\varepsilon}{|\Sigma'|+\varepsilon|\log\varepsilon|}\frac{1}{\varepsilon+dist^2(x',\Sigma')}+\frac{\varepsilon}{|\Sigma'|^{2}+\varepsilon}\frac{|x'|}{\varepsilon+dist^2(x',\Sigma')}\right)\|\varphi\|_{C^{2}(\partial D)},
\end{align}
if $d\geq4$,
\begin{align}\label{thmd-2}
&|\nabla{u}(x',x_{d})|\nonumber\\
&\leq C\left(\frac{\varepsilon}{|\Sigma'|+\varepsilon}\frac{1}{\varepsilon+dist^2(x',\Sigma')}+\frac{\varepsilon}{|\Sigma'|^{\frac{d+1}{d-1}}+\varepsilon}\frac{|x'|}{\varepsilon+dist^2(x',\Sigma')}\right)\|\varphi\|_{C^{2}(\partial D)};
\end{align}
and
\begin{align}\label{thmd-3}
\|\nabla{u}\|_{L^{\infty}(\Omega\setminus\Omega_{R})}\leq C\|\varphi\|_{C^{2}(\partial D)}.
\end{align}
\end{theorem}

Comparing boundedness of $|\nabla u|$ in Theorem \ref{thm1} and \ref{thmd} with the blow-up results in \cite{bll1,bll2} where we assume the relative convexity between inclusions is of order $2$, and the blow-up rate is proved to be, respectively, $\frac{1}{\sqrt{\varepsilon}}$ in dimension $d=2$, $\frac{1}{\varepsilon|\log\varepsilon|}$ in dimension $d=3$, and $\frac{1}{\varepsilon}$ in higher dimensions $d\geq4$, a nature question is raised as follows: what does exactly determine the blow-up rate?

In order to further explore the blow-up mechanism of the gradient and answer this question, we using the following example to reveal the relationship between the blow-up rate of the stress and the order of the relative convexity of the two adjacent inclusions. Under the same assumptions on $h_{1}$ and $h_{2}$ as before except for the flatness condition \eqref{h1h20} and \eqref{h1-h21}, we assume that the relative convexity between $D_{1}$ and $D_{2}$ is of order $m$, $m\geq2$, namely,
\begin{equation}\label{h1-h2m}
(h_{1}-h_{2})(x')=\kappa_{0}|x'|^m,~m\geq2,\quad\mbox{for}~~x'\in B'_{2R},
\end{equation}
and
\begin{equation}\label{h1h2m}
|\nabla h_i(x')|\leq C|x'|^{m-1}, \quad\,|\nabla^2 h_i(x')|\leq C|x'|^{m-2},~i=1,2,\quad\mbox{for}~~x'\in B'_{2R}.
\end{equation}
This example gives an essentially complete answer to the above question.

\begin{theorem}\label{thm3}
Suppose that $D_{1}, D_{2}\subset{D}\subset\mathbb{R}^{d}$ $(d\geq2)$ be defined as above, satisfying \eqref{h1h2'}--\eqref{h1-h2}, (\ref{h1-h2m})--(\ref{h1h2m}) and \eqref{h1h3}. Let
$u\in{H}^{1}(D;\mathbb{R}^{d})\cap{C}^{1}(\overline{\Omega};\mathbb{R}^{d})$
be the solution to \eqref{maineqn}. Then  we have
\begin{align*}
\|\nabla{u}\|_{L^{\infty}(\Omega)}\leq\,C
\begin{cases}
\varepsilon^{-1}\|\varphi\|_{C^{2}(\partial D)},&\mbox{if}~2\leq m<d-1,\\
\frac{1}{\varepsilon|\log\varepsilon|}\|\varphi\|_{C^{2}(\partial D)},&\mbox{if}~m=d-1,\\
\frac{1}{\varepsilon^{min\{1-\frac{1}{m},\frac{d-1}{m}\}}}\|\varphi\|_{C^{2}(\partial D)},&\mbox{if}~d-1<m<d+1,\\
\frac{1}{|\log\varepsilon|\varepsilon^{1-\frac{1}{m}}}\|\varphi\|_{C^{2}(\partial D)},&\mbox{if}~m=d+1,\\
\varepsilon^{-\frac{d}{m}}\|\varphi\|_{C^{2}(\partial D)},&\mbox{if}~m>d+1.
\end{cases}
\end{align*}
\end{theorem}

\begin{remark}
Actually, we have the following pointwise upper bounds for $x\in\Omega_{R}$,
\begin{align}\label{thm3-1}
|\nabla{u}(x)|&\leq \,C
\begin{cases}
\frac{1}{\varepsilon+|x'|^m}\|\varphi\|_{C^{2}(\partial D)},&\mbox{if}~2\leq m<d-1,\\
\left(\frac{1}{|\log\varepsilon|(\varepsilon+|x'|^m)}+\frac{|x'|}{\varepsilon+|x'|^m}+1\right)\|\varphi\|_{C^2(\partial D)},~~&\mbox{if}~m=d-1,\\\left(\frac{\varepsilon^{1-\frac{d-1}{m}}}{\varepsilon+|x'|^m}+\frac{|x'|}{\varepsilon+|x'|^m}+1\right)\|\varphi\|_{C^2(\partial D)},~~&\mbox{if}~d-1<m<d+1,\\
\left(\frac{\varepsilon^{1-\frac{d-1}{m}}}{\varepsilon+|x'|^m}+\frac{|x'|}{|\log\varepsilon|(\varepsilon+|x'|^m)}+1\right)\|\varphi\|_{C^2(\partial D)},~~&\mbox{if}~m=d+1,\\
\left(\frac{\varepsilon^{1-\frac{d-1}{m}}}{\varepsilon+|x'|^m}+\frac{\varepsilon^{1-\frac{d+1}{m}}|x'|}{\varepsilon+|x'|^m}+1\right)\|\varphi\|_{C^2(\partial D)},~~&\mbox{if}~m> d+1.
\end{cases}
\end{align}
\end{remark}

\begin{remark}
We now draw some conclusions from \eqref{thm3-1} in order:
\begin{itemize}
\item [(a)] When $m=2$, from the first three lines in \eqref{thm3-1}, we can find the maximum attain in the shortest segment between $\partial{D_{1}}$ and $\partial{D}_{2}$, $\overline{P_{1}P_{2}}$, moreover the blow-up rate are $\varepsilon^{-1/2}$ if $d=2$, $(\varepsilon|\log\varepsilon|)^{-1}$ if $d=3$, and $\varepsilon^{-1}$ if $d\geq4$, respectively, as obtained in \cite{bll1,bll2}.

\item[(b)] From the second line in \eqref{thm3-1}, $|\log\varepsilon|$ shows that $d=m+1$ is a critical dimension with respect to the convexity. That is the reason why the blow-up rate is $(\varepsilon|\log\varepsilon|)^{-1}$ if $d=3$ and $m=2$.

\item[(c)] The last line in \eqref{thm3-1} is an important improvement of theorem 5.1 in \cite{bll1}, where we only can see the blow-up rate of $|\nabla u|$ is $\varepsilon^{\frac{1}{m}-1}$, which tends to $\varepsilon^{-1}$ as $m\rightarrow+\infty$. However, from last line in \eqref{thm3-1}, we can obtain the blow-up rate is $\varepsilon^{-\frac{m}{d}}$, which tends to $1$ as $m\rightarrow+\infty$. This is exactly the reason why $|\nabla u|$ is bounded in Theorem \ref{thm1} and \ref{thmd} whenever $|\Sigma'|>0$. 

This improvement is due to our estimates of $|C_{1}^{\alpha}-C_{2}^{\alpha}|\leq\,C\varepsilon^{1-\frac{d+1}{m}}$, $\alpha=d+1,\cdots,\frac{d(d+1)}{2}$, for $m>d+1$. For more details, see Proposition \ref{mprop2}.

\item[(d)] Moreover, for any fixed dimension $d$, if $m\geq d+1$ then we will find that the maximum of the upper bounds will attain at $|x'|^{1/m}$, not at the origin $x'=0'$ any more. Thus, as $m$ increases, it is easy to see that
$|x'|^{1/m}$ will be more and more away from the origin and goes to $1$ as $m\rightarrow+\infty$.  This means that the stress concentration may diffuse as the relative convexity between $D_{1}$ and $D_{1}$ is weakened when $m$ increases. 
\end{itemize}

All in all, we finally reveal the important role of the relative convexity between $D_{1}$ and $D_{2}$ playing in the concentration mechanism of the stress. These results may be valuable to make a composite material.
\end{remark}

The rest of this paper is organized as follows. In Section \ref{sec1}, we first give some elementary properties for the Lam\'{e} system and a decomposition of the solution to \eqref{maineqn}, and then establish the gradient estimates for a general boundary value problem. For the sake of readability and presentation, with particular emphasis on the fact that the flatness between $\partial{D}_{1}$ and $\partial{D}_{2}$ leads to the boundedness of $|\nabla u|$, we restrict ourselves in dimension $d=2$ in Section \ref{sec_d=2} and give the proof of Theorem \ref{thm1}. For higher dimensions $d\geq3$, the proof of Theorem \ref{thmd} is given in Section \ref{sec_highd}. To investigate the relationship between the blow-up rate of the stress and the order of the relative convexity of $D_{1}$ and $D_{2}$, we prove Theorem \ref{thm3} in Section \ref{sec_m}. In the Appendix, we make use of the iteration technique developed in \cite{bll1,bll2} to give sketches of the proofs of Theorem \ref{thm2.1} and \ref{thm6.1} for general Dirichlet boundary problem \eqref{eq1.1} with different assumptions on $\partial{D}_{1}$ and $\partial{D}_{2}$.

\section{Preliminary}\label{sec1}

In this section, we begin by recalling some basic properties of the tensor $\mathbb{C}$ in Subsection \ref{subsec_2.1}, then decompose the solution $u$ into several $v_{i}^{\alpha}$ in Subsection \ref{sec_thm1}, which are solutions of a class of Dirichlet boundary value problems. At the same time, a family of free constant $C_{i}^{\alpha}$ are introduced. Thus, the proof of the main theorem is reduced to the estimates of $|\nabla v_{i}^{\alpha}|$ and $C_{i}^{\alpha}$. For the sake of simplicity, we consider a general boundary value problem in subsection \ref{subsec_general_bvp} to obtain the estimates of $|\nabla v_{i}^{\alpha}|$ in various cases in a unified way. Thus, the rest sections of this paper can be mainly devoted to the estimates of $C_{i}^{\alpha}$.  

\subsection{Properties of the tensor $\mathbb{C}$}\label{subsec_2.1}

We first recall some properties of the tensor $\mathbb{C}$, mainly from book \cite{osy} of Oleinik, Shamaev, and Yosifian. For the isotropic elastic material, let
$$\mathbb{C}:=(C_{ijkl})=(\lambda\delta_{ij}\delta_{kl}+\mu(\delta_{ik}\delta_{jl}+\delta_{il}\delta_{jk})),\quad \mu>0,\quad d\lambda+2\mu>0.$$
The components $C_{ijkl}$ satisfy the following symmetry condition:
\begin{align}\label{symm}
C_{ijkl}=C_{klij}=C_{klji},\quad i,j,k,l=1,2,\cdots, d.
\end{align}
We will use the following notations:
\begin{align*}
(\mathbb{C}A)_{ij}=\sum_{k,l=1}^dC_{ijkl}A_{kl},\quad\hbox{and}\quad(A,B)\equiv A:B=\sum_{i,j=1}^dA_{ij}B_{ij},
\end{align*}
for every pair of $d\times d$ matrices $A=(A_{ij})$, $B=(B_{ij})$. Clearly,
$$(\mathbb{C}A, B)=(A, \mathbb{C}B).$$
If $A$ is symmetric, then, by the symmetry condition (\ref{symm}), we have that
$$(\mathbb{C}A, A)=C_{ijkl}A_{kl}A_{ij}=\lambda A_{ii}A_{kk}+2\mu A_{kj}A_{kj}.$$
Thus $\mathbb{C}$ satisfies the following ellipticity condition: For every $d\times d$ real symmetric matrix $\eta=(\eta_{ij})$,
\begin{align}\label{ellip}
\min\{2\mu, d\lambda+2\mu\}|\eta|^2\leq(\mathbb{C}\eta, \eta)\leq\max\{2\mu, d\lambda+2\mu\}|\eta|^2,
\end{align}
where $|\eta|^2=\sum_{ij}\eta_{ij}^2.$ In particular,
\begin{align}\label{2.15}
\min\{2\mu, d\lambda+2\mu\}|A+A^T|^2\leq(\mathbb{C}(A+A^T), (A+A^T)).
\end{align}

It is well known that for any open set $O$ and $u, v\in C^2(O;\mathbb{R}^{d})$,
\begin{align}\label{eu}
\int_O(\mathbb{C}^0e(u), e(v))dx=-\int_O\left(\mathcal{L}_{\lambda, \mu}u\right)\cdot v+\int_{\partial O}\frac{\partial u}{\partial \nu_0}\Big|_{+}\cdot v.
\end{align}

\subsection{Decomposition of $u$}\label{sec_thm1}

As in \cite{bll1,bll2}, we decompose the solution $u(x)$ of \eqref{maineqn} as follows
\begin{equation}\label{decom_u}
u(x)=\sum_{\alpha=1}^{\frac{d(d+1)}{2}}C_1^{\alpha}v_{1}^{\alpha}(x)+\sum_{\alpha=1}^{\frac{d(d+1)}{2}}C_2^{\alpha}v_2^{\alpha}(x)+v_{3}(x),\qquad~x\in\,\Omega ,
\end{equation}
where $v_{i}^{\alpha}\in{C}^{2}(\Omega;R^d)$, $i=1,2$, $\alpha=1,2,\cdots,\frac{d(d+1)}{2}$, and $v_{3}$, respectively, satisfying
\begin{equation}\label{equ_v1}
\begin{cases}
\mathcal{L}_{\lambda,\mu}v_{i}^{\alpha}=0,&\mathrm{in}~\Omega,\\
v_{i}^{\alpha}=\psi_{\alpha},&\mathrm{on}~\partial{D}_{i},~i=1,2,\\
v_{i}^{\alpha}=0,&\mathrm{on}~\partial{D_{j}}\cup\partial{D},~j\neq i,
\end{cases}
\end{equation}
and
\begin{equation}\label{equ_v3}
\begin{cases}
\mathcal{L}_{\lambda,\mu}v_{3}=0,&\mathrm{in}~\Omega,\\
v_{3}=0,&\mathrm{on}~\partial{D}_{1}\cup\partial{D_{2}},\\
v_{3}=\varphi,&\mathrm{on}~\partial{D}.
\end{cases}
\end{equation}
Then by \eqref{decom_u}, we have
\begin{align}\label{decom_nabla_u}
\nabla{u}(x)=&\sum_{\alpha=1}^{\frac{d(d+1)}{2}}C_{1}^\alpha\nabla{v}_{1}^\alpha(x)+\sum_{\alpha=1}^{\frac{d(d+1)}{2}}C_{2}^\alpha\nabla{v}_{2}^\alpha(x)+\nabla{v}_{3}(x)\nonumber\\
=&\sum_{\alpha=1}^{\frac{d(d+1)}{2}}(C_{1}^\alpha-C_{2}^\alpha)\nabla{v}_{1}^\alpha(x)+\sum_{\alpha=1}^{\frac{d(d+1)}{2}}C_{2}^\alpha\nabla({v}_{1}^\alpha+{v}_{2}^\alpha)(x)+\nabla{v}_{3}(x),\qquad~x\in\,\Omega.
\end{align}
Thus, the proofs of our theorems are reduced to the establishment of the following two kinds of estimates:

(i) Estimates of $|\nabla v_{i}^{\alpha}|$, $i=1,2$, $\alpha=1,\cdots,\frac{d(d+1)}{2}$, and $|\nabla v_{3}|$;

(ii) Estimates of $|C_i^{\alpha}|$, $i=1,2$, $\alpha=1,\cdots,\frac{d(d+1)}{2}$, and $|C_{1}^{\alpha}-C_{2}^{\alpha}|$, $\alpha=1,\cdots,\frac{d(d+1)}{2}$.

We notice that decomposition \eqref{decom_nabla_u} is a little different with that in \cite{bll1,bll2}. Here we also need to estimate the differences of $|C_{1}^{\alpha}-C_{2}^{\alpha}|$, $\alpha=d+1,\cdots,\frac{d(d+1)}{2}$, which is new and important part for our main results. These two kinds of estimates are connected. Estimates (ii), especially that of $|C_{1}^{\alpha}-C_{2}^{\alpha}|$, heavily depend on how good estimates we can obtain for Estimates (i).

\subsection{A general boundary value problem}\label{subsec_general_bvp}

First, by theorem 1.1 in \cite{llby}, we know that $|\nabla v_{3}|$ is bounded. 
Because $v_{i}^{\alpha}$ is the unique solution to the Dirichlet boundary problem on $\Omega$, in order to obtain the estimates of $|\nabla v_{i}^{\alpha}|$ under a unified framework, we consider the following general Dirichlet boundary value problem:
\begin{align}\label{eq1.1}
\begin{cases}
\mathcal{L}_{\lambda,\mu}v:=\nabla\cdot(\mathbb{C}^0e(v))=0,\quad&
\hbox{in}\  \Omega,  \\
v=\psi(x),\quad &\hbox{on}\ \partial{D}_{1},\\
v=0, \quad&\hbox{on} \ \partial{D}_{2}\cup\partial{D},
\end{cases}
\end{align}
where $ \psi(x)=(\psi^1(x), \psi^2(x),\cdots,\psi^{d}(x))^{T}\in C^2(\partial{D}_{1}; \mathbb{R}^{d})$ is a given vector-valued function. Thus, if we have obtained estimate of $|\nabla v|$, then taking $\psi=\psi_{\alpha}$, $\alpha=1,2,\cdots,\frac{d(d+1)}{2}$, respectively, we can obtain that of $|\nabla v_{1}^{\alpha}|$ immediately. By the same way, we can also have that of $|\nabla v_{2}^{\alpha}|$.

Making use of the idea in \cite{bll1,bll2}, we decompose $v$ as follows:
\begin{align*}
v=v_{1}+v_2+\cdots+v_{d},
\end{align*}
where $v_l=(v_l^1, v_l^2,\cdots,v_{l}^{d})^{T}$, $l=1,2,\cdots,d$, with  $v_l^j=0$ for $j\neq l$, and $v_{l}$ satisfy the following boundary value problem, respectively,
\begin{align}\label{eq_v2.1}
\begin{cases}
  \mathcal{L}_{\lambda,\mu}v_{l}:=\nabla\cdot(\mathbb{C}^0e(v_{l}))=0,\quad&
\hbox{in}\  \Omega,  \\
v_{l}=( 0,\cdots,0,\psi^{l}, 0,\cdots,0)^{T},\ &\hbox{on}\ \partial{D}_{1},\\
v_{l}=0,&\hbox{on} \ \partial{D}_{2}\cup\partial{D}.
\end{cases}
\end{align}
Thus,
\begin{equation}\label{equ_nablav}
\nabla{v}=\sum_{l=1}^{d}\nabla{v}_{l}.
\end{equation}
Then it suffices to estimate $|\nabla v_l|$ one by one.

To this end, we now introduce a scalar auxiliary function $\bar{v}\in C^2(\mathbb{R}^d)$ such that $\bar{v}=1$ on $\partial{D}_{1}$, $\bar{v}=0$ on $\partial{D}$ and
\begin{equation}\label{vvd}
\bar{v}(x)=\frac{x_{d}-h_2(x')}{\varepsilon+h_1(x')-h_2(x')},\quad\hbox{in}\ \Omega_{2R},
\end{equation}
and
\begin{equation}\label{nabla_vbar_outside}
\|\bar{v}\|_{C^{2}(\Omega\setminus\Omega_{R})}\leq\,C.
\end{equation}
We now extend $\psi\in{C}^{2}(\partial{D}_{1};\mathbb{R}^{d})$ to $\psi\in{C}^{2}(\overline{\Omega};\mathbb{R}^{d})$ such that $\|\psi^{l}\|_{C^{2}(\overline{\Omega\setminus\Omega_{R}})}\leq\,C\|\psi^{l}\|_{C^2(\partial D_1)}$, for $l=1,2,\cdots,d$. We can find a cutoff function $\rho\in{C}^{2}(\overline{\Omega})$ such that
\begin{equation*}
0\leq\rho\leq1,~~
\rho=1~\mbox{on}~\overline{\Omega}_{2R},\qquad~\rho=0~\mbox{on}~\overline{\Omega}\setminus\Omega_{3R}, \qquad~\mbox{and}~|\nabla\rho|\leq\,C\quad~\mbox{on}~\overline{\Omega}.
\end{equation*}
Define
\begin{align}\label{equ_tildeu}
\tilde{v}_{l}(x)=( 0,\cdots,0,\left[\rho(x)\psi^{l}(x',\varepsilon+h_{1}(x'))+(1-\rho(x))\psi^{l}(x',\varepsilon+h_{1}(x'))\right]\bar{v}(x), 0,\cdots,0)^{T},\quad\,x\in\Omega.
\end{align}
Thus,
\begin{align}\label{equ_tildeu_in}
\tilde{v}_{l}(x)=( 0,\cdots,0,\psi^{l}(x',\varepsilon+h_{1}(x'))\bar{v}(x), 0,\cdots,0)^{T},\quad\,x\in\Omega_{2R},
\end{align}
and in view of \eqref{nabla_vbar_outside},
\begin{equation}\label{nabla_vtilde_outside}
\|\tilde{v}_{l}\|_{C^{2}(\Omega\setminus\Omega_{R})}\leq\,C\|\psi^{l}\|_{C^2(\partial D_1)}.
\end{equation}

Denote
\begin{equation*}
\delta(x'):=\varepsilon+h_{1}(x')-h_{2}(x'),\qquad\forall~(x',x_{d})\in\Omega_{2R},
\end{equation*}
and
$$d(x'):=d_{\Sigma'}(x')=dist(x',\Sigma').$$
By a direct calculation, we obtain that for $k=1,\cdots, d-1$, \begin{equation}
|\partial_{x_{k}}\bar{v}(x)|\leq\frac{Cd(x')}{\varepsilon+d^2(x')},\qquad \qquad~~\partial_{x_{d}}\bar{v}(x)=\frac{1}{\delta(x')},\quad\,x\in\Omega_{2R}.
\label{e2.4}
\end{equation}
Due to (\ref{e2.4}), for $l=1,2,\cdots,d$, and $k=1,2,\cdots,d-1$, we have
\begin{align}\label{eq1.7}
|\partial_{x_{k}}\tilde{v}_{l}(x)|\leq\frac{Cd(x')|\psi^{l}(x',\varepsilon+h_{1}(x'))|}{\varepsilon+d^2(x')}+C\|\nabla\psi^{l}\|_{L^{\infty}},\quad\,x\in\Omega_{2R},
\end{align}
and
\begin{align}\label{eq1.7a}
|\partial_{x_{d}}\tilde{v}_{l}(x)|=\frac{|\psi^{l}(x',\varepsilon+h_{1}(x'))|}{\delta(x')},\quad\,x\in\Omega_{2R}.
\end{align}
Then, we have locally piontwise gradient estimates as follows:

\begin{theorem}\label{thm2.1}
Assume that hypotheses \eqref{h1h2}--\eqref{h1h3} are satisfied, and let $v\in H^{1}(\Omega; \mathbb{R}^{d})$ be a weak solution of problem \eqref{eq1.1}. Then for sufficiently small $0<\varepsilon<1/2$,
\begin{align}\label{equa2.9}
|\nabla (v_l-\tilde{v}_{l})(x)|\leq \frac{C|\psi^{l}(x',\varepsilon+h_{1}(x'))|}{\sqrt{\delta(x')}}+C\|\psi^{l}\|_{C^2(\partial{D}_{1})},\quad\forall~x\in\Omega_{R}.
\end{align}
Consequently, by \eqref{eq1.7}, \eqref{eq1.7a} and \eqref{equa2.9}, we have, for $x\in\Omega_{R}$,
\begin{align}\label{estimate_vl}
\frac{|\psi^{l}(x',\varepsilon+h_{1}(x'))|}{C(\varepsilon+d^2(x'))}&\leq|\nabla v_l(x',x_{d})|\leq \frac{C|\psi^{l}(x',\varepsilon+h_{1}(x'))|}{\varepsilon+d^2(x')}+C\|\psi^{l}\|_{C^2(\partial{D}_{1})},
\end{align}
and finally,
\begin{align*}
|\nabla v(x',x_{d})|
\leq&\, \frac{C}{\varepsilon+d^2(x')}\Big|\psi(x',\varepsilon+h_{1}(x'))\Big|+C\|\psi\|_{C^{2}(\partial{D}_{1})},\quad\forall\,x\in \Omega_{R},
\end{align*}
and
$$\|\nabla v\|_{L^{\infty}(\Omega\setminus\Omega_{R})}
\leq\,C\|\psi\|_{C^{2}(\partial{D}_{1})}.
$$
\end{theorem}

Actually, for more general Dirichlet boundary value problems:
\begin{align*}
\begin{cases}
\mathcal{L}_{\lambda,\mu}v:=\nabla\cdot(\mathbb{C}^0e(v))=0,\quad&
\hbox{in}\  \Omega,  \\
v=\psi(x),\quad &\hbox{on}\ \partial{D}_{1},\\
v=\widetilde\psi(x),\quad &\hbox{on}\ \partial{D}_{2},\\
v=0, \quad&\hbox{on} \ \partial{D},
\end{cases}
\end{align*}
where $ \psi(x)\in C^2(\partial{D}_{1}; \mathbb{R}^{d})$ and $\widetilde\psi(x)\in C^2(\partial{D}_{2}; \mathbb{R}^{d})$ are given vector-valued functions, with a slight modification necessary, we have

\begin{corollary}\label{cor2.2}
Under the same assumptions as in Theorem \ref{thm2.1}, we have, for $x\in \Omega_{R}$,
\begin{align*}
|\nabla v(x',x_{d})|
\leq&\, \frac{C}{\varepsilon+d^2(x')}\Big|\psi(x',\varepsilon+h_{1}(x'))-\widetilde\psi(x',\varepsilon+h_{2}(x'))\Big|\nonumber\\
&+C\|\psi\|_{C^{2}(\partial{D}_{1})}+C\|\widetilde\psi\|_{C^{2}(\partial{D}_{2})},
\end{align*}
and
$$\|\nabla v\|_{L^{\infty}(\Omega\setminus\Omega_{R})}
\leq\,C\left(\|\psi\|_{C^{2}(\partial{D}_{1})}+\|\widetilde\psi\|_{C^{2}(\partial{D}_{2})}\right).
$$
\end{corollary}

Because the reason resulting in the boundedness of $|\nabla u|$ of problem \eqref{maineqn} is mainly from the estimates of $|C_{1}^{\alpha}-C_{2}^{\alpha}|$, $\alpha=1,2,\cdots,\frac{d(d+1)}{2}$, we move the proof of Theorem \ref{thm2.1} to the Appendix.

\section{Proof of Theorem \ref{thm1}}\label{sec_d=2}

In order to present our idea clearer, with particular emphasis on the fact that the flatness between $\partial{D}_{1}$ and $\partial{D}_{2}$ leads to the boundedness of $|\nabla u|$, we restrict ourselves in this section in dimension $d=2$. The proof of Theorem \ref{thm1} relies heavily on the estimates $|C_{1}^{\alpha}-C_{2}^{\alpha}|$, $\alpha=1,2,3$, in Subsection \ref{subsec_c1c2}, besides that of $|\nabla v_{i}^{\alpha}|$, which is the key to determine whether $|\nabla u|$ blows up or not. So the main difference is from the estimates in  Lemma \ref{lem_a11}, the proof is given in Subsection \ref{subsec34}. First, we have

\subsection{Estimates of $|\nabla v_{i}^{\alpha}|$}

For problem (\ref{equ_v1}) in dimension $d=2$, taking
$$\psi=\psi_{\alpha},\quad\mbox{and}\quad \tilde{v}_1^\alpha:=\bar{v}\psi_{\alpha},\qquad\alpha=1,2,3.$$
and applying Theorem \ref{thm2.1}, we have

\begin{corollary}\label{cor312}
Assume as in Theorem \ref{thm1}, let $v_{i}^\alpha$, $i=1,2$, $\alpha=1,2,3$, be the
weak solutions of \eqref{equ_v1}, respectively. Then for sufficiently small $0<\varepsilon<1/2$, we have
\begin{equation}\label{nabla_w_i02}
|\nabla(v_{i}^\alpha-\tilde{v}_i^\alpha)(x_{1},x_{2})|\leq\,\frac{C}{\sqrt{\varepsilon+d^{2}(x_{1})}},\quad i,\alpha=1,2,~x\in\Omega_{R};
\end{equation}
\begin{equation}\label{nabla_w_i00}
|\nabla(v_{i}^3-\tilde{v}_i^3)(x_{1},x_{2})|\leq\,\frac{C|x_{1}|}{\sqrt{\varepsilon+d^{2}(x_{1})}},\quad i=1,2,~x\in\Omega_{R}.
\end{equation}
Consequently, by using \eqref{eq1.7} and \eqref{eq1.7a},
\begin{equation}\label{v1-bounded12}
\frac{1}{C\left(\varepsilon+d^{2}(x_{1})\right)}\leq|\nabla v_{i}^\alpha(x)|\leq\,\frac{C}{\varepsilon+d^{2}(x_{1})},\quad\,~i,\alpha=1,2,~x\in\Omega_{R};
\end{equation}
\begin{equation}\label{v1---bounded1}
|\nabla v_i^{3}(x)|\leq\frac{C(\varepsilon+|x_{1}|)}{\varepsilon+d^2(x_{1})},\   \ \qquad\,~i=1,2,~x\in \Omega_{R};  
\end{equation}
and
\begin{equation}\label{v1--bounded12}
\quad\ \ \ |\nabla v_{i}^\alpha(x)|\leq\,C,\qquad\,~i=1,2;~\alpha=1,2,3,~x\in\Omega\setminus\Omega_{R}.
\end{equation}
Especially, by \eqref{eq1.7} and \eqref{nabla_w_i02},
\begin{equation}\label{v1-x'2}
\quad\quad\ |\partial_{x_{1}}v_{i}^\alpha(x)|\leq\,\frac{C}{\sqrt{\varepsilon+d^{2}(x_{1})}},\qquad\,~i,\alpha=1,2,~x\in\Omega_{R}.
\end{equation}
\end{corollary}

Since the estimates for $|\nabla v_{i}^{\alpha}|$, even for $|\nabla u|$, in $\Omega\setminus\Omega_{R}$ can be obtained by the standard interior and boundary estimates for elliptic systems, see \cite{AD1,AD2}, we will concentrate the estimate in the narrow region $\Omega_{R}$ in the following.

\begin{lemma}\label{lemma23}
\begin{equation}\label{v1+v2_bounded1}
|\nabla(v_{1}^{\alpha}+v_{2}^\alpha)(x)|\leq\,C\|\varphi\|_{C^{2}(\partial D)},\quad\,\alpha=1,2,3,\quad~x\in\Omega;
\end{equation}
and
\begin{equation}\label{nabla_v3}
|\nabla{v}_{3}(x)|\leq C\|\varphi\|_{C^{2}(\partial D)},\quad\,x\in\Omega;
\end{equation}
\end{lemma}

\begin{proof}
This is an immediate consequence of Corollary \ref{cor2.2}.
\end{proof}

In the following we will concentrate on the estimates of $|C_{1}^{\alpha}-C_{2}^{\alpha}|$.

\subsection{Estimates of $|C_{1}^{\alpha}-C_{2}^{\alpha}|$, $\alpha=1,2,3$}\label{subsec_c1c2}

Denote
\begin{align*}
a_{ij}^{\alpha\beta}:=-\int_{\partial{D}_{j}}\frac{\partial v_{i}^{\alpha}}{\partial \nu_0}\large\Big|_{+}\cdot\psi_{\beta},\quad b_j^{\beta}:=\int_{\partial{D}_{j}}\frac{\partial v_{3}}{\partial \nu_0}\large\Big|_{+}\cdot\psi_{\beta},\quad \alpha, \beta=1,2,3.
\end{align*}
Multiplying the first line of (\ref{equ_v1}) and (\ref{equ_v3}),  respectively, by $v_j^\beta$, and applying integration by parts over $\Omega$ leads to
$$a_{ij}^{\alpha\beta}=\int_{\Omega}(\mathbb{C}^0e(v_{i}^{\alpha}), e(v_j^\beta))dx,\quad b_j^{\beta}=-\int_{\Omega}(\mathbb{C}^0e(v_{3}), e(v_j^\beta))dx.$$
It follows from the fourth line of \eqref{maineqn}, we have the following linear system of $C_{i}^{\alpha}$,
\begin{align}\label{system}
\begin{cases}
\sum_{\alpha=1}^{3}C_{1}^\alpha a_{11}^{\alpha\beta}+\sum_{\alpha=1}^{3}C_{2}^\alpha a_{21}^{\alpha\beta}=b_1^\beta,\\\\
\sum_{\alpha=1}^{3}C_{1}^\alpha a_{12}^{\alpha\beta}+\sum_{\alpha=1}^{3}C_{2}^\alpha a_{22}^{\alpha\beta}=b_2^\beta,
\end{cases}~~~\beta=1,2,3.
\end{align}

For the sake of simplicity, let $a_{ij}$ denote the $3\times3$ matrix $(a_{ij}^{\alpha\beta})$. To estimate $|C_{1}^{\alpha}-C_{2}^{\alpha}|$, $\alpha=1,2,3$,  we only use the first three equations in \eqref{system}:
\begin{equation}\label{c1-c2}
a_{11}C_1+a_{21}C_2=b_1.
\end{equation}
where
\begin{equation*}
C_1=(C_1^1,C_1^2,C_1^3)^T, \quad C_2=(C_2^1,C_2^2,C_2^3)^T,\quad b_1=(b_1^1,b_1^2,b_1^3)^T.
\end{equation*}
\eqref{c1-c2} can be rewritten as
\begin{equation}\label{p}
a_{11}(C_1-C_2)=p,
\end{equation}
where
$$p:=b_1-(a_{11}+a_{21})C_2.$$

In order to solve $C_{1}-C_{2}$, we need to estimate each element of $a_{11}$ and $p$
to obtain a good control on $a_{11}^{-1}$ and $p$ by using the estimates of $|\nabla v_{i}^{\alpha}|$ and $|\nabla v_{3}|$ obtained in Corollary \ref{cor312} and Lemma \ref{lemma23}. First,

\begin{lemma}\label{lemma p}
\begin{align*}
|a_{11}^{\alpha\beta}+a_{21}^{\alpha\beta}|&\leq C\|\varphi\|_{C^{2}(\partial D)}, \quad \alpha, \beta=1,2,3;
\end{align*}
and
\begin{align*}
|b_1^\beta|&\leq C\|\varphi\|_{C^{2}(\partial D)}, \quad \beta=1,2,3.
\end{align*}
Consequently,
\begin{equation}\label{p-c}
|p|\leq C\|\varphi\|_{C^{2}(\partial D)}.
\end{equation}
\end{lemma}

\begin{proof}
By definition and \eqref{nabla_v3}, we have
\begin{align*}
|b_1^{\beta}|&=\left|\int_{\partial{D}_{1}}\frac{\partial v_{3}}{\partial \nu_0}\large\Big|_{+}\cdot\psi_{\beta}\right|\leq C\|\varphi\|_{C^2(\partial D)}.
\end{align*}
Similarly, it follows from \eqref{v1+v2_bounded1} that
\begin{align*}
|a_{11}^{\alpha\beta}+a_{21}^{\alpha\beta}|&=\left|-\int_{\partial{D}_{1}}\frac{\partial v_{1}^{\alpha}}{\partial \nu_0}\large\Big|_{+}\cdot\psi_{\beta}-\int_{\partial{D}_{1}}\frac{\partial v_{2}^{\alpha}}{\partial \nu_0}\large\Big|_{+}\cdot\psi_{\beta}\right|\\
&=\left|\int_{\partial{D}_{1}}\frac{\partial (v_{1}^{\alpha}+ v_{2}^{\alpha})}{\partial \nu_0}\large\Big|_{+}\cdot\psi_{\beta}\right|\\
&\leq C\|\varphi\|_{C^{2}(\partial D)}.
\end{align*}

On the other hand, by trace theorem and an adaption of the proof of Lemma 4.1 in \cite{bll1}, with minor modification, we can obtain that
\begin{equation*}
|C_i^\alpha|\leq C, \quad i=1,2,\quad \alpha=1,2,3;
\end{equation*}
Combining these estimates yields \eqref{p-c}.
\end{proof}

\begin{lemma}\label{lem_a11}
For $d=2$, we have
\begin{align}
\frac{1}{C}\left(\frac{|\Sigma'|}{\varepsilon}+\frac{1}{\sqrt{\varepsilon}}\right)\leq a_{11}^{\alpha\alpha}
&\leq C\left(\frac{|\Sigma'|}{\varepsilon}+\frac{1}{\sqrt{\varepsilon}}\right),\quad \alpha=1, 2;\label{lem34-1}\\
\frac{1}{C}\left(\frac{|\Sigma'|^{3}}{\varepsilon}+1\right)\leq|a_{11}^{33}|&\leq C\left(\frac{|\Sigma'|^{3}}{\varepsilon}+1\right)\label{lem34-2};
\end{align}
\begin{align}\label{lem34-3}
|a_{11}^{12}|=|a_{11}^{21}|&\leq C\left(\frac{|\Sigma'|}{\sqrt{\varepsilon}}+|\log\varepsilon|\right);
\end{align}
\begin{align}\label{lem34-4}
|a_{11}^{13}|=|a_{11}^{31}|&\leq C\left(\frac{|\Sigma'|^{2}}{\sqrt{\varepsilon}}+1\right);
\end{align}
\begin{align}\label{lem34-5}
|a_{11}^{23}|=|a_{11}^{32}|&\leq C\left(\frac{|\Sigma'|^{2}}{\sqrt{\varepsilon}}+1\right).
\end{align}
Consequently,
\begin{equation*}
\frac{1}{C}\left(\frac{|\Sigma'|}{\varepsilon}+\frac{1}{\sqrt{\varepsilon}}\right)^{2}\left(\frac{|\Sigma'|^{3}}{\varepsilon}+1\right)\leq\det a_{11}\leq\,C\left(\frac{|\Sigma'|}{\varepsilon}+\frac{1}{\sqrt{\varepsilon}}\right)^{2}\left(\frac{|\Sigma'|^{3}}{\varepsilon}+1\right).
\end{equation*}
\end{lemma}

Here we remark that if $\Sigma'=\{0'\}$, then these estimates has been obtained in \cite{bll1}, except \eqref{lem34-3} becoming better a little bit. In order to stress the important role of $|\Sigma'|$ in the blow-up analysis of $|\nabla u|$, we first use Lemma \ref{lem_a11} to solve $|C_{1}^{\alpha}-C_{2}^{\alpha}|$, $\alpha=1,2,3$, and therefore obtain the estimate of $|\nabla u|$. The proof of Lemma \ref{lem_a11} is left to Subsection \ref{subsec34} below. The following proposition is important, especially the estimate \eqref{c-c3} for $|C_{1}^{3}-C_{2}^{3}|$.

\begin{prop}\label{prop1}
Let $C_i^\alpha$ be defined in \eqref{decom_u}, $i=1,2$, $\alpha=1,2,3$. Then
\begin{align}\label{c-c}
|C_{1}^{\alpha}-C_{2}^{\alpha}|\leq \frac{C\varepsilon}{|\Sigma'|+\sqrt{\varepsilon}}\|\varphi\|_{C^{2}(\partial D)},\quad\alpha=1,2,
\end{align}
and
\begin{align}\label{c-c3}
|C_{1}^{3}-C_{2}^{3}|\leq \frac{C\varepsilon}{|\Sigma'|^3+\varepsilon}\|\varphi\|_{C^{2}(\partial D)}.
\end{align}
\end{prop}

\begin{proof}
From \eqref{p}, we have
\begin{equation}\label{X_1}
\begin{pmatrix}
a_{11}^{11}&a_{11}^{12}&a_{11}^{13}\\ \\
a_{11}^{21}&a_{11}^{22}&a_{11}^{23}\\\\
a_{11}^{31}&a_{11}^{32}&a_{11}^{33}
\end{pmatrix}
\begin{pmatrix}
C_1^1-C_2^1\\\\
 C_1^2-C_2^2\\\\
C_1^3-C_2^3\end{pmatrix}
=\begin{pmatrix}
p_1\\\\
p_2\\\\
p_3
\end{pmatrix}.
\end{equation}
By Cramer's rule, we see  from \eqref{X_1},
\begin{gather*}C_1^1-C_2^1=\frac{1}{\det a_{11}}
\begin{vmatrix} p_1&a_{11}^{12}&a_{11}^{13}\\  p_2&a_{11}^{22}&a_{11}^{23}\\ p_3&a_{11}^{32}&a_{11}^{33}
\end{vmatrix},\quad
C_1^2-C_2^2=\frac{1}{\det a_{11}}
\begin{vmatrix} a_{11}^{11}&p_1&a_{11}^{13}\\  a_{11}^{21}&p_2&a_{11}^{23}\\ a_{11}^{31}&p_3&a_{11}^{33}
\end{vmatrix},
\end{gather*}
and
$$C_1^3-C_2^3=\frac{1}{\det a_{11}}
\begin{vmatrix} a_{11}^{11}&a_{11}^{12}&p_{1}\\  a_{21}&a_{22}&p_{2}\\ a_{11}^{31}&a_{11}^{32}&p_{3}
\end{vmatrix}.$$
That is,
\begin{gather*}\label{c1}
C_1^1-C_2^1=\frac{1}{\det a_{11}}
\left(p_1\begin{vmatrix} a_{11}^{22}&a_{11}^{23}\\  a_{11}^{32}&a_{11}^{33}\\  \end{vmatrix}-p_2\begin{vmatrix} a_{11}^{12}&a_{11}^{13}\\  a_{11}^{32}&a_{33}\\  \end{vmatrix}+p_3\begin{vmatrix} a_{11}^{12}&a_{11}^{13}\\ a_{11}^{22}&a_{11}^{23}\\  \end{vmatrix}\right),
\end{gather*}

\begin{gather*}\label{c2}
C_1^2-C_2^2=\frac{1}{\det a_{11}}
\left(-p_1\begin{vmatrix} a_{11}^{21}&a_{11}^{23}\\  a_{11}^{31}&a_{11}^{33}\\  \end{vmatrix}+p_2\begin{vmatrix} a_{11}^{11}&a_{11}^{13}\\  a_{11}^{31}&a_{11}^{33}\\  \end{vmatrix}-p_3\begin{vmatrix} a_{11}^{11}&a_{11}^{13}\\ a_{11}^{21}&a_{11}^{23}\\  
\end{vmatrix}\right),
\end{gather*}

\begin{gather*}\label{c3}
C_1^3-C_2^3=\frac{1}{\det a_{11}}
\left(p_1\begin{vmatrix} a_{11}^{21}&a_{11}^{22}\\  a_{11}^{31}&a_{11}^{32}\\  \end{vmatrix}-p_2\begin{vmatrix} a_{11}^{11}&a_{11}^{12}\\  a_{11}^{31}&a_{11}^{32}\\  \end{vmatrix}+p_3\begin{vmatrix} a_{11}^{11}&a_{11}^{12}\\ a_{11}^{21}&a_{11}^{22}\\  \end{vmatrix}\right).
\end{gather*}
Put the estimates obtained in Lemma \ref{lem_a11} and \eqref{p-c} into these formulae, we obtain \eqref{c-c} and \eqref{c-c3}.
\end{proof}

\begin{remark}
(i) If $\Sigma'=\{0'\}$, then $|\Sigma'|=0$, thus we have
\begin{align*}
|C_{1}^{\alpha}-C_{2}^{\alpha}|\leq C\sqrt{\varepsilon}\|\varphi\|_{C^{2}(\partial D)}.
\end{align*}
This is exactly proposition 4.2 in \cite{bll1}.

(ii) Here the estimate of $|C_{1}^{3}-C_{2}^{3}|$ is new and necessary.
\end{remark}

With the lemma and proposition established we turn to the proof of Theorem \ref{thm1}.

\subsection{Completion of the proof of Theorem \ref{thm1}}

\begin{proof}[Proof of Theorem \ref{thm1}]
Using  Corollary \ref{cor312}, Lemma \ref{lemma23}, and Proposition \ref{prop1}, we obtain for $x\in\Omega_{R}$,
\begin{align*}
|\nabla{u}(x)|&
\leq\sum_{\alpha=1}^{3}\left|(C_{1}^\alpha-C_{2}^\alpha)\nabla{v}_{1}^\alpha(x)\right|+\sum_{\alpha=1}^{3}\left|C_{2}^\alpha\nabla({v}_{1}^\alpha+{v}_{2}^\alpha)(x)\right|+\left|\nabla{v}_{3}(x)\right|\\
&
\leq\sum_{\alpha=1}^{2}\left|(C_{1}^\alpha-C_{2}^\alpha)\nabla{v}_{1}^\alpha(x)\right|+\left|C_{1}^3-C_{2}^3\right|\left|\nabla{v}_{1}^3(x)\right|+C\|\varphi\|_{C^2(\partial D)}\\
&\leq C\left(\frac{\varepsilon}{|\Sigma'|+\sqrt{\varepsilon}}\frac{1}{\varepsilon+d^{2}(x_{1})}+\frac{\varepsilon}{|\Sigma'|^3+\varepsilon}\frac{|x_{1}|}{\varepsilon+d^{2}(x_{1})}\right)\|\varphi\|_{C^2(\partial D)}+C\|\varphi\|_{C^2(\partial D)}.
\end{align*}
By the interior and boundary estimates for elliptic systems, see \cite{AD1,AD2}, we have \eqref{thm1-2}. 
This completes the proof of Theorem \ref{thm1}.
\end{proof}

\subsection{Proof of Lemma \ref{lem_a11}}\label{subsec34}

We now give the proof of Lemma \ref{lem_a11} and pay attention to the role of $|\Sigma'|>0$.

\begin{proof}[Proof of Lemma \ref{lem_a11}]

{\bf Step 1.} Estimates of $a_{11}^{\alpha\alpha}$, $\alpha=1,2$.

We only estimate the case $\alpha=1$ for instance, since the case $\alpha=2$ is the same. In view of  (\ref{ellip}), for $\alpha=1$,
\begin{equation*}
a_{11}^{11}=\int_{\Omega}(\mathbb{C}^0e(v_{1}^{1}), e(v_{1}^{1}))dx\leq C\int_{\Omega}|\nabla v_{1}^{1}|^2dx.
\end{equation*}
We decompose the integral on the right hand side into three parts,
\begin{align}\label{a11'}
\int_{\Omega}|\nabla v_{1}^{1}|^2=\int_{\Sigma}|\nabla v_{1}^{1}|^{2}+\int_{\Omega_{R}\setminus\Sigma}|\nabla v_{1}^{1}|^{2}+\int_{\Omega\setminus \Omega_{R}}|\nabla v_{1}^{1}|^{2}.
\end{align}
For the first integral on the right hand side, by using (\ref{v1-bounded12}), we have
\begin{align}\label{first_term}
\frac{|\Sigma'|}{C\varepsilon}\leq\int_{\Sigma'}\int_{0}^{\varepsilon}\frac{1}{C\varepsilon^2}dx_2dx_{1}\leq\int_{\Sigma}|\nabla v_{1}^{1}|^{2}\leq\int_{\Sigma'}\int_{0}^{\varepsilon}\frac{C}{\varepsilon^2}dx_2dx_{1}\leq\frac{C|\Sigma'|}{\varepsilon}.
\end{align}
The last integral of \eqref{a11'} is finite,
\begin{align}\label{3-term}
\int_{\Omega\setminus \Omega_{R}}|\nabla v_{1}^{1}|^{2}\leq C,
\end{align}
since clearly from (\ref{v1--bounded12}). For the middle term of \eqref{a11'}, recalling that $\Sigma'=(-R_{0},R_{0})$ in dimension two, and $d(x_{1})=|x_{1}|-R_{0}$ on $\Omega_{R}\setminus\Sigma$, we have
\begin{align*}
\int_{\Omega_{R}\setminus\Sigma}|\nabla v_{1}^{1}|^{2}&\leq C\int_{\Omega_{R}\setminus\Sigma}\frac{dx}{(\varepsilon+d^2(x_{1}))^2}\\
&\leq C\int_{B'_{R}\setminus\Sigma'}\frac{dx_{1}}{\varepsilon+d^2(x_{1})}\\
&= C\int_{R_0}^{R}\frac{dr}{\varepsilon+(r-R_0)^2}\\
&=C\int_{0}^{R-R_0}\frac{dr}{\varepsilon+r^2}\\
&\leq \frac{C}{\sqrt{\varepsilon}},
\end{align*}
which, together with \eqref{first_term} and \eqref{3-term}, implies that
\begin{align*}
a_{11}^{11}\leq C\left(\frac{|\Sigma'|}{\varepsilon}+\frac{1}{\sqrt{\varepsilon}}\right).
\end{align*}

On the other hand, to obtain the lower bound, we argue as follows. First,
\begin{align*}
a_{11}^{11}&=\int_{\Omega}(\mathbb{C}^0e(v_{1}^{1}), e(v_{1}^{1}))dx\geq \frac{1}{C}\int_{\Omega}|e(v_{1}^{1})|^2dx\geq \frac{1}{C}\int_{\Omega_R}|\partial_{x_2}(v_{1}^{1})^{1}|^2dx.
\end{align*}
Notice that $(v_{1}^{1})^{1}|_{\partial{D}_{1}}=\bar{v}|_{\partial{D}_{1}}=1, (v_{1}^{1})^{1}|_{\partial D_2}
=\bar{v}|_{\partial D_2}=0$. By the definition of $\bar{v}$, the linearity of $\bar{v}(x_{1},x_2)$ in $x_2$ for any fixed $x_{1}$ clearly implies that $\bar{v}(x_{1}, \cdot)$ is harmonic. Hence its energy is minimal, that is
\begin{align*}
\int_{h_2(x_{1})}^{h_1(x_{1})+{\varepsilon}}|\partial_{x_2}(v_{1}^{1})^{1}|^2dx_2\geq
\int_{h_2(x_{1})}^{h_1(x_{1})+{\varepsilon}}|\partial_{x_2}\bar{v}|^2dx_2=\frac{1}{\varepsilon+h_1(x_{1})-h_2(x_{1})}.
\end{align*}
Now integrating from $-R$ to $R$ for $x_{1}$, we obtain
\begin{align*}
\int_{\Omega_R}|\partial_{x_2}(v_{1}^{1})^{1}|^2dx&=\int_{-R}^{R}\int_{h_2(x_{1})}^{h_1(x_{1})+\varepsilon}|\partial_{x_2}(v_{1}^{1})^{1}|^2dx_2dx_{1}\\
&\geq\int_{{B'_{R}\setminus\ B'_{R_0}}}\int_{h_2(x_{1})}^{h_1(x_{1})+\varepsilon}|\partial_{x_2}\bar{v}|^2dx_2dx_{1}+\int_{-R_{0}}^{R_{0}}\int_0^\varepsilon|\partial_{x_2}\bar{v}|^2dx_2dx_{1}\\
&\geq\frac{1}{C}\int_{{R_0}<|x_{1}|<R}\frac{dx_{1}}{\varepsilon+(|x_{1}|-R_0)^2}+\int_{-R_{0}}^{R_{0}}\int_0^\varepsilon \frac{1}{C\varepsilon^2}dx_2dx_{1}\\
&\geq\frac{1}{C}\left(\frac{|\Sigma'|}{\varepsilon}+\frac{1}{\sqrt{\varepsilon}}\right),
\end{align*}
So we have \eqref{lem34-1}.

{\bf Step 2.} Estimate of $a_{11}^{33}$.
\begin{align*}
|a_{11}^{33}|&=\left|\int_{\Omega}(\mathbb{C}^0e(v_{1}^{3}), e(v_{1}^{3}))dx\right|\leq C\int_{\Omega}|\nabla v_{1}^{3}|^2dx\\
&\leq\,C\int_{\Omega_R}\frac{|x_{1}|^2}{(\varepsilon+d^{2}(x_{1}))^2}dx+C\\
&\leq\,C\int_{\Sigma}\frac{|x_{1}|^2}{\varepsilon^2}dx+C\int_{\Omega_{R}\setminus \Sigma}\frac{((|x_{1}|-R_0)+R_0)^2}{(\varepsilon+(|x_{1}|-R_0)^2)^2}dx+C\\
&\leq \frac{C|\Sigma'|^{3}}{\varepsilon}+C\int_{{B'_R}\setminus\Sigma'}\frac{(|x_{1}|-R_0)^2}{\varepsilon+(|x_{1}|-R_0)^2}dx_{1}+CR_{0}^{2}\int_{{B'_R}\setminus\Sigma'}\frac{dx_{1}}{\varepsilon+(|x_{1}|-R_0)^2}+C\\
&\leq \frac{C|\Sigma'|^{3}}{\varepsilon}+\frac{C|\Sigma'|^{2}}{\sqrt{\varepsilon}}+C.
\end{align*}

On the other hand, by the reasoning used for the lower bound of $a_{11}^{11}$,
\begin{align*}
|a_{11}^{33}|&=\left|\int_{\Omega}(\mathbb{C}^0e(v_{1}^{3}), e(v_{1}^{3}))dx\right| \geq \frac{1}{C}\int_{\Omega}|e(v_{1}^{3})|^2dx\geq \frac{1}{C}\int_{\Omega_R}|\partial_{x_3}(v_{1}^{3})^2|^2dx.
\end{align*}
Notice that $(v_{1}^{3})^2|_{\partial{D}_{1}}=x_{1}\bar{v}|_{\partial{D}_{1}}=x_{1}, (v_{1}^{3})^2|_{\partial{D}_2}
=x_{1}\bar{v}|_{\partial D_2}=0$, and recalling the definition of $\bar{v}$, $x_{1}\bar{v}(x_{1},x_2)$ is linear in $x_2$ for fixed $x_1$, so $x_{1}\bar{v}(x_1, \cdot)$ is also harmonic. Hence its energy is minimal, that is,
\begin{align*}
\int_{h_2(x_{1})}^{h_1(x_{1})+{\varepsilon}}|\partial_{x_2}(v_{1}^{3})^2|^2dx_2\geq
\int_{h_2(x_{1})}^{h_1(x_{1})+{\varepsilon}}|\partial_{x_2}(x_{1}\bar{v})|^2dx_2=\frac{|x_{1}|^{2}}{\varepsilon+h_1(x_{1})-h_2(x_{1})}.
\end{align*}
Integrating on $[-R,R]$ for $x_{1}$, we obtain
\begin{align*}
\int_{\Omega_R}|\partial_{x_2}(v_{1}^{3})^2|^2dx&=\int_{-R}^{R}\int_{h_2(x_{1})}^{h_1(x_{1})+\varepsilon}|\partial_{x_2}(v_{1}^{3})^2|^2dx_2dx_{1}\\
&\geq\int_{R_{0}}^{R}\int_{h_2(x_{1})}^{h_1(x_{1})+\varepsilon}|\partial_{x_2}(v_{1}^{3})^2|^2dx_2dx_{1}+\int_{0}^{R_{0}}\int_0^\varepsilon|\partial_{x_2}(v_{1}^{3})^2|^2dx_2dx_{1}\\
&\geq\frac{1}{C}\int_{R_{0}}^{R}\frac{|x_{1}|^2}{\varepsilon+d^2(x_{1})}dx_{1}+\int_{0}^{R_{0}}\int_0^\varepsilon \frac{|x_{1}|^2}{C\varepsilon^2}dx_2dx_{1}\\
&\geq \frac{1}{C}\left(\frac{|\Sigma'|^{3}}{\varepsilon}+1\right).
\end{align*}
So, for $\varepsilon>0$ small enough, we have \eqref{lem34-2}.

{\bf Step 3.} Estimates of $a_{11}^{12}$ and $a_{11}^{21}$.

By the definition,
\begin{align*}
a_{11}^{12}&=a_{11}^{21}=-\int_{\partial{D}_{1}}\frac{\partial v_{1}^{1}}{\partial \nu_0}\large\Big|_{+}\cdot\psi_2\\
&=-\int_{\partial{D}_{1}}\lambda(\nabla\cdot v_{1}^{1})n_2+\mu\left((\nabla v_{1}^{1}+(\nabla v_{1}^{1})^T)\vec{n}\right)_2\\
&=-\int_{\partial{D}_{1}}\lambda\left(\sum_{k=1}^2\partial_{x_k} (v_{1}^{1})^{k}\right)n_2
+\mu\sum_{l=1}^2\left(\partial_{x_2}(v_{1}^{1})^{l}+\partial_{x_l}(v_{1}^{1})^{2}\right)n_l,
\end{align*}
where $$\vec{n}=\frac{(h'_{1}(x_{1}), -1)}{\sqrt{1+|h'_{1}(x_{1})|^2}}.$$
Due to (\ref{h1h20}), we have
\begin{align}\label{normal}
|n_{1}|=\left|\frac{h'_{1}(x_{1})}{\sqrt{1+|h'_{1}(x_{1})|^2}}\right|\leq Cd(x_{1}),\qquad~\mbox{and}~ |n_2|=\frac{ 1}{\sqrt{1+|h'_{1}(x_{1})|^2}}\leq1.
\end{align}

Let
$$\mathrm{I}_{12}:=\int_{\partial{D}_{1}\cap B_R}\partial_{x_{1}}(v_{1}^{1})^{1}n_2,\quad
\mathrm{II}_{12}:=\int_{\partial{D}_{1}\cap B_R} \partial_{x_2}(v_{1}^{1})^{2}n_2,
$$
and
$$
\mathrm{III}_{12}:=\int_{\partial{D}_{1}\cap B_R} \left(\partial_{x_2}(v_{1}^{1})^{1}+\partial_{x_{1}}(v_{1}^{1})^{2}\right)n_1.
$$
It follows from (\ref{v1-x'2}) and (\ref{normal}) that
\begin{align*}
|\mathrm{I}_{12}|\leq&\int_{\partial{D}_{1}\cap B_R}|\partial_{x_{1}}(v_{1}^{1})^{1}|\leq C\int_{\partial{D}_{1}\cap B_R}\frac{1}{\sqrt{\varepsilon+d^2(x_{1})}}\\
\leq&\,C\left(\int_{\Sigma'}\frac{dx_1}{\sqrt{\varepsilon}}
+\int_{{B'_R}\setminus\Sigma'}\frac{dx_1}{\sqrt{\varepsilon+d^2(x_{1})}}\right),
\end{align*}
where
\begin{equation}\label{eqn_log}
\int_{{B'_R}\setminus\Sigma'}\frac{dx_{1}}{\sqrt{\varepsilon+d^2(x_{1})}}
=\int_{R_0}^{R}\frac{1}{\sqrt{\varepsilon+(r-R_0)^2}}dr
\leq\,C|\log\varepsilon|.
\end{equation}
Thus,
\begin{align*}
|\mathrm{I}_{12}|\leq C\left(\frac{|\Sigma'|}{\sqrt{\varepsilon}}+|\log\varepsilon|\right).
\end{align*}
Recalling the definition of $\tilde{v}_{1}^{1}$, we know $(\tilde{v}_{1}^{1})^{2}=0$. So by (\ref{nabla_w_i02}), we have
\begin{align*}
|\mathrm{II}_{12}|&=\left|\int_{\partial{D}_{1}\cap B_R}\partial_{x_2}(v_{1}^{1})^{2} n_2\right|\\
&\leq\left|\int_{\partial{D}_{1}\cap B_R}\partial_{x_2}(\tilde{v}_{1}^{1})^{2}n_2\right|+\left|\int_{\partial{D}_{1}\cap B_R}\partial_{x_2}((v_{1}^{1})^{2}-(\tilde{v}_{1}^{1})^2)n_2\right|\nonumber\\
&\leq\int_{\Sigma'}\frac{dx_1}{\sqrt{\varepsilon}}+C\int_{{B'_R}\backslash\Sigma'}    \frac{dx_1}{\sqrt{\varepsilon+d^2(x_{1})}}\\
&\leq C\left(\frac{|\Sigma'|}{\sqrt{\varepsilon}}+|\log\varepsilon|\right),
\end{align*}
Similarly,
\begin{align*}
|\mathrm{III}_{12}|\leq&C\int_{\partial{D}_{1}\cap B_R}\frac{d(x_{1})}{\varepsilon+d^2(x_{1})}
\leq C\int_{{B'_R}\setminus\Sigma'}\frac{dx_{1}}{\sqrt{\varepsilon+d^2(x_{1})}}\leq C |\log\varepsilon|.
\end{align*}
Therefore, together with the estimates of $|\mathrm{I}_{12}|$ and $|\mathrm{II}_{12}|$, and the boundedness of $|\nabla{v}_{1}^{1}|$ on $\partial{D}_{1}\setminus B_R$, we obtain
\begin{align*}
|a_{11}^{12}|=|a_{11}^{21}|&\leq |\lambda||\mathrm{I}_{12}|+|\lambda+2\mu||\mathrm{II}_{12}|+|\mu||\mathrm{III}_{12}|+C\leq C\left(\frac{|\Sigma'|}{\sqrt{\varepsilon}}+|\log\varepsilon|\right).
\end{align*}

{\bf Step 4.} Estimates of $a_{11}^{13}$ and $a_{11}^{31}$.

Recalling $\psi_{3}=(-x_{2},x_{1})^{T}$, and making use of the boundedness of $|\nabla{v}_{1}^{1}|$ on $\partial{D}_{1}\setminus B_R$, we have
\begin{align*}
a_{11}^{13}=a_{11}^{31}=&-\int_{\partial{D}_{1}}\frac{\partial v_{1}^{1}}{\partial \nu_0}\large\Big|_{+}\cdot\psi_{3}\\
=&\int_{\partial{D}_{1}\cap B_R}\left(\lambda\big(\sum_{k=1}^2\partial_{x_k} (v_{1}^{1})^{k}\big)n_1
+\mu\sum_{l=1}^2\left(\partial_{x_{1}}(v_{1}^{1})^{l}+\partial_{x_l}(v_{1}^{1})^{1}\right)n_l\right)x_{2}\\
&-\int_{\partial{D}_{1}\cap B_R}\left(\lambda\big(\sum_{k=1}^2\partial_{x_k} (v_{1}^{1})^{k}\big)n_2
+\mu\sum_{l=1}^2\left(\partial_{x_2}(v_{1}^{1})^{l}+\partial_{x_l}(v_{1}^{1})^{2}\right)n_l\right)x_{1}+O(1)\\
=:&\mathrm{I}_{13}-\mathrm{II}_{13}+O(1),
\end{align*}
where, since $x_2=\varepsilon+h_1(x_{1})$, and $|x_2|\leq C(\varepsilon+d^2(x_{1}))$, it follows from \eqref{v1-bounded12} that
\begin{align*}
|\mathrm{I}_{13}|\leq\,C\int_{B'_{R}}|\nabla v_{1}^{1}||x_{2}|dx_1\leq\,C,
\end{align*}
and we divide $\mathrm{II}_{13}$ further,
\begin{align*}
\mathrm{II}_{13}:=&\,\lambda\int_{\partial{D}_{1}\cap B_R}\partial_{x_{1}}(v_{1}^{1})^{1}n_2x_{1}+(\lambda+2\mu)\int_{\partial{D}_{1}\cap B_R}\partial_{x_2}(v_{1}^{1})^{2}n_2x_{1}\\
&+\mu\int_{\partial{D}_{1}\cap B_R}\left(\partial_{x_2}(v_{1}^{1})^{2}+\partial_{x_{1}}(v_{1}^{1})^{2}\right)n_1x_{1}\\
=&:\,\lambda\mathrm{II}_{13}^1+(\lambda+2\mu)\mathrm{II}_{13}^2+\mu\mathrm{II}_{13}^3.
\end{align*}
From Corollary \ref{cor312}, and using \eqref{eqn_log} and $|x_2|\leq C(\varepsilon+d^2(x_{1}))$ again, we have
\begin{align*}
|\mathrm{II}_{13}^1|&\leq\left|\int_{\partial{D}_{1}\cap B_R}\partial_{x_{1}}(v_{1}^{1})^{1}n_2x_{1}\right|\\
&\leq\int_{\Sigma'}\frac{C|x_{1}|}{\sqrt{\varepsilon}}+\int_{B'_R\backslash\Sigma'}\frac{C|x_{1}|}{\sqrt{\varepsilon+d^2(x_{1})}}\\
&\leq\frac{C|\Sigma'|^{2}}{\sqrt{\varepsilon}}+\int_{{B'_R}\backslash\Sigma'}\frac{R_0+d(x_{1})}{\sqrt{\varepsilon+d^2(x_{1})}}\\
&\leq \frac{C|\Sigma'|^{2}}{\sqrt{\varepsilon}}+C|\Sigma'||\log\varepsilon|+C.
\end{align*}
In view of the fact that $(\tilde{v}_{1}^{1})^2=0$,
\begin{align*}
|\mathrm{II}_{13}^2|&=\left|\int_{\partial{D}_{1}\cap B_R}\partial_{x_2}(\tilde{v}_{1}^{1})^2n_2x_{1}+\int_{\partial{D}_{1}\cap B_R}\partial_{x_2}(v_{1}^{1}-\tilde{v}_{1}^{1})^2n_2x_{1}\right|\\
&\leq\int_{\Sigma'}\frac{C|x_{1}|}{\sqrt{\varepsilon}}+\int_{{B'_R}\backslash\Sigma'}\frac{C|x_{1}|}{\sqrt{\varepsilon+d^2(x_{1})}}\\
&\leq \frac{C|\Sigma'|^{2}}{\sqrt{\varepsilon}}+C|\Sigma'||\log\varepsilon|+C,
\end{align*}
The other fact that $n_{1}=0$ on $\partial{D}_{1}\cap\overline{\Sigma}$ clearly implies that
\begin{align*}
|\mathrm{II}_{13}^3|&\leq C\int_{{B'_R}\backslash\Sigma'}\frac{d(x_{1})|x_{1}|}{\varepsilon+d^2(x_{1})}\leq C\int_{{B'_R}\backslash\Sigma'}\frac{|x_{1}|}{\sqrt{\varepsilon+d^2(x_{1})}}\\
&\leq C|\Sigma'||\log\varepsilon|+C.
\end{align*}
Combining these estimates together yields
\begin{align*}
|\mathrm{II}_{13}|&\leq \frac{C|\Sigma'|^{2}}{\sqrt{\varepsilon}}+C.
\end{align*}
So \eqref{lem34-4} is obtained.

{\bf Step 5.} Estimates of $a_{11}^{23}$ and $a_{11}^{32}$.

By the definition,
\begin{align*}
a_{11}^{23}=a_{11}^{32}=&-\int_{\partial{D}_{1}}\frac{\partial v_{1}^{2}}{\partial \nu_0}\large\Big|_{+}\cdot\psi_{3}\\
=&\int_{\partial{D}_{1}\cap B_R}\left(\lambda\big(\sum_{k=1}^2\partial_{x_k} (v_{1}^{2})^k\big)n_1
+\mu\sum_{l=1}^2\left(\partial_{x_{1}}(v_{1}^{2})^{l}+\partial_{x_l}(v_{1}^{2})^{1}\right)n_2\right)x_{2}\\
&-\int_{\partial{D}_{1}\cap B_R}\left(\lambda\big(\sum_{k=1}^2\partial_{x_k} (v_{1}^{2})^k\big)n_2
+\mu\sum_{l=1}^2\left(\partial_{x_2}(v_{1}^{2})^{l}+\partial_{x_l}(v_{1}^{2})^2\right)n_l\right)x_{1}+O(1)\\
=&:\mathrm{I}_{23}-\mathrm{II}_{23}+O(1).
\end{align*}
First, noting that $|x_2|\leq C(\varepsilon+d^2(x_{1}))$, we have
\begin{equation}\label{I23}
|\mathrm{I_{23}}|\leq C\int_{B_R'}\frac{|x_{2}|}{\varepsilon+d^2(x_{1})}dx_1\leq C.
\end{equation}
Next, denote
$$
\mathrm{II}_{23}^1:=\int_{\partial{D}_{1}\cap B_R}\partial_{x_{1}}(v_{1}^{2})^{1}n_2x_{1},\quad
\mathrm{II}_{23}^2:=\int_{\partial{D}_{1}\cap B_R}\partial_{x_2}(v_{1}^{2})^2n_2x_{1},
$$
$$
\mathrm{II}_{23}^3:=\int_{\partial{D}_{1}\cap B_R}\big(\partial_{x_2}(v_{1}^{2})^{1}+\partial_{x_{1}}(v_{1}^{2})^{2}\big)n_1x_{1}.
$$
By using the fact that $(\tilde{v}_{1}^{2})^{1}=0$ on $\partial{D}_{1}$  and \eqref{eqn_log}, it is easy to see as before that
\begin{align}\label{II231}
|\mathrm{II}_{23}^1|&=\left|\int_{\partial{D}_{1}\cap B_R}\partial_{x_{1}}(\tilde{v}_{1}^{2})^{1}n_2x_{1}+\int_{\partial{D}_{1}\cap B_R}\partial_{x_{1}}(v_{1}^{2}-\tilde{v}_{1}^{2})^{1}n_2x_{1}\right|\nonumber\\
&\leq\,C\int_{\Sigma'}\frac{|x_{1}|}{\sqrt{\varepsilon}}dx_{1}+C\int_{{ B'_R}\backslash\Sigma'}\frac{|x_{1}|}{\sqrt{\varepsilon+d^2(x_{1})}}dx_{1}\nonumber\\
&\leq\frac{C|\Sigma'|^{2}}{\sqrt{\varepsilon}}+C|\Sigma'||\log\varepsilon|+C.
\end{align}
By $n_{1}=0$ on $\partial{D}_{1}\cap\overline{\Sigma}$  and (\ref{normal}),
\begin{align}\label{II233}
|\mathrm{II}_{23}^3|&\leq C\int_{B'_R\backslash\Sigma'}\frac{d(x_{1})|x_{1}|}{\varepsilon+d^2(x_{1})}dx_{1}\leq C\int_{B'_R\backslash\Sigma'}\frac{|x_{1}|}{\sqrt{\varepsilon+d^2(x_{1})}}dx_{1}\leq C|\Sigma'||\log\varepsilon|+C.
\end{align}

Finally, the term $\mathrm{II}_{23}^2$ is not immediate, we further divide it into three parts,
\begin{align*}
\mathrm{II}_{23}^2&=\int_{\partial{D}_{1}\cap \overline{\Sigma}}\partial_{x_2}(\tilde{v}_{1}^{2})^{2}n_2x_{1}+\int_{\partial{D}_{1}\cap B_R\backslash \overline\Sigma}\partial_{x_2}(\tilde{v}_{1}^{2})^{2}n_2x_{1}+\int_{\partial{D}_{1}\cap B_R}\partial_{x_2}(v_{1}^{2}-\tilde{v}_{1}^{2})^{2}n_2x_{1}\\
&=:\mathrm{II}_{23}^{2(1)}+\mathrm{II}_{23}^{2(2)}+\mathrm{II}_{23}^{2(3)}.
\end{align*}
Here the first and third terms are still easy to handle. Noticing that $\Sigma'=(-R_0,R_0)$, 
\begin{align*}
|\mathrm{II}_{23}^{2(1)}|=\left|\int_{\partial{D}_{1}\cap \overline{\Sigma}}\partial_{x_2}(\tilde{v}_{1}^{2})^{2}n_2x_{1}\right|=\left|\int_{\Sigma'}\frac{x_{1}}{\varepsilon}dx_{1}\right|=0,
\end{align*}
and by using \eqref{nabla_w_i02} and \eqref{eqn_log} again,
\begin{align*}
|\mathrm{II}_{23}^{2(3)}|&\leq\int_{\Sigma'}\frac{C|x_{1}|}{\sqrt{\varepsilon}}dx_{1}+\int_{B'_R\backslash\Sigma'}\frac{C|x_{1}|}{\sqrt{\varepsilon+d^2(x_{1})}}dx_{1}\leq \frac{C|\Sigma'|^{2}}{\sqrt{\varepsilon}}+C|\Sigma'||\log\varepsilon|+C.
\end{align*}

In the following, we use the Taylor expansions of $h_{1}-h_{2}$ to estimate the middle term $\mathrm{II}_{23}^{2(2)}$.
\begin{align*}
|\mathrm{II}_{23}^{2(2)}|&\leq\left|\int_{{B'_R}\backslash\Sigma'}\frac{x_{1}}{\varepsilon+h_1(x_{1})-h_2(x_{1})}dx_{1}\right|\\
&=\left|\int_{R_{0}}^{R}\frac{x_{1}}{\varepsilon+h_1(x_{1})-h_2(x_{1})}dx_{1}+\int_{-R}^{-R_{0}}\frac{x_{1}}{\varepsilon+h_1(x_{1})-h_2(x_{1})}dx_{1}\right|.
\end{align*}
By the assumptions on $h_{1}$ and $h_{2}$, we have $(h_{1}-h_{2})$'s Taylor expansions at $x_{1}=\pm R_{0}$ as follows,
$$(h_{1}-h_{2})(x_{1})=c_{\pm}(x_{1}\mp R_{0})^{2}+o\left((x_{1}\mp R_{0})^{2}\right),\quad\mbox{where}~ c_{\pm}=\frac{1}{2}(h_{1}-h_{2})''(\pm R_{0}).$$
Thus,
\begin{align*}
\int_{R_{0}}^{R}\frac{x_{1}}{\varepsilon+h_1(x_{1})-h_2(x_{1})}dx_{1}&=\int_{R_0}^{R}\frac{x_{1}}{\varepsilon+c_{+}(x_{1}-R_0)^2}dx_{1}\\
&+\int_{R_0}^{R}\frac{x_{1}o\left((x_{1}-R_{0})^{2}\right)}{[\varepsilon+h_1(x_{1})-h_2(x_{1})][\varepsilon+c_{+}(x_{1}-R_0)^2]}dx_{1},
\end{align*}
and
\begin{align*}
\int^{-R_{0}}_{-R}\frac{x_{1}}{\varepsilon+h_1(x_{1})-h_2(x_{1})}dx_{1}&=\int^{-R_{0}}_{-R}\frac{x_{1}}{\varepsilon+c_{-}(x_{1}+R_0)^2}dx_{1}\\
&+\int^{-R_{0}}_{-R}\frac{x_{1}o\left((x_{1}+R_{0})^{2}\right)}{[\varepsilon+h_1(x_{1})-h_2(x_{1})][\varepsilon+c_{-}(x_{1}+R_0)^2]}dx_{1}.
\end{align*}
By calculation, we have
\begin{align*}
&\left|\int_{R_0}^{R}\frac{x_{1}}{\varepsilon+c_{+}(x_{1}-R_0)^2}dx_{1}+\int^{-R_{0}}_{-R}\frac{x_{1}}{\varepsilon+c_{-}(x_{1}+R_0)^2}dx_{1}\right|\\
&\leq\left|\frac{1}{2c_{+}}-\frac{1}{2c_{-}}\right||\log\varepsilon|+\frac{C|\Sigma'|}{\sqrt{\varepsilon}}+C,
\end{align*}
and the remaiders are  bounded by $o(1/\sqrt{\varepsilon})$. Hence
\begin{align*}
|\mathrm{II}_{23}^{2(2)}|&\leq \,\left|\frac{1}{2c_{+}}-\frac{1}{2c_{-}}\right||\log\varepsilon|+\frac{C|\Sigma'|}{\sqrt{\varepsilon}}+C.
\end{align*}
So that
$$|\mathrm{II}_{23}^2|\leq \frac{C|\Sigma'|^{2}}{\sqrt{\varepsilon}}+\left(C|\Sigma'|+\left|\frac{1}{2c_{+}}-\frac{1}{2c_{-}}\right|\right)|\log\varepsilon|+C.$$
We remark that if $\Sigma'=\{0'\},$ then by the assumption that $\partial{D_{1}}$ and $\partial{D}_{2}$ are of $C^{2,\alpha}$, we have $c_{+}=c_{-}$, thus it is obvious that $|\mathrm{II}_{23}^2|\leq C$. Therefore, the above estimates of $\mathrm{II}_{23}^2$ is a new phenomenon due to $|\Sigma'|>0$. So that for $\varepsilon>0$ small enough,
$$|\mathrm{II}_{23}^2|\leq \frac{C|\Sigma'|^{2}}{\sqrt{\varepsilon}}+C.$$
Combining with \eqref{I23}, \eqref{II231}  and \eqref{II233} yields estimate \eqref{lem34-5}.

The proof of Lemma \ref{lem_a11} is finished.
\end{proof}

\section{Proof of Theorem \ref{thmd}}\label{sec_highd}

To extend the result to general $d\geq3$ in this section we neeed to modify the argument. Here again, we remark that it is worth paying attention to the role of  the flatness between $\partial{D}_{1}$ and $\partial{D}_{2}$ in the blow-up analysis of $|\nabla u|$.

\subsection{Main Ingredients and Proof of Theorem \ref{thmd}}

For problem (\ref{equ_v1}) in higher dimensions $d\geq3$, taking
$$\psi=\psi_{\alpha},\quad\mbox{and}\quad \tilde{v}_1^\alpha:=\bar{v}\psi_{\alpha},\qquad\alpha=1,2,\cdots,\frac{d(d+1)}{2},$$
in Theorem \ref{thm2.1}, we have
\begin{corollary}\label{corol3.4}
Assume the above, let $v_{i}^\alpha$, $i=1,2$, be the
weak solutions of \eqref{equ_v1}, respectively. Then for sufficiently small $0<\varepsilon<1/2$, we have
\begin{equation}\label{dnabla_w_i0}
\|\nabla(v_{i}^\alpha-\tilde{v}_i^\alpha)\|_{L^{\infty}(\Omega)}\leq\,\frac{C}{\sqrt{\varepsilon+d^{2}(x')}},\quad \alpha=1,2,\cdots,d;
\end{equation}
\begin{equation}\label{dnabla_w_i00}
\|\nabla(v_{i}^\alpha-\tilde{v}_i^\alpha)\|_{L^{\infty}(\Omega)}\leq\,\frac{C|x'|}{\sqrt{\varepsilon+d^{2}(x')}},\quad \alpha=d+1,\cdots,\frac{d(d+1)}{2};
\end{equation}
Consequently, by using \eqref{eq1.7},
\begin{equation}\label{v1-bounded1}
\frac{1}{C\left(\varepsilon+d^{2}(x')\right)}\leq|\nabla v_{i}^\alpha(x)|\leq\,\frac{C}{\varepsilon+d^{2}(x')},\quad\,~i=1,2;~\alpha=1,\cdots,d,~x\in\Omega_{R};
\end{equation}
\begin{equation}\label{v1-x'}
|\nabla_{x'}v_{i}^\alpha(x)|\leq\,\frac{C}{\sqrt{\varepsilon+d^{2}(x')}},\qquad\,~i=1,2;~\alpha=1,\cdots,d,~x\in\Omega_{R};
\end{equation}
\begin{equation}\label{dv1---bounded1}
|\nabla v_i^{\alpha}(x)|\leq\frac{C(\varepsilon+|x'|)}{\varepsilon+d^2(x')},\qquad\,~i=1,2;~\alpha=d+1,\cdots,\frac{d(d+1)}{2},~x\in \Omega_{R};
\end{equation}
\begin{equation}\label{v1--bounded1}
|\nabla v_{i}^\alpha(x)|\leq\,C,\qquad\,~i=1,2;~\alpha=1,2,\cdots,\frac{d(d+1)}{2},~x\in\Omega\setminus\Omega_{R};
\end{equation}
and
$$|\nabla(v_{1}^{\alpha}+v_{2}^\alpha)(x)|\leq\,C\|\varphi\|_{C^{2}(\partial D)},\quad\,\alpha=1,2,\cdots,\frac{d(d+1)}{2},\quad~x\in\Omega;
$$$$|\nabla{v}_{3}(x)|\leq C\|\varphi\|_{C^{2}(\partial D)},\quad\,x\in\Omega.
$$
where $C$ is a {\it universal constant}, independent of $|\Sigma'|$.
\end{corollary}

For $d\geq$3, from the fourth line of \eqref{maineqn}, we have
\begin{equation}\label{system3}
\begin{cases}
\sum_{\alpha=1}^{\frac{d(d+1)}{2}}C_{1}^\alpha a_{11}^{\alpha\beta}+\sum_{\alpha=1}^{\frac{d(d+1)}{2}}C_{2}^\alpha a_{21}^{\alpha\beta}=b_1^\beta,\\\\
\sum_{\alpha=1}^{\frac{d(d+1)}{2}}C_{1}^\alpha a_{12}^{\alpha\beta}+\sum_{\alpha=1}^{\frac{d(d+1)}{2}}C_{2}^\alpha a_{22}^{\alpha\beta}=b_2^\beta,
\end{cases}~~~\beta=1,2,\cdots,\frac{d(d+1)}{2}.
\end{equation}
Similarly as in Section \ref{sec_d=2}, for simplicity, we use $a_{ij}$ to denote the $\frac{d(d+1)}{2}\times\frac{d(d+1)}{2}$ matrix $(a_{ij}^{\alpha\beta})$ and make use of the first $\frac{d(d+1)}{2}$ equations in \eqref{system3}:
\begin{equation}\label{c1-c23}
a_{11}C_1+a_{21}C_2=b_1,
\end{equation}
where
\begin{equation*}
C_i=(C_i^1,C_i^2,\cdots,C_{i}^{\frac{d(d+1)}{2}})^T, \quad i=1,2,\quad\mbox{and}\quad b_1=(b_1^1,b_1^2,\cdots,b_{1}^{\frac{d(d+1)}{2}})^T,
\end{equation*}
to estimate
$$|C_{1}^{\alpha}-C_{2}^{\alpha}|, \quad\alpha=1,\cdots,\frac{d(d+1)}{2}.
$$
Now rewrite \eqref{c1-c23} as
\begin{equation}\label{dp}
a_{11}(C_1-C_2)=p,\quad\quad p:=b_1-(a_{11}+a_{21})C_2.
\end{equation}

In order to emphasize the role of $|\Sigma'|$ in the blow-up analysis of $|\nabla u|$ and avoid complicated calculation, we here assume $\Sigma'$ is a ball in $\mathbb{R}^{d-1}$. More general domain is considered after the proof of Lemma \ref{glemma4.1}.

\begin{lemma}\label{glemma4.1}
Suppose $\Sigma'=B'_{R_{0}}$. Then
\begin{align}
\hspace{-.3cm}\mbox{for}~d=3,\qquad \frac{1}{C}\left(
\frac{|\Sigma'|}{\varepsilon}+|\log\varepsilon|\right)\leq a_{11}^{\alpha\alpha}
&\leq
C\left(
\frac{|\Sigma'|}{\varepsilon}+|\log\varepsilon|\right)
,\quad \alpha=1, 2,3;\label{glem4.1-13}
\end{align}
\begin{align}
\hspace{-.3cm}\mbox{for}~d\geq4,\quad\qquad\qquad\frac{1}{C}\left(\frac{|\Sigma'|}{\varepsilon}+1\right)\leq a_{11}^{\alpha\alpha}
&\leq C\left(\frac{|\Sigma'|}{\varepsilon}+1\right),\quad \alpha=1, 2,\cdots,d;\label{glem4.1-1}
\end{align}
and for $d\geq3$,
\begin{align}
\frac{1}{C}\left(\frac{|\Sigma'|^{\frac{d+1}{d-1}}}{\varepsilon}+1\right)\leq a_{11}^{\alpha\alpha}&\leq C\left(\frac{|\Sigma'|^{\frac{d+1}{d-1}}}{\varepsilon}+1\right),\quad \alpha=d+1,\cdots,\frac{d(d+1)}{2};\label{gdd4}
\end{align}
\begin{align}\label{gdd3}
|a_{11}^{\alpha\beta}|=|a_{11}^{\beta\alpha}|&\leq C\left(\frac{|\Sigma'|}{\sqrt{\varepsilon}}+|\Sigma'|^{\frac{d-2}{d-1}}|\log\varepsilon|+1\right), \quad \alpha,\beta=1,\cdots,d; \quad \alpha\neq\beta;
\end{align}
\begin{align}\label{gd1d4}
|a_{11}^{\alpha\beta}|=|a_{11}^{\beta\alpha}|&\leq C\left(\frac{|\Sigma'|^{\frac{d}{d-1}}}{\sqrt{\varepsilon}}+|\Sigma'||\log\varepsilon|+1\right), \quad \alpha=1,\cdots,d;\quad \beta=d+1,\cdots,\frac{d(d+1)}{2};
\end{align}
\begin{align}\label{gd4d5}
|a_{11}^{\alpha\beta}|=|a_{11}^{\beta\alpha}|&\leq C\left(\frac{|\Sigma'|^{\frac{d+1}{d-1}}}{\sqrt{\varepsilon}}+|\Sigma'|^{\frac{d}{d-1}}|\log\varepsilon|+1\right), \quad \alpha,\beta=d+1,\cdots,\frac{d(d+1)}{2};\quad \alpha\neq\beta.
\end{align}
\end{lemma}
From these estimates for each element of matrix $(a_{11}^{\alpha\beta})$ above, it is obvious that $a_{11}$ is invertible whenever $|\Sigma'|>0$.

\begin{lemma}\label{lemma_pd}
Let $C_i^\alpha$ be defined in \eqref{decom_u}, $i=1,2$, $\alpha=1,2,\cdots,\frac{d(d+1)}{2}$. Then $C_{i}^{\alpha}$ is bounded, and
\begin{align*}
|a_{11}^{\alpha\beta}+a_{21}^{\alpha\beta}|&\leq C\|\varphi\|_{C^{2}(\partial D)}, \quad \alpha, \beta=1,\cdots,\frac{d(d+1)}{2};\\
|b_1^\beta|&\leq C\|\varphi\|_{C^{2}(\partial D)}, \quad \beta=1,\cdots,\frac{d(d+1)}{2}.
\end{align*}
Consequently,
\begin{equation}\label{dp3-c}
|p|\leq C\|\varphi\|_{C^{2}(\partial D)}.
\end{equation}
\end{lemma}

The proof is very similar with that of Lemma \ref{lemma p}. we omit it. Using Lemma \ref{glemma4.1} and \ref{lemma_pd}, we have

\begin{prop}\label{dprop2}
Let $C_i^\alpha$ be defined in \eqref{decom_u}, $i=1,2$, $\alpha=1,2,\cdots,\frac{d(d+1)}{2}$. Then
\begin{align}\label{dc-c}
\hspace{-.5cm}\mbox{for}~d=3,\qquad\qquad |C_{1}^{\alpha}-C_{2}^{\alpha}|\leq \frac{C\varepsilon}{|\Sigma'|+\varepsilon|\log\varepsilon|}\|\varphi\|_{C^{2}(\partial D)},\quad\alpha=1,2,3,
\end{align}
\begin{align}\label{dc-c}
\hspace{-.5cm}\mbox{for}~d\geq4,\qquad\qquad |C_{1}^{\alpha}-C_{2}^{\alpha}|\leq \frac{C\varepsilon}{|\Sigma'|+\varepsilon}\|\varphi\|_{C^{2}(\partial D)},\quad\alpha=1,2,\cdots,d,
\end{align}
and for $d\geq3$,
\begin{align}\label{dc-c3}
|C_{1}^{\alpha}-C_{2}^{\alpha}|\leq \frac{C\varepsilon}{|\Sigma'|^{\frac{d+1}{d-1}}+\varepsilon}\|\varphi\|_{C^{2}(\partial D)}.\quad\alpha=d+1,\cdots,\frac{d(d+1)}{2}.
\end{align}
\end{prop}

Now we are in position to prove Theorem \ref{thmd}.

\begin{proof}[Proof of Theorem \ref{thmd}] Since
\begin{align*}
|\nabla{u}(x)|&\leq\sum_{\alpha=1}^{\frac{d(d+1)}{2}}\Big|(C_{1}^\alpha-C_{2}^\alpha)\nabla{v}_{1}^\alpha(x)\Big|+\sum_{\alpha=1}^{\frac{d(d+1)}{2}}\Big|C_{2}^\alpha\nabla({v}_{1}^\alpha+{v}_{2}^\alpha)(x)\Big|
+|\nabla{v}_{3}(x)|\\
&\leq\sum_{\alpha=1}^{d}\Big|(C_{1}^\alpha-C_{2}^\alpha)\nabla{v}_{1}^\alpha(x)\Big|+\sum_{\alpha=d+1}^{\frac{d(d+1)}{2}}\Big|(C_{1}^\alpha-C_{2}^\alpha)\nabla{v}_{1}^\alpha(x)\Big|+C\|\varphi\|_{C^2(\partial D)},
\end{align*}
using  Corollary \ref{corol3.4}, Lemma \ref{lemma_pd} and Proposition \ref{dprop2}, we obtain \eqref{thmd-1} and \eqref{thmd-2} for $x\in\Omega_{R}$. Estimate \eqref{thmd-3} can be obtained by standard theory of elliptic systems, see e.g. \cite{AD1,AD2}. The proof of Theorem \ref{thmd} is completed.
\end{proof}

\subsection{Proof of Lemma \ref{glemma4.1}}

\begin{proof}[Proof of Lemma \ref{glemma4.1}]

{\bf Step 1.} Estimate $a_{11}^{\alpha\alpha}$, $\alpha=1,\cdots,\frac{d(d+1)}{2}$.

Firstly, for $\alpha=1,\cdots,d$,
\begin{equation*}
a_{11}^{\alpha\alpha}=\int_{\Omega}(\mathbb{C}^0e(v_{1}^{\alpha}), e(v_{1}^{\alpha}))dx\leq C\int_{\Omega}|\nabla v_{1}^{\alpha}|^2dx.
\end{equation*}
We decompose the integral on the right hand side into three parts as before,
\begin{align}\label{da11'}
\int_{\Omega}|\nabla v_{1}^{\alpha}|^2=\int_{\Sigma}|\nabla v_{1}^{\alpha}|^{2}+\int_{\Omega_{R}\setminus\Sigma}|\nabla v_1^{\alpha}|^{2}+\int_{\Omega\setminus \Omega_{R}}|\nabla v_{1}^{\alpha}|^{2}.
\end{align}
For the first term, by using (\ref{v1-bounded1}),
\begin{align}\label{dfirst_term}
\int_{\Sigma}|\nabla v_{1}^{\alpha}|^{2}\leq\int_{\Sigma'}\int_{0}^{\varepsilon}\frac{C}{\varepsilon^2}dx_ddx'\leq\frac{C|\Sigma'|}{\varepsilon}.
\end{align}
For the last term of \eqref{da11'}, by using (\ref{v1--bounded1}),
\begin{align}\label{d-termd}
\int_{\Omega\setminus \Omega_{R}}|\nabla v_{1}^{\alpha}|^{2}\leq C.
\end{align}
For the middle term of \eqref{da11'},
\begin{align*}
\int_{\Omega_{R}\setminus\Sigma}|\nabla v_{1}^{\alpha}|^{2}&\leq C\int_{\Omega_{R}\setminus\Sigma}\frac{dx}{(\varepsilon+d^2(x'))^2}\leq C\int_{B'_{R_{}}\setminus\Sigma'}\frac{dx'}{\varepsilon+d^2(x')}\\
&\leq C\int_{R_{0}}^{R}\frac{Cr^{d-2}}{\varepsilon+(r-R_{0})^{2}}dr\leq C\int_{0}^{R-R_0}\frac{(t+R_0)^{d-2}}{\varepsilon+t^2}dt\\
&\leq CR_0^{d-2}\int_{0}^{R-R_0}\frac{1}{\varepsilon+t^2}dt+C\int_{0}^{R-R_0}\frac{{t^{d-2}}}{\varepsilon+t^2}dt\\
&\leq C
\begin{cases}
\frac{|\Sigma'|^{\frac{1}{2}}}{\sqrt{\varepsilon}}+|\log\varepsilon|,&d=3,\\
\frac{|\Sigma'|^{\frac{d-2}{d-1}}}{\sqrt{\varepsilon}}+1,&d\geq4,
\end{cases}
\end{align*}
which, together with \eqref{dfirst_term} and \eqref{d-termd}, implies that
\begin{align*}
a_{11}^{\alpha\alpha}&\leq
C\int_{\Omega}|\nabla v_{1}^{\alpha}|^2dx\leq C
\begin{cases}
\frac{|\Sigma'|}{\varepsilon}+|\log\varepsilon|,&d=3,\\
\frac{|\Sigma'|}{\varepsilon}+1,&d\geq4.
\end{cases}
\end{align*}

On the other hand,
\begin{align*}
a_{11}^{\alpha\alpha}&=\int_{\Omega}(\mathbb{C}^0e(v_{1}^{\alpha}), e(v_{1}^{\alpha}))dx\geq \frac{1}{C}\int_{\Omega}|e(v_{1}^{\alpha})|^2dx\\
&\geq\frac{1}{C}\int_{\Omega}|e_{\alpha d}(v_{1}^{\alpha})|^2dx\geq \frac{1}{C}\int_{\Omega_R}|\partial_{x_d}(v_{1}^{\alpha})^\alpha|^2dx.
\end{align*}
By the reasoning as in Lemma \ref{lem_a11}, noticing that $(v_{1}^{\alpha})^\alpha|_{\partial{D}_{1}}=\bar{v}|_{\partial{D}_{1}}=1, (v_{1}^{\alpha})^\alpha|_{\partial{D}_2}
=\bar{v}|_{\partial D_2}=0$, and recalling the definition of $\bar{v}$, $\bar{v}(x',x_d)$ is linear in $x_d$ for any fixed $x'$, so $\bar{v}(x', \cdot)$ is harmonic, hence its energy is minimal, that is,
\begin{align*}
\int_{h_2(x')}^{h_1(x')+{\varepsilon}}|\partial_{x_d}(v_{1}^{\alpha})^\alpha|^2dx_d\geq
\int_{h_2(x')}^{h_1(x')+{\varepsilon}}|\partial_{x_d}\bar{v}|^2dx_d=\frac{1}{\varepsilon+h_1(x')-h_2(x')}.
\end{align*}
Integrating on $B'_R(0')$ for $x'$, we obtain
\begin{align*}
\int_{\Omega_R}|\partial_{x_d}(v_{1}^{\alpha})^\alpha|^2dx&=\int_{B'_R}\int_{h_2(x')}^{h_1(x')+\varepsilon}|\partial_{x_d}(v_{1}^{\alpha})^\alpha|^2dx_ddx'\\
&\geq\int_{{B'_{R}\setminus\Sigma'}}\int_{h_2(x')}^{h_1(x')+\varepsilon}|\partial_{x_d}(v_{1}^{\alpha})^\alpha|^2dx_ddx'+\int_{\Sigma'}\int_0^\varepsilon|\partial_{x_d}(v_{1}^{\alpha})^\alpha|^2dx_ddx'\\
&\geq\frac{1}{C}\int_{{B'_{R}\setminus\Sigma'}}\frac{dx'}{\varepsilon+d^2(x')}+\int_{\Sigma'}\int_0^\varepsilon \frac{1}{C\varepsilon^2}dx_ddx'\\
&\geq \frac{1}{C}\begin{cases}
\frac{|\Sigma'|}{\varepsilon}+|\log\varepsilon|,&d=3,\\
\frac{|\Sigma'|}{\varepsilon}+1,&d\geq4.
\end{cases}
\end{align*}
Thus, \eqref{glem4.1-13} and \eqref{glem4.1-1} are obtained.

For $\alpha=d+1,\cdots,\frac{d(d+1)}{2}$, we have
\begin{align*}
a_{11}^{\alpha\alpha}&=\int_{\Omega}(\mathbb{C}^0e(v_{1}^{\alpha}), e(v_{1}^{\alpha}))dx\leq C\int_{\Omega}|\nabla v_{1}^{\alpha}|^2dx\\
&\leq\,C\int_{\Omega_R}\frac{|x'|^2}{(\varepsilon+d^2(x'))^2}dx+C\\
&\leq\,C\int_{\Sigma'}\frac{R^2_0}{\varepsilon}dx'+C\int_{{B'_R}\setminus\Sigma'}\frac{(d(x')+R_0)^2}{\varepsilon+d^2(x')}dx'+C\\
&\leq \frac{CR_0^{d+1}}{\varepsilon}+C\int_{{B'_R}\setminus\Sigma'}\frac{d^2(x')}{\varepsilon+d^2(x')}dx'+CR_{0}^{2}\int_{{B'_R}\setminus\Sigma'}\frac{dx'}{\varepsilon+d^2(x')}+C\\
&\leq C\begin{cases}
\frac{|\Sigma'|^2}{\varepsilon}+|\Sigma'||\log\varepsilon|+1,&d=3,\\
\frac{|\Sigma'|^{\frac{d+1}{d-1}}}{\varepsilon}+1,&d\geq4.
\end{cases}
\end{align*}

On the other hand, by observation, there exist some $1\leq j<d$ and $1\leq k<d$ such that $(\tilde{v}_1^\alpha)^j=x_k\bar{v}$. For such $j$, we have
\begin{align*}
a_{11}^{\alpha\alpha}&=\int_{\Omega}(\mathbb{C}^0e(v_{1}^{\alpha}), e(v_{1}^{\alpha}))dx\geq \frac{1}{C}\int_{\Omega}|e(v_{1}^{\alpha})|^2dx\\
&\geq\frac{1}{C}\int_{\Omega}|e_{j d}(v_{1}^{\alpha})|^2dx\geq \frac{1}{C}\int_{\Omega_R}|\partial_{x_d}(v_{1}^{\alpha})^j|^2dx.
\end{align*}
Since $(v_{1}^{\alpha})^j|_{\partial{D}_{1}}=x_k\bar{v}|_{\partial{D}_{1}}=x_k, (v_{1}^{\alpha})^j|_{\partial{D}_2}
=x_k\bar{v}|_{\partial D_2}=0$, and $x_k\bar{v}(x',x_d)$ is linear in $x_d$ for fixed $x'$, so $x_k\bar{v}(x', \cdot)$ is harmonic, and hence it is a minimizer of the energy functional, 
\begin{align*}
\int_{h_2(x')}^{h_1(x')+{\varepsilon}}|\partial_{x_d}(v_{1}^{\alpha})^j|^2dx_d\geq
\int_{h_2(x')}^{h_1(x')+{\varepsilon}}|\partial_{x_d}(x_k\bar{v})|^2dx_d=\frac{x_k^2}{\varepsilon+h_1(x')-h_2(x')}.
\end{align*}
Integrating on $B'_R(0')$ for $x'$, we obtain
\begin{align*}
\int_{\Omega_R}|\partial_{x_d}(v_{1}^{\alpha})^j|^2dx&=\int_{B'_R}\int_{h_2(x')}^{h_1(x')+\varepsilon}|\partial_{x_d}(v_{1}^{\alpha})^j|^2dx_ddx'\\
&\geq\int_{{B'_{R}\setminus\Sigma'}}\int_{h_2(x')}^{h_1(x')+\varepsilon}|\partial_{x_d}(v_{1}^{\alpha})^j|^2dx_ddx'+\int_{\Sigma'}\int_0^\varepsilon|\partial_{x_d}(v_{1}^{\alpha})^j|^2dx_ddx'\\
&\geq\frac{1}{C}\int_{{B'_{R}\setminus\Sigma'}}\frac{x_k^2dx'}{\varepsilon+d^2(x')}+\int_{\Sigma'} \frac{x_k^2}{\varepsilon}dx'\\
&\geq \frac{1}{C}
\left(\frac{|\Sigma'|^{\frac{d+1}{d-1}}}{\varepsilon}+1\right).
\end{align*}
So, we have \eqref{gdd4}.

{\bf Step 2.} Estimate $a_{11}^{\alpha\beta}$ for $\alpha,\beta=1,\cdots,d$ with $\alpha\neq\beta$.

By the definition,
\begin{align*}
a_{11}^{\alpha\beta}&=a_{11}^{\beta\alpha}=-\int_{\partial{D}_{1}}\frac{\partial v_{1}^{\alpha}}{\partial \nu_0}\large\Big|_{+}\cdot\psi_\beta\\
&=-\int_{\partial{D}_{1}}\lambda(\nabla\cdot v_{1}^{\alpha})n_\beta+\mu\left((\nabla v_{1}^{\alpha}+(\nabla v_{1}^{\alpha})^T)\vec{n}\right)_\beta\\
&=-\int_{\partial{D}_{1}\cap B_R}\lambda\left(\sum_{k=1}^d\partial_{x_k} (v_1^{\alpha})^{k}\right)n_\beta
+\mu\sum_{l=1}^d\left(\partial_{x_\beta}(v_1^{\alpha})^{l}+\partial_{x_l}(v_1^{\alpha})^{\beta}\right)n_l+O(1)\\
&=:-\lambda\mathrm{I}_{\alpha\beta}-\mu\mathrm{II}_{\alpha\beta}+O(1),
\end{align*}
where
$$\mathrm{I}_{\alpha\beta}:=\int_{\partial{D}_{1}\cap B_R}\left(\sum_{k=1}^d\partial_{x_k} (v_1^{\alpha})^{k}\right)n_\beta,\quad\mathrm{II}_{\alpha\beta}:=\int_{\partial{D}_{1}\cap B_R} \sum_{l=1}^{d}\left(\partial_{x_\beta}(v_1^{\alpha})^{l}+\partial_{x_l}(v_1^{\alpha})^{\beta}\right)n_l,$$
and
$$\vec{n}=\frac{(\nabla_{x'}h_1(x'), -1)}{\sqrt{1+|\nabla_{x'}h_1(x')|^2}}.$$
Due to (\ref{h1h20}), for $k=1,\cdots,d-1$,
\begin{align}\label{dn} |n_{k}|=\left|\frac{\partial_{x_{k}}h_1(x')}{\sqrt{1+|\nabla_{x'}h_1(x')|^2}}\right|\leq Cd(x'),\quad~\mbox{and}~ |n_d|=\frac{ 1}{\sqrt{1+|\nabla_{x'}h_1(x')|^2}}\leq1.
\end{align}
Then, for $\alpha=1,\cdots,d$, $\beta=1,\cdots,d-1$, it follows from \eqref{v1-bounded1} and \eqref{dn} that
\begin{align*}\label{da-3}
|\mathrm{I}_{\alpha\beta}|\leq&\int_{\partial{D}_{1}\cap B_R}\left|\left(\sum_{k=1}^d\partial_{x_k} (v_1^{\alpha})^{k}\right)n_\beta\right|\nonumber\\
\leq&C\int_{B'_R\setminus\Sigma'}\frac{d(x')}{\varepsilon+d^2(x')}dx'\leq C\int_{{B'_R}\setminus\Sigma'}\frac{dx'}{\sqrt{\varepsilon+d^2(x')}}\\
\leq&C\int_{R_{0}}^{R}\frac{r^{d-2}}{\sqrt{\varepsilon+(r-R_{0})^{2}}}dr\leq C\int_{0}^{R-R_0}\frac{(t+R_0)^{d-2}}{\sqrt{\varepsilon+t^2}}dt\\
\leq&CR_0^{d-2}\int_{0}^{R-R_0}\frac{1}{\sqrt{\varepsilon+t^2}}dt+C\int_{0}^{R-R_0}\frac{{t^{d-2}}}{\sqrt{\varepsilon+t^2}}dt\\
\leq&C\left(|\Sigma'|^{\frac{d-2}{d-1}}|\log\varepsilon|+1\right).
\end{align*}
We devide $\mathrm{II}_{\alpha\beta}$ further into three parts,
\begin{align*}
\mathrm{II}_{\alpha\beta}=&\int_{\partial{D}_{1}\cap B_R} \sum_{l=1}^{d-1}\left(\partial_{x_\beta}(v_1^{\alpha})^{l}+\partial_{x_l}(v_1^{\alpha})^{\beta}\right)n_l+\int_{\partial{D}_{1}\cap B_R} \partial_{x_\beta}(v_1^{\alpha})^{d}n_d+\int_{\partial{D}_{1}\cap B_R} \partial_{x_d}(v_1^{\alpha})^{\beta} n_d\\
=&:\mathrm{II}_{\alpha\beta}^{1}+\mathrm{II}_{\alpha\beta}^{2}+\mathrm{II}_{\alpha\beta}^{3}.
\end{align*}
Then
\begin{align*}
|\mathrm{II}_{\alpha\beta}^{1}|&=\left|\int_{\partial{D}_{1}\cap B_R} \sum_{l=1}^{d-1}\left(\partial_{x_\beta}(v_1^{\alpha})^{l}+\partial_{x_l}(v_1^{\alpha})^{\beta}\right)n_l\right|\\
&\leq C\int_{B'_R\setminus\Sigma'}\frac{d(x')}{\varepsilon+d^2(x')}dx'\leq C\int_{B'_R\setminus\Sigma'}\frac{dx'}{\sqrt{\varepsilon+d^2(x')}}\\
&\leq C\left(|\Sigma'|^{\frac{d-2}{d-1}}|\log\varepsilon|+1\right).
\end{align*}
From \eqref{v1-x'} and \eqref{dn},
\begin{align*}
|\mathrm{II}_{\alpha\beta}^{2}|&=\left|\int_{\partial{D}_{1}\cap B_R} \partial_{x_\beta}(v_1^{\alpha})^{d}n_d\right|\leq\int_{B'_R}\frac{C}{\sqrt{\varepsilon+d^2(x')}}dx'\\
&\leq C\int_{\Sigma'}\frac{dx'}{\sqrt{\varepsilon}}+C\int_{B'_R\setminus\Sigma'}\frac{dx'}{\sqrt{\varepsilon+d^2(x')}}\\
&\leq C\left(\frac{|\Sigma'|}{\sqrt{\varepsilon}}+|\Sigma'|^{\frac{d-2}{d-1}}|\log\varepsilon|+1\right).
\end{align*}
By the definition of $\tilde{v}_1^\alpha$, using the fact that $(\tilde{v}_1^\alpha)^{\beta}=0$ if $\alpha\neq\beta$, and (\ref{dnabla_w_i0}),
\begin{align*}
|\mathrm{II}_{\alpha\beta}^{3}|&=\left|\int_{\partial{D}_{1}\cap B_R}\partial_{x_d} (v_1^{\alpha})^{\beta} n_d\right|\nonumber\\
&\leq\left|\int_{\partial{D}_{1}\cap B_R}\partial_{x_d}( (v_1^{\alpha})^{\beta}-(\tilde{v}_1^\alpha)^\beta)n_d\right|\\
&\leq C\int_{\Sigma'}\frac{dx'}{\sqrt{\varepsilon}}+C\int_{B'_R\setminus\Sigma'}\frac{dx'}{\sqrt{\varepsilon+d^2(x')}}\\
&\leq C\left(\frac{|\Sigma'|}{\sqrt{\varepsilon}}+|\Sigma'|^{\frac{d-2}{d-1}}|\log\varepsilon|+1\right).
\end{align*}
Hence
\begin{align*}
|\mathrm{II}_{\alpha\beta}|
\leq C\left(\frac{|\Sigma'|}{\sqrt{\varepsilon}}+|\Sigma'|^{\frac{d-2}{d-1}}|\log\varepsilon|+1\right).
\end{align*}
Therefore for $ \alpha,\beta=1,\cdots,d$ with $\alpha\neq\beta$, we have
\begin{align*}
|a_{11}^{\alpha\beta}|=|a_{11}^{\beta\alpha}|&\leq |\lambda||\mathrm{I}_{\alpha\beta}|+|\mu||\mathrm{II}_{\alpha\beta}|+C\leq C\left(\frac{|\Sigma'|}{\sqrt{\varepsilon}}+|\Sigma'|^{\frac{d-2}{d-1}}|\log\varepsilon|+1\right).
\end{align*}
Thus, \eqref{gdd3} is obtained.

{\bf Step 3.} Estimate of $a_{11}^{\alpha\beta}$ for $\alpha=1,\cdots,d$, $\beta=d+1,\cdots,\frac{d(d+1)}{2}$.

We take the case that $\alpha=1$, $\beta=d+1$ for instance. The other cases are similar. Since $\psi_{d+1}=(-x_{2},x_{1},0,\cdots,0)^{T}$, using the boundedness of $|\nabla{v}_1^{\alpha}|$ on $\partial{D}_{1}\setminus B_R$, we have
\begin{align*}
a_{11}^{1~d+1}=&-\int_{\partial{D}_{1}}\frac{\partial v_{1}^{1}}{\partial \nu_0}\large\Big|_{+}\cdot\psi_{d+1}\\
=&\int_{\partial{D}_{1}\cap B_R}\left(\lambda\big(\sum_{k=1}^d\partial_{x_k} (v_{1}^1)^{k}\big)n_1
+\mu\sum_{l=1}^d\left(\partial_{x_1}(v_{1}^1)^{l}+\partial_{x_l}(v_{1}^1)^{1}\right)n_l\right)x_{2}\\
&-\int_{\partial{D}_{1}\cap B_R}\left(\lambda\big(\sum_{k=1}^d\partial_{x_k} (v_{1}^1)^{k}\big)n_2
+\mu\sum_{l=1}^d\left(\partial_{x_2}(v_{1}^1)^{l}+\partial_{x_l}(v_{1}^1)^{2}\right)n_l\right)x_{1}+O(1)\\
=&:\mathrm{I}_{1~d+1}-\mathrm{II}_{1~d+1}+O(1),
\end{align*}
where
\begin{align*}
\mathrm{I}_{1~d+1}:=&\lambda\int_{\partial{D}_{1}\cap B_R}\big(\sum_{k=1}^d\partial_{x_k} (v_{1}^1)^{k}\big)n_1x_2+\mu\int_{\partial{D}_{1}\cap B_R}\sum_{l=1}^{d-1}\left(\partial_{x_1}(v_{1}^1)^{l}+\partial_{x_l}(v_{1}^1)^{1}\right)n_lx_{2}\\
&+\mu\int_{\partial{D}_{1}\cap B_R}\partial_{x_1}(v_{1}^1)^{d}n_dx_2+\mu\int_{\partial{D}_{1}\cap B_R}\partial_{x_d}(v_{1}^1)^{1}n_dx_2\\
=&:\,\lambda\mathrm{I}_{1~d+1}^{1}+\mu(\mathrm{I}_{1~d+1}^{2}+\mathrm{I}_{1~d+1}^{3}+\mathrm{I}_{1~d+1}^{4}).
\end{align*}
From \eqref{v1-bounded1} and \eqref{dn}, we have
\begin{align*}
|\mathrm{I}_{1~d+1}^1|&=\left|\int_{\partial{D}_{1}\cap B_R}\big(\sum_{k=1}^d\partial_{x_k} (v_{1}^1)^{k}\big)n_1x_2\right|\\
&\leq\int_{B'_R\backslash\Sigma'}\frac{Cd(x')|x_2|}{\varepsilon+d^2(x')}dx'\leq C\int_{B'_R\backslash\Sigma'}\frac{|x_2|}{\sqrt{\varepsilon+d^2(x')}}dx'.
\end{align*}
In view of the fact that $|x_2|\leq|x'|\leq d(x')+R_0$,
\begin{align*}
\int_{B'_R\backslash\Sigma'}\frac{|x_2|}{\sqrt{\varepsilon+d^2(x')}}dx'&\leq\int_{B'_R\backslash\Sigma'}\frac{d(x')+R_0}{\sqrt{\varepsilon+d^2(x')}}dx'\\
&\leq CR_0\int_{R_{0}}^{R}\frac{r^{d-2}}{\sqrt{\varepsilon+(r-R_{0})^{2}}}dr+C\\
&\leq C(|\Sigma'||\log\varepsilon|+1).
\end{align*}
Similarly, by using \eqref{v1-bounded1} and \eqref{dn}, we have
\begin{align*}
|\mathrm{I}_{1~d+1}^2|&=\left|\int_{\partial{D}_{1}\cap B_R}\sum_{l=1}^{d-1}\left(\partial_{x_1}(v_{1}^1)^{l}+\partial_{x_l}(v_{1}^1)^{1}\right)n_lx_{2}\right|\\
&\leq\int_{B'_R\backslash\Sigma'}\frac{Cd(x')|x_2|}{\varepsilon+d^2(x')}dx'\leq C(|\Sigma'||\log\varepsilon|+1).
\end{align*}
Making use of the fact $(\tilde{v}_1^1)^d=0$ and \eqref{dnabla_w_i0}, we obtain
\begin{align*}
|\mathrm{I}_{1~d+1}^{3}|&=\left|\int_{\partial{D}_{1}\cap B_R}\partial_{x_1}(v_{1}^1)^{d}n_dx_2\right|\\
&=\left|\int_{\partial{D}_{1}\cap B_R}\partial_{x_1}( v_{1}^{1}-\tilde{v}_1^1)^dn_dx_2\right|\nonumber\\
&\leq\int_{\Sigma'}\frac{C|x'|}{\sqrt{\varepsilon}}dx'+\int_{B'_R\setminus\Sigma'}\frac{C|x'|}{\sqrt{\varepsilon+d^2(x')}}dx'\\
&\leq C\left(\frac{|\Sigma'|^{\frac{d}{d-1}}}{\sqrt{\varepsilon}}+|\Sigma'||\log\varepsilon|+1\right).
\end{align*}
Similarly, we estimate $\mathrm{I}_{1~d+1}^{4}$.
\begin{align*}
|\mathrm{I}_{1~d+1}^{4}|&= \left|\int_{\partial{D}_{1}\cap B_R}\partial_{x_d}(v_{1}^1)^{1}n_dx_2\right|\\
&\leq\left|\int_{\partial{D}_{1}\cap B_R}\partial_{x_d} (\tilde{v}_1^1)^1n_dx_2\right|+\left|\int_{\partial{D}_{1}\cap B_R}\partial_{x_d}( v_{1}^{1}-\tilde{v}_1^1)^1n_dx_2\right|\nonumber\\
&=:\mathrm{I}_{1~d+1}^{4(1)}+\mathrm{I}_{1~d+1}^{4(2)}.
\end{align*}
It is easy to obtain that
\begin{align*}
|\mathrm{I}_{1~d+1}^{4(1)}|&=\left|\int_{\partial{D}_{1}\cap B_R}\partial_{x_d} (\tilde{v}_1^1)^1n_dx_2\right|\\
&=\left|\int_{\partial{D}_{1}\cap\overline{\Sigma}}\frac{x_2n_d}{\varepsilon}+\int_{\partial{D}_{1}\cap B_R\backslash\overline{\Sigma}}\frac{x_2n_d}{\varepsilon+d^2(x')}\right|\\
&\leq\int_{ B'_R\backslash\Sigma'}\frac{|x'|}{\varepsilon+d^2(x')}dx'
\leq \int_{ B'_R\backslash\Sigma'}\frac{R_0}{\varepsilon+d^2(x')}dx'+\int_{ B'_R\backslash\Sigma'}\frac{dx'}{\sqrt{\varepsilon+d^2(x')}}\\
&\leq C\left( \frac{|\Sigma'|}{\sqrt{\varepsilon}}+|\Sigma'|^{\frac{d-2}{d-1}}|\log\varepsilon|+1\right),
\end{align*}
and
\begin{align*}
|\mathrm{I}_{1~d+1}^{4(2)}|&=\left|\int_{\partial{D}_{1}\cap B_R}\partial_{x_d}( v_{1}^{1}-\tilde{v}_1^1)^1n_dx_2\right|\\
&\leq\int_{\Sigma'}\frac{C|x'|}{\sqrt{\varepsilon}}dx'+\int_{B'_R\setminus\Sigma'}\frac{C|x'|}{\sqrt{\varepsilon+d^2(x')}}dx'\leq C\left(\frac{|\Sigma'|^{\frac{d}{d-1}}}{\sqrt{\varepsilon}}+|\Sigma'||\log\varepsilon|+1\right).
\end{align*}
Therefore,
\begin{align*}
|\mathrm{I}_{1~d+1}^{4}|\leq C\left(\frac{|\Sigma'|^{\frac{d}{d-1}}}{\sqrt{\varepsilon}}+|\Sigma'||\log\varepsilon|+1\right).
\end{align*}
The estimate of $\mathrm{II}_{1~d+1}$ is the same as $\mathrm{I}_{1~d+1}$. Combining these estimates, we have
\begin{align*}
|a_{11}^{1~d+1}|\leq C\left(\frac{|\Sigma'|^{\frac{d}{d-1}}}{\sqrt{\varepsilon}}+|\Sigma'||\log\varepsilon|+1\right).
\end{align*}
So we have \eqref{gd1d4}.

{\bf Step 4.} Estimates of $a_{11}^{\alpha\beta}$, $\alpha, \beta=d+1,\cdots, \frac{d(d+1)}{2}$ with $\alpha\neq\beta$.

We only estimate $a_{11}^{d+1~d+2}$ with
$\tilde{v}_1^{d+1}=(-\bar{v}x_2, \bar{v}x_1, 0,\cdots,0)^T$ and $\psi_{d+2}=(-x_{3},0,x_{1}, 0, \cdots, 0)^{T}$ for instance. The other cases are the same.
\begin{align*}
a_{11}^{d+1~d+2}=&-\int_{\partial{D}_{1}}\frac{\partial v_1^{d+1}}{\partial \nu_0}\large\Big|_{+}\cdot\psi_{d+2}\\
=&\int_{\partial{D}_{1}\cap B_R}\left(\lambda\big(\sum_{k=1}^d\partial_{x_k} (v_1^{d+1})^k\big)n_1
+\mu\sum_{l=1}^d\left(\partial_{x_1}(v_1^{d+1})^l+\partial_{x_l}(v_1^{d+1})^1\right)n_l\right)x_{3}\\
&-\int_{\partial{D}_{1}\cap B_R}\left(\lambda\big(\sum_{k=1}^d\partial_{x_k} (v_1^{d+1})^k\big)n_3
+\mu\sum_{l=1}^d\left(\partial_{x_3}(v_1^{d+1})^l+\partial_{x_l}(v_1^{d+1})^3\right)n_l\right)x_{1}+O(1)\\
=&:\mathrm{I}_{d+1~d+2}-\mathrm{II}_{d+1~d+2}+O(1).
\end{align*}

We divide it into two cases: (i) for $d=3$, (ii) for $d\geq4$.

(i) For $d=3$, in view of that $|x_3|=|\varepsilon+h_1(x')|\leq C(\varepsilon+d^2(x'))$,  it is easy to see that
\begin{align*}
|\mathrm{I}_{45}|\leq C\int_{B'_R}|\nabla v_1^4||x_3|dx'\leq C.
\end{align*}
Further divide
\begin{align*}
\mathrm{II}_{45}:=&\lambda\int_{\partial{D}_{1}\cap B_R}\sum_{k=1}^{2}\partial_{x_k} (v_1^4)^kn_3x_1+\mu\int_{\partial{D}_{1}\cap B_R}\sum_{l=1}^{2}(\partial_{x_3}(v_1^4)^l+\partial_{x_l}(v_1^4)^3)n_lx_1\\
&+(\lambda+2\mu)\int_{\partial{D}_{1}\cap B_R}\partial_{x_3}(v_1^4)^3n_3x_{1}\\
=&:\,\lambda\mathrm{II}_{45}^{1}+\mu\mathrm{II}_{45}^{2}+(\lambda+2\mu)\mathrm{II}_{45}^{3}.
\end{align*}
By \eqref{e2.4}, we know that
 $$\partial_{x_1}(\tilde{v}_1^4)^1=\partial_{x_1}\bar{v}\cdot (-x_2)=0, \qquad\quad~~~\partial_{x_2}(\tilde{v}_1^4)^2=\partial_{x_2}\bar{v}\cdot x_1=0, \quad\quad~~~~x'\in \Sigma'.$$
Together with \eqref{dnabla_w_i00}, we have
\begin{align*}
|\mathrm{II}_{45}^1|&=\left|\int_{\partial{D}_{1}\cap B_R}\sum_{k=1}^{2}\partial_{x_k} (v_1^4)^kn_3x_1\right|\\
&\leq\left|\int_{\partial{D}_{1}\cap B_R\backslash\Sigma}\left(\sum_{k=1}^{2}\partial_{x_k} (\tilde{v}_1^4)^k\right)n_3x_1\right|+\left|\int_{\partial{D}_{1}\cap B_R}\left(\sum_{k=1}^{2}\partial_{x_k} (v_1^4-\tilde{v}_1^4)^k\right)n_3x_1\right|\\
&\leq\int_{B'_R\backslash\Sigma'}\frac{Cd(x')|x'|^2}{\varepsilon+d^2(x')}dx'+\int_{ B'_R}\frac{C|x'|^2}{\sqrt{\varepsilon+d^2(x')}}dx'\\
&\leq\int_{\Sigma'}\frac{C|x'|^2}{\sqrt{\varepsilon}}dx'+\int_{B'_R\setminus\Sigma'}\frac{C|x'|^2}{\sqrt{\varepsilon+d^2(x')}}dx'\\
&\leq \frac{C|\Sigma'|^{2}}{\sqrt{\varepsilon}}+C|\Sigma'|^{\frac{3}{2}}|\log\varepsilon|+C.
\end{align*}
By \eqref{dv1---bounded1} and \eqref{dn}, we have
\begin{align*}
|\mathrm{II}_{45}^2|&=\left|\int_{\partial{D}_{1}\cap B_R}\sum_{l=1}^{2}(\partial_{x_3}(v_1^4)^l+\partial_{x_l}(v_1^4)^3)n_lx_1\right|\\
&\leq\int_{B'_R\backslash\Sigma'}\frac{Cd(x')|x'|^2}{\varepsilon+d^2(x')}dx'\leq C(|\Sigma'|^{\frac{3}{2}}|\log\varepsilon|+1).
\end{align*}
Since $\partial_{x_3}(\tilde{v}_1^4)^3=0$,
\begin{align*}
|\mathrm{II}_{45}^3|&=\left|\int_{\partial{D}_{1}\cap B_R}\partial_{x_3}(v_1^4)^3n_3x_{1}\right|\\
&=\left|\int_{\partial{D}_{1}\cap B_R}\partial_{x_3}(v_1^4-\tilde{v}_1^4)^3n_3x_{1}\right|\\
&\leq\int_{ \Sigma'}\frac{C|x'|^2}{\sqrt{\varepsilon}}dx'+\int_{B'_R\backslash\Sigma'}\frac{C|x'|^2}{\sqrt{\varepsilon+d^2(x')}}dx'\\
&\leq \frac{C|\Sigma'|^{2}}{\sqrt{\varepsilon}}+C|\Sigma'|^{\frac{3}{2}}|\log\varepsilon|+C.
\end{align*}
Combining these estimates, thus
\begin{align*}
|a_{11}^{45}|\leq \frac{C|\Sigma'|^{2}}{\sqrt{\varepsilon}}+C|\Sigma'|^{\frac{3}{2}}|\log\varepsilon|+C.
\end{align*}

(ii) For $d\geq4$, denote
\begin{align*}
\mathrm{I}_{d+1~d+2}:=&\lambda\int_{\partial{D}_{1}\cap B_R}\sum_{k=1}^{d}\partial_{x_k} (v_1^{d+1})^kn_2
x_{3}+\mu\int_{\partial{D}_{1}\cap B_R}\sum_{l=1}^{d-1}\left(\partial_{x_1}(v_1^{d+1})^l+\partial_{x_l}(v_1^{d+1})^1\right)n_lx_{3}\\
&+\mu\int_{\partial{D}_{1}\cap B_R}\partial_{x_1}(v_1^{d+1})^dn_dx_{3}+\mu\int_{\partial{D}_{1}\cap B_R}\partial_{x_d}(v_1^{d+1})^1n_dx_{3}\\
=:&\lambda\mathrm{I}_{d+1~d+2}^1+\mu(\mathrm{I}_{d+1~d+2}^2+\mathrm{I}_{d+1~d+2}^3+\mathrm{I}_{d+1~d+2}^4).
\end{align*}
By (\ref{dn}),
\begin{align*}
|\mathrm{I}_{d+1~d+2}^1|&=\left|\int_{\partial{D}_{1}\cap B_R}\sum_{k=1}^{d}\partial_{x_k} (v_1^{d+1})^kn_2
x_{3}\right|\\
&\leq\int_{B_R'\setminus\Sigma'}\frac{Cd(x')|x'|^2}{\varepsilon+d^2(x')}dx'\leq C(|\Sigma'|^{\frac{d}{d-1}}|\log\varepsilon|+1).
\end{align*}
By \eqref{dv1---bounded1} and \eqref{dn},
\begin{align*}
|\mathrm{I}_{d+1~d+2}^2|&=\left|\int_{\partial{D}_{1}\cap B_R}\sum_{l=1}^{d-1}\left(\partial_{x_1}(v_1^{d+1})^l+\partial_{x_l}(v_1^{d+1})^1\right)n_lx_{3}\right|\\
&\leq\int_{B'_R\backslash\Sigma'}\frac{Cd(x')|x'|^2}{\varepsilon+d^2(x')}dx'\leq C(|\Sigma'|^{\frac{d}{d-1}}|\log\varepsilon|+1).
\end{align*}
Since $(\tilde{v}_1^{d+1})^d=0$, it follows that
\begin{align*}
|\mathrm{I}_{d+1~d+2}^3|&\leq\left|\int_{\partial{D}_{1}\cap B_R} \partial_{x_1}(v_1^{d+1}-\tilde{v}_1^{d+1})^dn_dx_{3}\right|\\
&\leq\int_{B_R'}\frac{C|x'|^2}{\sqrt{\varepsilon+d^2(x')}}dx'\leq \frac{C|\Sigma'|^{\frac{d+1}{d-1}}}{\sqrt{\varepsilon}}+C|\Sigma'|^{\frac{d}{d-1}}|\log\varepsilon|+C.
\end{align*}
Similarly, using the symmetry,
\begin{align*}
|\mathrm{I}_{d+1~d+2}^4|&=\left|\int_{\partial{D}_{1}\cap B_R}\partial_{x_d}(v_1^{d+1})^1n_dx_{3}\right|\\
&\leq\left|\int_{\partial{D}_{1}\cap B_R}\partial_{x_d}(\tilde{v}_1^{d+1})^1n_dx_{3}\right|+\left|\int_{\partial{D}_{1}\cap B_R}\partial_{x_d}(v_1^{d+1}-\tilde{v}_1^{d+1})^1n_dx_{3}\right|\\
&\leq\int_{B_R'\setminus\Sigma'}\frac{C|x'|^2}{\varepsilon+d^2(x')}dx'+\int_{B_R'}\frac{C|x'|^2}{\sqrt{\varepsilon+d^2(x')}}dx'\\
&\leq \frac{C|\Sigma'|^{\frac{d+1}{d-1}}}{\sqrt{\varepsilon}}+C|\Sigma'|^{\frac{d}{d-1}}|\log\varepsilon|+C.
\end{align*}
Therefore,
\begin{align*}
|\mathrm{I}_{d+1~d+2}|\leq \frac{C|\Sigma'|^{\frac{d+1}{d-1}}}{\sqrt{\varepsilon}}+C|\Sigma'|^{\frac{d}{d-1}}|\log\varepsilon|+C.
\end{align*}
$\mathrm{II}_{d+1~d+2}$ is similar to $\mathrm{I}_{d+1~d+2}$. Then we have
\begin{align*}
|a_{11}^{d+1~d+2}|\leq \frac{C|\Sigma'|^{\frac{d+1}{d-1}}}{\sqrt{\varepsilon}}+C|\Sigma'|^{\frac{d}{d-1}}|\log\varepsilon|+C.
\end{align*}
Thus, \eqref{gd4d5} is obtained. 

The proof of Lemma \ref{glemma4.1} is completed.
\end{proof}

\begin{remark}
We would like to remark that our method can be used to deal with more general cases, say $\Sigma'=\{x'\in\mathbb{R}^{d-1}~|~\sum_{i=1}^{d-1}(\frac{|x_{i}|}{a_{i}})^{m}\leq1, ~a_{i}>0,~m\geq2\}$. In the following, we will present the idea below in the contex where the flat domain $\Sigma'$ is an ellipse in dimension $d=3$, and show the estimates in Lemma \ref{glemma4.1} still hold, and so Theorem \ref{thmd}.
\end{remark}
\begin{proof}
 We suppose that $\Sigma'$ is symmetric about $x_1$, $x_2$, respectively. Moreover, we assume that for $x_0'\in\partial\Sigma'$, $$(h_1-h_2)(x_0')=0,\quad\nabla_{x'}(h_1-h_2)(x_0')=0,$$  and
$$(h_1-h_2)(x')=\frac{1}{2}(x'-x_0')^T(\nabla^2_{x'}(h_1-h_2)(x_0'))(x'-x_0')+o((x'-x_0')^2).$$

We only estimate  $a_{11}^{56}$ and $a_{11}^{14}$.
\begin{align*}
a_{11}^{56}=&-\int_{\partial{D}_{1}}\frac{\partial v_1^5}{\partial \nu_0}\large\Big|_{+}\cdot\psi_{6}\\
=&-\int_{\partial{D}_{1}}\Big(\lambda(\nabla\cdot v_1^5)\vec{n}+\mu(\nabla v_1^5+(\nabla v_1^5)^T)\vec{n}\Big)\cdot(0,-x_{3},x_{2})^{T}\\
=&\int_{\partial{D}_{1}\cap B_R}\left(\lambda\big(\sum_{k=1}^3\partial_{x_k} (v_1^5)^k\big)n_2
+\mu\sum_{l=1}^3\left(\partial_{x_2}(v_1^5)^l+\partial_{x_l}(v_1^5)^2\right)n_l\right)x_{3}\\
&-\int_{\partial{D}_{1}\cap B_R}\left(\lambda\big(\sum_{k=1}^3\partial_{x_k} (v_1^5)^k\big)n_3
+\mu\sum_{l=1}^3\left(\partial_{x_3}(v_1^5)^l+\partial_{x_l}(v_1^5)^3\right)n_l\right)x_{2}+O(1)\\
=&:\mathrm{I}_{56}-\mathrm{II}_{56}+O(1).
\end{align*}
In view of that $|x_3|\leq C(\varepsilon+d^2(x'))$ for $x'\in B'_R$, we have
\begin{align*}
|\mathrm{I}_{56}|\leq C\int_{B'_R}|\nabla v_1^5||x_3|dx'\leq C.
\end{align*}
Denote
\begin{align*}
\mathrm{II}_{56}:=&\,\lambda\int_{\partial{D}_{1}\cap B_R}\left(\sum_{k=1}^{2}\partial_{x_k} (v_1^5)^k\right)n_3x_2+\mu\int_{\partial{D}_{1}\cap B_R}\sum_{l=1}^{2}(\partial_{x_3} (v_1^5)^l+\partial_{x_l} (v_1^5)^3)n_lx_2\\
&\qquad+(\lambda+2\mu)\int_{\partial{D}_{1}\cap B_R}\partial_{x_3}(v_1^5)^3n_3x_2\\
=&:\lambda\mathrm{II}_{56}^1+\mu\mathrm{II}_{56}^2+(\lambda+2\mu)\mathrm{II}_{56}^3.
\end{align*}
The estimates of $\mathrm{II}_{56}^1$ and $\mathrm{II}_{56}^2$ are similar to $\mathrm{II}_{45}^1$ and $\mathrm{II}_{45}^2$, respectively, we only need to estimate $\mathrm{II}_{56}^3$.
\begin{align*}
|\mathrm{II}_{56}^3|&=\left|\int_{\partial{D}_{1}\cap B_R}\partial_{x_3}(v_1^5)^3n_3x_2\right|\\
&\leq\left|\int_{\partial{D}_{1}\cap\overline{\Sigma}}\partial_{x_3} (\tilde{v}_1^5)^3n_3x_2\right|+\left|\int_{\partial{D}_{1}\cap B_R\backslash\Sigma}\partial_{x_3} (\tilde{v}_1^5)^3n_3x_2\right|+\left|\int_{\partial{D}_{1}\cap B_R}\partial_{x_3} (v_1^5-\tilde{v}_1^5)^3n_3x_2\right|\\
&\leq \left|\int_{\Sigma'}\frac{Cx_1x_2dx'}{\varepsilon}\right|+\left|\int_{ B_R'\backslash\Sigma'}\frac{Cx_1x_2dx'}{\varepsilon+h_1(x')-h_2(x')}\right|+\left|\int_{B'_R}\frac{C|x'|^2dx'}{\sqrt{\varepsilon+d^2(x')}}\right|\\
&=:|\mathrm{II}_{56}^{3(1)}|+|\mathrm{II}_{56}^{3(2)}|+|\mathrm{II}_{56}^{3(3)}|.
\end{align*}
Obviously,
\begin{align*}
|\mathrm{II}_{56}^{3(1)}|=\left|\int_{\Sigma'}\frac{Cx_1x_2dx'}{\varepsilon}\right|=0,
\end{align*}
and
\begin{align*}
|\mathrm{II}_{56}^{3(3)}|&\leq\int_{B'_R}\frac{C(R_0^2+d^2(x'))}{\sqrt{\varepsilon+d^2(x')}}dx'\leq\frac{C|\Sigma'|^2}{\sqrt{\varepsilon}}+C|\Sigma'|^{\frac{3}{2}}|\log\varepsilon|+C.
\end{align*}

\begin{figure}[t]
\begin{minipage}[c]{0.9\linewidth}
\centering
\includegraphics[width=2.5in]{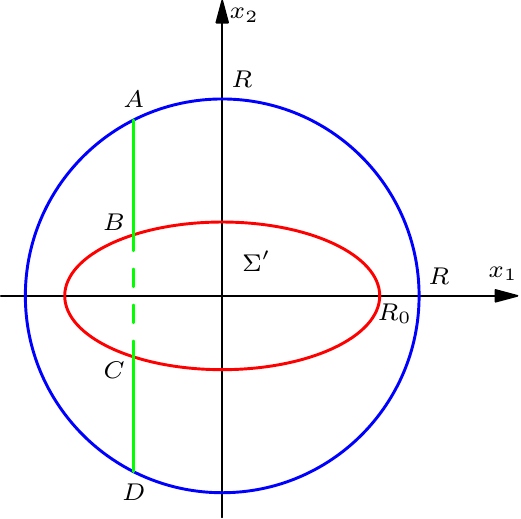}
\caption{\small $\Sigma'$ is an ellipse in dimension two.}
\label{fig2}
\end{minipage}
\end{figure}
Now we estimate $\mathrm{II}_{56}^{3(2)}$. Since  $\Sigma'$ is symmetric about $x_1$, $x_2$ respectively, we assume that the corresponding partial boundaries of $\partial\Sigma'$ is $g_1(x_1)$ and $g_2(x_1)$. Obviously, $g_2(x_1)=-g_1(x_1)$. For fixed $x_1$, the points of intersection to $\Sigma'$ and $B_R$ are $A,B,C,D$, see Figure \ref{fig2}, then we have
\begin{align*}
\mathrm{II}_{56}^{3(2)}=&\int_{-R_0}^{R_0}x_1dx_1\int_{\overline{AB}\cup\overline{CD}}\frac{x_2dx_2}{\varepsilon+(h_1-h_2)(x')}
+\int_{-R}^{-R_0}x_1dx_1\int_{-\sqrt{R^2-x_1^2}}^{\sqrt{R^2-x_1^2}}\frac{x_2dx_2}{\varepsilon+(h_1-h_2)(x')}\\
&\quad+\int_{R_0}^Rx_1dx_1\int_{-\sqrt{R^2-x_1^2}}^{\sqrt{R^2-x_1^2}}\frac{x_2dx_2}{\varepsilon+(h_1-h_2)(x')}.
\end{align*}
By Taylor expansion, for $x'\in\overline{AB}$,
 $$(h_1-h_2)(x_1,x_2)=h_{B}(x_1)(x_2-g_1(x_1))^2+o((x_2-g_1(x_1))^2),\quad h_{B}(x_1):=\frac{1}{2}\partial_{x_2x_2}(h_1-h_2)(B),$$
and for $x'\in\overline{CD}$,
 $$(h_1-h_2)(x_1,x_2)=h_{C}(x_1)(x_2-g_2(x_1))^2+o((x_2-g_2(x_1))^2),\quad h_{C}(x_1):=\frac{1}{2}\partial_{x_2x_2}(h_1-h_2)(C).$$
Then
\begin{align*}
&\quad\int_{\overline{AB}\cup\overline{CD}}\frac{x_2}{\varepsilon+(h_1-h_2)(x')}dx_2\\
&=\int_B^A\frac{x_2dx_2}{\varepsilon+h_{B}(x_1)(x_2-g_1(x_1))^2}+\int_D^C\frac{x_2dx_2}{\varepsilon+h_{C}(x_1)(x_2-g_2(x_1))^2}\\
&\quad +\int_B^Ax_2\left[\frac{1}{\varepsilon+(h_1-h_2)(x')}-\frac{1}{\varepsilon+h_{B}(x_1)(x_2-g_1(x_1))^2}\right]dx_2\\
&\quad +\int_D^Cx_2\left[\frac{1}{\varepsilon+(h_1-h_2)(x')}-\frac{1}{\varepsilon+h_{C}(x_1)(x_2-g_2(x_1))^2}\right]dx_2\\
&=:\mathrm{II}_{56}^{3(21)}+\mathrm{II}_{56}^{3(22)}+\mathrm{II}_{56}^{3(23)}+\mathrm{II}_{56}^{3(24)}.
\end{align*}
We only deal with the integral on the segment $\overline{AB}$, that on $\overline{CD}$ is the same.
\begin{align*}
\mathrm{II}_{56}^{3(21)}=&\int_{g_1(x_1)}^{\sqrt{R^2-x_1^2}}\frac{x_2}{\varepsilon+h_{B}(x_1)(x_2-g_1(x_1))^2}dx_2\\
=&\int_{g_1(x_1)}^{\sqrt{R^2-x_1^2}}\frac{(x_2-g_1(x_1))+g_1(x_1)}{\varepsilon+h_{B}(x_1)(x_2-g_1(x_1))^2}dx_2.
\end{align*}
After direct calculation, we can obtain
\begin{align*}
\mathrm{II}_{56}^{3(21)}=&\,\frac{1}{2h_{B}(x_1)}\left[\ln\big(\varepsilon+h_{B}(x_1)(\sqrt{R^2-x_1^2}-g_1(x_1))^2\big)-\log\varepsilon\right]\\
&+\frac{g_1(x_1)}{\sqrt{h_{B}(x_1)\varepsilon}}\arctan\frac{\sqrt{h_{B}(x_1)}(\sqrt{R^2-x_1^2}-g_1(x_1))}{\sqrt{\varepsilon}}.
\end{align*}
It is easy to see that $$\frac{1}{C}\leq h_{B}(x_1), h_{C}(x_1)\leq C.$$
So
\begin{align*}
|\mathrm{II}_{56}^{3(21)}+\mathrm{II}_{56}^{3(22)}|\leq \frac{Cg_1(x_1)}{\sqrt{\varepsilon}}+C|\log\varepsilon|.
\end{align*}
Note that $|g_1(x_1)|\leq C|x_1|$, we have
\begin{align*}
\left|\int_{-R_0}^{R_0}x_1(\mathrm{II}_{56}^{3(21)}+\mathrm{II}_{56}^{3(22)})dx_1\right|&\leq C\left(\int_{-R_0}^{R_0}\frac{|x_1|^2}{\sqrt{\varepsilon}}dx_1+\int_{-R_0}^{R_0}|x_1||\log\varepsilon|dx_1\right)\\
&\leq \frac{C|\Sigma'|^{\frac{3}{2}}}{\sqrt{\varepsilon}}+C|\Sigma'||\log\varepsilon|.
\end{align*}

We remark that if $\Sigma'=\{0'\}$, then it is easy to see that
$$\mathrm{II}_{56}^{3(21)}+\mathrm{II}_{56}^{3(22)}=0.$$

As for the terms $\mathrm{II}_{56}^{3(23)}$ and $\mathrm{II}_{56}^{3(24)}$,   similar to estimate $\mathrm{II}_{23}^{2}$ in Step 5 of the proof of Lemma \ref{lem_a11},  we have
\begin{align*}
|\mathrm{II}_{56}^{3(23)}|&=\left|\int^{\sqrt{R^2-|x_1|^2}}_{g_1(x_1)}x_2\left[\frac{1}{\varepsilon+(h_1-h_2)(x')}-
\frac{1}{\varepsilon+h_{B}(x_1)(x_2-g_1(x_1))^2}\right]dx_2\right|\\
&=\left|\int^{\sqrt{R^2-|x_1|^2}}_{g_1(x_1)}\frac{x_2 o((x_2-g_1(x_1))^2)}{(\varepsilon+(h_1-h_2)(x'))(\varepsilon+h_{B}(x_1)(x_2-g_1(x_1))^2)}dx_2\right|\\
&\leq C\int^{\sqrt{R^2-|x_1|^2}}_{g_1(x_1)}\frac{ o((x_2-g_1(x_1))^3)+|g_1(x_1)|o((x_2-g_1(x_1))^2)}{(\varepsilon+(x_2-g_1(x_1))^2)^2}dx_2\\
&\leq C\left[|g_1(x_1)|o\left(\frac{1}{\sqrt{\varepsilon}}\right)+o\left(|\log\varepsilon|\right)\right],
\end{align*}
which, together with $|g_1(x_1)|\leq C|x_1|$, implies that
\begin{align*}
\int_{\Sigma'}x_1\mathrm{II}_{56}^{3(23)}dx_1
= o\left(\frac{1}{\sqrt{\varepsilon}}\right).
\end{align*}
By the same technique, we can obtain that
\begin{align*}
\int_{-\sqrt{R^2-x_1^2}}^{\sqrt{R^2-x_1^2}}\frac{x_2dx_2}{\varepsilon+(h_1-h_2)(x')}= o\left(|\log\varepsilon|\right).
\end{align*}
Therefore,
\begin{align*}
\int_{R_0}^{R}x_1dx_1\int_{-\sqrt{R^2-x_1^2}}^{\sqrt{R^2-x_1^2}}\frac{x_2dx_2}{\varepsilon+(h_1-h_2)(x')}= o\left(|\log\varepsilon|\right).
\end{align*}
Combining these estimates above, we have
\begin{align*}
|\mathrm{II}_{56}^{3(2)}|
\leq\begin{cases}\frac{C|\Sigma'|^{\frac{3}{2}}}{\sqrt{\varepsilon}}+C|\Sigma'||\log\varepsilon|+C,&|\Sigma'|\neq 0,\\
C,&|\Sigma'|= 0.
\end{cases}
\end{align*}
Therefore, we have
\begin{align*}
|a_{11}^{56}|\leq  \frac{C|\Sigma'|^2}{\sqrt{\varepsilon}}+C|\Sigma'|^{\frac{3}{2}}|\log\varepsilon|+C.
\end{align*}
Thus, we have the same estimate \eqref{gd4d5} for $a_{11}^{56}$.

As for $a_{11}^{14}$, we only need to estimate $\mathrm{I}_{14}^{4(1)}$ as in Step 3 of the proof of Lemma \ref{glemma4.1} for $d=3$,
\begin{align*}
|\mathrm{I}_{14}^{4(1)}|&\leq\left|\int_{\partial{D}_{1}\cap\overline{\Sigma}}\frac{Cx_2n_3}{\varepsilon}+\int_{\partial{D}_{1}\cap B_R\backslash\Sigma}\frac{Cx_2}{\varepsilon+h_1(x')-h_2(x')}n_3\right|\\
&\leq\left|\int_{B'_R\backslash\Sigma'}\frac{Cx_2}{\varepsilon+h_1(x')-h_2(x')}dx'\right|.
\end{align*}
Similar to  $\mathrm{II}_{56}^{3(2)}$, we have
\begin{align*}
|\mathrm{I}_{14}^{4(1)}|=&\left|\int_{-R_0}^{R_0}dx_1\int_{\overline{AB}\cup\overline{CD}}\frac{x_2dx_2}{\varepsilon+(h_1-h_2)(x')}
+\int_{-R}^{-R_0}dx_1\int_{-\sqrt{R^2-x_1^2}}^{\sqrt{R^2-x_1^2}}\frac{x_2dx_2}{\varepsilon+(h_1-h_2)(x')}\right.\\
&+\left.\int_{R_0}^Rdx_1\int_{-\sqrt{R^2-x_1^2}}^{\sqrt{R^2-x_1^2}}\frac{x_2dx_2}{\varepsilon+(h_1-h_2)(x')}\right|.
\end{align*}
Using the same procedure to estimate $a_{11}^{56}$, we have
\begin{align*}
\left|\int_{-R_0}^{R_0}dx_1\int_{\overline{AB}\cup\overline{CD}}\frac{x_2}{\varepsilon+(h_1-h_2)(x')}dx_2\right|\leq \frac{C|\Sigma'|}{\sqrt{\varepsilon}}+C|\Sigma'|^{\frac{1}{2}}|\log\varepsilon|+C.
\end{align*}
Similarly, we can obtain that
\begin{align*}
\int_{R_0}^{R}dx_1\int_{-\sqrt{R^2-x_1^2}}^{\sqrt{R^2-x_1^2}}\frac{x_2dx_2}{\varepsilon+(h_1-h_2)(x')}= o\left(|\log\varepsilon|\right).
\end{align*}
Combining these estimates above, we have
\begin{align*}
|\mathrm{I}_{14}^{4(1)}|
\leq\begin{cases}\frac{C|\Sigma'|^{\frac{3}{2}}}{\sqrt{\varepsilon}}+C|\Sigma'||\log\varepsilon|+C,&|\Sigma'|\neq 0,\\
C,&|\Sigma'|= 0.
\end{cases}
\end{align*}
Thus, we obtain the same upper bounds of $a_{11}^{14}$ as in Lemma \ref{glemma4.1}.
\end{proof}

\section{Two inclusions with relative convexity of order $m$}\label{sec_m}

In this section, we assume that the relative convexity between $D_{1}$ and $D_{2}$ is of order $m$, $m\geq2$, and reveal the relationship between the blow-up rate of $|\nabla u|$ and the order $m$ of the relative convexity between $D_{1}$ and $D_{2}$ to explore the blow-up mechanism of the stress in composite materials.

For $d\geq 2$, under the assumptions of Theorem \ref{thm3}, we also denote
\begin{equation*}
\delta(x'):=\varepsilon+h_{1}(x')-h_{2}(x'),\qquad\forall~(x',x_{d})\in\Omega_{R}.
\end{equation*}
Clearly, at this moment,
\begin{equation}
\delta(x')=\varepsilon+\kappa_{0}|x'|^m.
\end{equation}
Similar to Theorem \ref{thm2.1}, we have

\begin{theorem}\label{thm6.1}
Under the assumptions in Theorem \ref{thm3}, let $v\in H^{1}(\Omega; \mathbb{R}^{d})$ be a weak solution of problem \eqref{eq1.1}. Then for sufficiently small  $0<\varepsilon<1/2$,
\begin{equation}\label{mmainestimate}
|\nabla v(x',x_{d})|
\leq\, \frac{C}{\varepsilon+|x'|^m}\Big|\psi(x',\varepsilon+h_{1}(x'))\Big|+C\|\psi\|_{C^{2}(\partial{D}_{1})},\quad\forall\,x\in \Omega_{R},
\end{equation}
and
$$|\nabla v(x)|
\leq\,C\|\psi\|_{C^{2}(\partial{D}_{1})}\quad\forall\,x\in \Omega\setminus\Omega_{R}.
$$
\end{theorem}

The idea to prove Theorem \ref{thm6.1} is the same as that for Theorem \ref{thm2.1} with the slight modifications necessary. We list the main differences in the Appendix.

We define $\bar{v}$ by \eqref{vvd} as before.
By a direct calculation, we obtain that for $k=1,\cdots, d-1$, and $x\in\Omega_{R}$,
\begin{equation}
|\partial_{x_{k}}\bar{v}(x)|\leq\frac{C|x'|^{m-1}}{\varepsilon+|x'|^m},\qquad |\partial_{x_{d}}\bar{v}(x)|\leq\frac{C}{\varepsilon+|x'|^m},\label{me2.4}
\end{equation}
Define $\tilde{v}_1^\alpha:=\bar{v}\psi_{\alpha}$ as before. For problem (\ref{equ_v1}), taking the boundary data $\psi=\psi_{\alpha}$ and applying Theorem \ref{thm6.1}, we have

\begin{corollary}\label{mcorol3.4}
Assume the above, let $v_{1}^{\alpha}, v_{2}^\alpha, v_{3}\in{H}^1(\Omega; \mathbb{R}^{d})$ be the
weak solutions of \eqref{equ_v1} and \eqref{equ_v3}, respectively. Then for sufficiently small  $0<\varepsilon<1/2$ and $i=1,2$, we have
\begin{equation}\label{mnabla_w_i0}
|\nabla(v_{i}^{\alpha}-\tilde{v}_i^\alpha)(x',x_{d})|\leq\,\frac{C}{\sqrt{\varepsilon+|x'|^m}},\quad  \alpha=1,\cdots,d,\quad\mbox{for}~x\in\Omega_{R},
\end{equation}
\begin{equation}\label{mnabla_w_i00}
|\nabla(v_{i}^{\alpha}-\tilde{v}_i^\alpha)(x',x_{d})|\leq\,\frac{C|x'|}{\sqrt{\varepsilon+|x'|^m}},\quad  \alpha=d+1,\cdots,\frac{d(d+1)}{2},\quad\mbox{for}~x\in\Omega_{R}.
\end{equation}
Consequently, 
\begin{equation}\label{mv1-bounded1}
|\nabla v_{i}^\alpha(x)|\leq\,\frac{C}{\varepsilon+|x'|^m},\qquad\  \alpha=1,\cdots,d,~x\in\Omega_{R};
\end{equation}
\begin{equation}\label{mv1-x'}
|\nabla_{x'}v_{i}^\alpha(x)|\leq\,\frac{C}{\sqrt{\varepsilon+|x'|^m}},\qquad\ \alpha=1,\cdots,d,~x\in\Omega_{R};
\end{equation}
\begin{equation}\label{mv1---bounded1}
|\nabla v_i^{\alpha}(x)|\leq\frac{C(\varepsilon+|x'|)}{\varepsilon+|x'|^m},\qquad\ \alpha=d+1,\cdots,\frac{d(d+1)}{2},~x\in \Omega_{R};
\end{equation}
\begin{equation}\label{mv1--bounded1}
|\nabla v_{i}^\alpha(x)|\leq\,C,\qquad\quad\ \alpha=1,2,\cdots,\frac{d(d+1)}{2},~x\in\Omega\setminus\Omega_{R}; 
\end{equation}
and
$$|\nabla(v_{1}^{\alpha}+v_{2}^\alpha)(x)|\leq\,C\|\varphi\|_{C^{2}(\partial D)},\quad\,\alpha=1,2,\cdots,\frac{d(d+1)}{2},\quad~x\in\Omega;
$$$$|\nabla{v}_{3}(x)|\leq C\|\varphi\|_{C^{2}(\partial D)},\quad\,x\in\Omega.
$$
\end{corollary}

Using these estimates for $|\nabla v_{i}^{\alpha}|$, we have the following key estimates for $a_{11}^{\alpha\beta}$.

\begin{lemma}\label{lemma6.1} For $d\geq2$, we have
\begin{align}
\frac{1}{C}\rho_{d-1,m}^{1}(\varepsilon)\leq a_{11}^{\alpha\alpha}&\leq C\rho_{d-1,m}^1(\varepsilon),\quad \alpha=1, \cdots,\ d;\quad\quad\quad  \label{mlem34-1'}\\
\frac{1}{C}\rho_{d+1,m}^{1}(\varepsilon)\leq a_{11}^{\alpha\alpha}&\leq C\rho_{d+1,m}^{1}(\varepsilon),\quad \alpha=d+1, \cdots, \frac{d(d+1)}{2};\quad\quad\quad  \label{mlem5.8-2'}
\end{align}
\begin{equation}\label{mlem5.8-3'}
|a_{11}^{\alpha \beta}|=|a_{11}^{\beta\alpha}|\leq C\rho_{2(d-1),m}^2(\varepsilon),\quad \alpha,\beta=1, \cdots,\ d,~\alpha\neq\beta.\quad\quad  \quad\quad\quad
\end{equation}
\begin{equation}\label{mlem5.8-4'}
|a_{11}^{\alpha \beta}|=|a_{11}^{\beta\alpha}|\leq C\rho_{2d,m}^2(\varepsilon),\quad \alpha=1, \cdots,\ d,~\beta=d+1, \cdots, \frac{d(d+1)}{2}.
\end{equation}
\begin{equation}\label{mlem5.8-5'}
|a_{11}^{\alpha\beta}|=|a_{11}^{\beta\alpha}|\leq C\rho_{2(d+1),m}^2(\varepsilon),\quad \alpha,\beta=d+1, \cdots, \frac{d(d+1)}{2},~\alpha\neq\beta.\ 
\end{equation}
where
\begin{align*}
\rho_{k,m}^1(\varepsilon)=\begin{cases}1,&m<k,\\
|\log\varepsilon|,&m=k,\\
\varepsilon^{\frac{k-m}{m}},&m>k;
\end{cases}\quad\quad
\rho_{k,m}^2(\varepsilon)=\begin{cases}1,&m<k,\\
|\log\varepsilon|,&m=k,\\
\varepsilon^{\frac{k-m}{2m}},&m>k.
\end{cases}
\end{align*}
\end{lemma}

Using Lemma \ref{lemma6.1}, we have the following improtant proposition, especially the new estimate \eqref{mprop2-2} for $|C_{1}^{\alpha}-C_{2}^{\alpha}|$, $\alpha=d+1,\cdots,\frac{d(d+1)}{2}$.

\begin{prop}\label{mprop2}
Let $C_i^\alpha$ be defined in \eqref{decom_u}, $i=1,2$, $\alpha=1,2,\cdots,\frac{d(d+1)}{2}$. Then $C_{i}^{\alpha}$ are bounded, and if $m\geq d-1$, then, for $\alpha=1,\cdots, d,$
\begin{align}\label{mprop2-1}
|C_{1}^{\alpha}-C_{2}^{\alpha}|\leq\begin{cases} \frac{C}{|\log\varepsilon|}\|\varphi\|_{C^{2}(\partial D)},&     m=d-1,\\
C\varepsilon^{1-\frac{d-1}{m}}\|\varphi\|_{C^{2}(\partial D)},& m>d-1;
\end{cases}\ \ \
\end{align}
and for $\alpha=d+1,\cdots,\frac{d(d+1)}{2}$,
\begin{align}\label{mprop2-2}
|C_{1}^{\alpha}-C_{2}^{\alpha}|\leq\begin{cases} C\|\varphi\|_{C^{2}(\partial D)},& d-1\leq m<d+1,\\
\frac{C}{|\log\varepsilon|}\|\varphi\|_{C^{2}(\partial D)},&     m=d+1,\\
C\varepsilon^{1-\frac{d+1}{m}}\|\varphi\|_{C^{2}(\partial D)},& m>d+1.
\end{cases}
\end{align}
\end{prop}

\begin{proof}
In view of the symmetry of $a_{11}^{\alpha\beta}$, we write it as a block matrix
\begin{gather*}a_{11}=\begin{pmatrix} A&B \\  B^T&D
\end{pmatrix},
\end{gather*}
where
\begin{gather*}A=\begin{pmatrix} a_{11}^{11}&\cdots&a_{11}^{1d} \\\\ \vdots&\ddots&\vdots\\\\a_{11}^{d1}&\cdots&a_{11}^{dd}\end{pmatrix}  ,\quad
B=\begin{pmatrix} a_{11}^{1~d+1}&\cdots&a_{11}^{1~\frac{d(d+1)}{2}} \\\\ \vdots&\ddots&\vdots\\\\a_{11}^{d~d+1}&\cdots&a_{11}^{d~\frac{d(d+1)}{2}}\end{pmatrix} ,\end{gather*}\begin{gather*}
D=\begin{pmatrix} a_{11}^{d+1~d+1}&\cdots&a_{11}^{d+1~\frac{d(d+1)}{2}} \\\\ \vdots&\ddots&\vdots\\\\a_{11}^{\frac{d(d+1)}{2}~d+1}&\cdots&a_{11}^{\frac{d(d+1)}{2}~\frac{d(d+1)}{2}} \end{pmatrix}.
\end{gather*}
Making use of the estimates obtained in Lemma \ref{lemma6.1}, and thanks to Lemma \ref{lemma3.3} below, we can obtain \eqref{mprop2-1} and \eqref{mprop2-2}.
\end{proof}

\begin{lemma}\label{lemma3.3}[Lemma 6.2 in \cite{bll2}]
For $m\geq1$, let $A$, $D$ be $m\times m$ invertible matrices and $B$ and $C$ be $m\times m$ matrices satisfying, for some $0<\theta<1$ and $\gamma>1$,
\begin{equation*}
\|A^{-1}\|\leq\frac{1}{\theta\gamma},\quad \|B\|+\|C\|+\|D^{-1}\|\leq\frac{1}{\theta}.
\end{equation*}
Then there exists $\bar{\gamma}=\bar{\gamma}(m)>1$ and $C(m)>1$, such that if $\gamma\geq\frac{\bar{\gamma}(m)}{\theta^4}$,
\begin{gather*}\begin{pmatrix} A&B \\  C&D
\end{pmatrix}
\end{gather*}
is invertible. Moreover,
\begin{gather*}
\begin{pmatrix} E_{11}&E_{12}\\  E_{12}^T&E_{22}\end{pmatrix}:=\begin{pmatrix}A&B\\C&D\end{pmatrix}^{-1}-\begin{pmatrix}A^{-1}&0\\0&D^{-1}\end{pmatrix},
\end{gather*}
satisfies
\begin{equation*}
\|E_{11}\|\leq\frac{C(m)}{\theta^5\gamma^2},\quad \|E_{12}\|\leq\frac{C(m)}{\theta^3\gamma},\quad \|E_{22}\|\leq\frac{C(m)}{\theta^5\gamma}.
\end{equation*}
\end{lemma}

\begin{proof}[Proof of Theorem \ref{thm3}]

If $m\geq d+1$, then by using  Corollary \ref{mcorol3.4}, and Proposition \ref{mprop2}, we obtain that for $x\in\Omega_{R}$,
\begin{align*}
|\nabla{u}(x)|&\leq\sum_{\alpha=1}^{\frac{d(d+1)}{2}}|C_{1}^\alpha-C_{2}^\alpha||\nabla{v}_{1}^\alpha(x)|+\sum_{\alpha=1}^{\frac{d(d+1)}{2}}|C_{2}^\alpha\nabla({v}_{1}^\alpha+{v}_{2}^\alpha)(x)|
+|\nabla{v}_{3}(x)|\\
&\leq\sum_{\alpha=1}^{d}|C_{1}^\alpha-C_{2}^\alpha||\nabla{v}_{1}^\alpha(x)|+\sum_{\alpha=d+1}^{\frac{d(d+1)}{2}}|C_{1}^\alpha-C_{2}^\alpha||\nabla{v}_{1}^\alpha(x)|+C\|\varphi\|_{C^2(\partial D)}\\
&\leq \begin{cases} C\left(\frac{\varepsilon^{1-\frac{d-1}{m}}}{\varepsilon+|x'|^m}
+\frac{\varepsilon^{1-\frac{d+1}{m}}|x'|}{\varepsilon+|x'|^m}+1\right)\|\varphi\|_{C^2(\partial D)},&m>d+1,\\
 C\left(\frac{\varepsilon^{1-\frac{d-1}{m}}}{\varepsilon+|x'|^m}+\frac{|x'|}{|\log\varepsilon|(\varepsilon+|x'|^m)}+1\right)\|\varphi\|_{C^2(\partial D)},&m=d+1.
 \end{cases}
\end{align*}

If $d-1\leq m<d+1$, then the elements $a_{11}^{\alpha\beta}$, $\alpha,\beta=d+1,\cdots,\frac{d(d+1)}{2}$ are only bounded. So we only can get the estimates for $|C_{1}^\alpha-C_{2}^\alpha|$, $\alpha=1,2,\cdots,d$. As before, see \cite{bll1,bll2},  besides using Corollary \ref{mcorol3.4}, Proposition \ref{mprop2}, we need Lemma \ref{lemma3.3} to obtain the following estimates for $x\in\Omega_{R}$,
\begin{align*}
|\nabla{u}(x)|&\leq\sum_{\alpha=1}^{d}|C_{1}^\alpha-C_{2}^\alpha||\nabla{v}_{1}^\alpha(x)|+\sum_{\alpha=1}^d|C_{2}^\alpha\nabla({v}_{1}^\alpha+{v}_{2}^\alpha)(x)|\\
&\qquad\hspace{1cm}+\sum_{i=1}^{2}\sum_{\alpha=d+1}^{\frac{d(d+1)}{2}}|C_{i}^\alpha\nabla{v}_{i}^\alpha(x)|
+|\nabla{v}_{3}(x)|\\
&\leq\sum_{\alpha=1}^{d}|C_{1}^\alpha-C_{2}^\alpha||\nabla{v}_{1}^\alpha(x)|+C\sum_{i=1}^{2}\sum_{\alpha=d+1}^{\frac{d(d+1)}{2}}|\nabla{v}_{i}^\alpha(x)|+C\|\varphi\|_{C^2(\partial D)}\\
&\leq
\begin{cases}
C\left(\frac{\varepsilon^{1-\frac{d-1}{m}}}{\varepsilon+|x'|^m}+\frac{|x'|}{\varepsilon+|x'|^m}+1\right)\|\varphi\|_{C^2(\partial D)},~~&\mbox{if}~d-1<m<d+1,\\
C\left(\frac{1}{|\log\varepsilon|(\varepsilon+|x'|^m)}+\frac{|x'|}{\varepsilon+|x'|^m}+1\right)\|\varphi\|_{C^2(\partial D)},~~&\mbox{if}~m=d-1.
\end{cases}
\end{align*}
Thus, the proof of Theorem \ref{thm3} is completed.
\end{proof}

\begin{proof}[Proof of Lemma \ref{lemma6.1}] We just list the main differences since the reasoning is the same.

{\bf Step 1.} Estimates of $a_{11}^{\alpha\alpha}$, $\alpha=1,2,\cdots,\frac{d(d+1)}{2}$.

First, for $\alpha=1, \cdots, d$, in view of  (\ref{ellip}) and (\ref{mv1-bounded1}), we have
\begin{align*}
a_{11}^{\alpha\alpha}&=\int_{\Omega}(\mathbb{C}^0e(v_{1}^{\alpha}), e(v_{1}^{\alpha}))dx\leq C\int_{\Omega}|\nabla v_{1}^{\alpha}|^2dx\\
&\leq C\int_0^R\frac{\rho^{d-2}}{\varepsilon+\rho^m}d\rho+C
\leq\begin{cases}C,&m<d-1,\\
C|\log\varepsilon|,&m=d-1,\\
C\varepsilon^{\frac{d-1-m}{m}},&m>d-1.
\end{cases}
\end{align*}
and
\begin{align*}
a_{11}^{\alpha\alpha}&=\int_{\Omega}(\mathbb{C}^0e(v_{1}^{\alpha}), e(v_{1}^{\alpha}))dx\geq \frac{1}{C}\int_{\Omega}|e(v_{1}^{\alpha})|^2dx\\
&\geq \frac{1}{C}\int_{\Omega_R}|\partial_{x_d}(v_{1}^{\alpha})^{\alpha}|^2dx\geq \frac{1}{C}\int_{\Omega_R}|\partial_{x_d}\bar{v}|^2dx\\
&\geq\frac{1}{C}\int_{|x'|<R}\frac{dx'}{\varepsilon+|x'|^m}
\geq\begin{cases}\frac{1}{C},&m<d-1,\\
\frac{1}{C}|\log\varepsilon|,&m=d-1,\\
\frac{1}{C}\varepsilon^{\frac{d-1-m}{m}},&m>d-1.
\end{cases}
\end{align*}
Estimate \eqref{mlem34-1'} is proved.

As for $a_{11}^{\alpha\alpha}$, $\alpha=d+1,\cdots,\frac{d(d+1)}{2}$, we have
\begin{align*}
a_{11}^{\alpha\alpha}&=\int_{\Omega}(\mathbb{C}^0e(v_{1}^{\alpha}), e(v_{1}^{\alpha}))dx\leq C\int_{\Omega}|\nabla v_{1}^{\alpha}|^2dx\\
&\leq\,C\int_{\Omega_R}\frac{|x'|^2}{(\varepsilon+|x'|^m)^2}dx+C\\
&\leq\,C\int_0^R\frac{\rho^{d}}{\varepsilon+\rho^m}d\rho+C
\leq\begin{cases}C,&m<d+1\\
C|\log\varepsilon|,&m=d+1,\\
C\varepsilon^{\frac{d+1-m}{m}},&m>d+1.
\end{cases}
\end{align*}
On the other hand, similar to the proof of $a_{11}^{\alpha\alpha}$ in Lemma 4.2, we have
\begin{align*}
a_{11}^{\alpha\alpha}\geq \frac{1}{C}
\begin{cases}1,&m<d+1\\
|\log\varepsilon|,&m=d+1,\\
\varepsilon^{\frac{d+1-m}{m}},&m>d+1.
\end{cases}
\end{align*}
Thus, \eqref{mlem5.8-2'} holds true.

{\bf Step 2.} Estimates of $a_{11}^{\alpha\beta}$, $\alpha,\beta=1,2,\cdots,d$, $\alpha\neq\beta$.

For $\alpha=1, \cdots, d,~ \beta=1, \cdots, d-1$ with $\alpha\neq\beta$. 
\begin{align*}
a_{11}^{\alpha\beta}&=a_{11}^{\beta\alpha}=-\int_{\partial{D}_{1}}\frac{\partial v_{1}^{\alpha}}{\partial \nu_0}\large\Big|_{+}\cdot\psi_\beta\\
&=-\int_{\partial{D}_{1}}\lambda\left(\sum_{k=1}^d\partial_{x_k} (v_1^{\alpha})^{k}\right)n_\beta
+\mu\sum_{l=1}^d\left(\partial_{x_\beta}(v_1^{\alpha})^{l}+\partial_{x_l}(v_1^{\alpha})^{\beta}\right)n_l.
\end{align*}
Denote
$$\mathrm{I}_{\alpha\beta}:=\int_{\partial{D}_{1}\cap B_R}\left(\sum_{k=1}^d\partial_{x_k} (v_1^{\alpha})^{k}\right)n_\beta,
$$
and
\begin{align*}
\mathrm{II}_{\alpha\beta}:=&\int_{\partial{D}_{1}\cap B_R} \sum_{l=1}^{d}\left(\partial_{x_\beta}(v_1^{\alpha})^{l}+\partial_{x_l}(v_1^{\alpha})^{\beta}\right)n_l\\
=&\int_{\partial{D}_{1}\cap B_R} \sum_{l=1}^{d-1}\left(\partial_{x_\beta}(v_1^{\alpha})^{l}+\partial_{x_l}(v_1^{\alpha})^{\beta}\right)n_l+\int_{\partial{D}_{1}\cap B_R} \partial_{x_\beta}(v_1^{\alpha})^{d}n_d+\int_{\partial{D}_{1}\cap B_R} \partial_{x_d}(v_1^{\alpha})^{\beta} n_d\\
=&:\,\mathrm{II}_{\alpha\beta}^{1}+\mathrm{II}_{\alpha\beta}^{2}+\mathrm{II}_{\alpha\beta}^{3},
\end{align*}
where $$\vec{n}=\frac{(\nabla_{x'}h_1(x'), -1)}{\sqrt{1+|\nabla_{x'}h_1(x')|^2}}.$$
Due to (\ref{h1h2m}), for $k=1,\cdots,d-1$,
\begin{align}\label{mn'} |n_{k}|=\left|\frac{\partial_{x_{k}}h_1(x')}{\sqrt{1+|\nabla_{x'}h_1(x')|^2}}\right|\leq C|x'|^{m-1},\quad~\mbox{and}~ |n_d|=\frac{ 1}{\sqrt{1+|\nabla_{x'}h_1(x')|^2}}\leq1.
\end{align}

For $\alpha=1,\cdots, d, ~\beta=1, \cdots, d-1$, it follows from (\ref{mv1-bounded1}) and (\ref{mn'}) that
\begin{align}\label{a-3'}
|\mathrm{I}_{\alpha\beta}|\leq&\int_{\partial{D}_{1}\cap B_R}\left|\left(\sum_{k=1}^d\partial_{x_k} (v_1^{\alpha})^{k}\right)n_\beta\right|
\leq\int_{\partial{D}_{1}\cap B_R}\frac{C|x'|^{m-1}}{\varepsilon+|x'|^m}dx'
\leq 
\begin{cases}C,&d\geq3,\\
C|\log\varepsilon|,&d=2,
\end{cases}
\end{align}
while
\begin{align*}
|\mathrm{II}_{\alpha\beta}^{1}|\leq\int_{\partial{D}_{1}\cap B_R} \left|\sum_{l=1}^{d-1}\left(\partial_{x_\beta}(v_1^{\alpha})^{l}+\partial_{x_l}(v_1^{\alpha})^{\beta}\right)n_l\right|
\leq\int_{B'_R}\frac{C|x'|^{m-1}}{\sqrt{\varepsilon+|x'|^m}}dx'
\leq\,C,
\end{align*}
and
\begin{align*}
|\mathrm{II}_{\alpha\beta}^{2}|\leq\left|\int_{B'_R} \partial_{x_\beta}(v_1^{\alpha})^{d}n_d\right|\leq\int_{B_R'}\frac{Cdx'}{\sqrt{\varepsilon+|x'|^m}}
\leq\begin{cases}C,&m<2(d-1),\\
C|\log\varepsilon|,&m=2(d-1),\\
C\varepsilon^{\frac{d-1}{m}-\frac{1}{2}},&m>2(d-1).
\end{cases}
\end{align*}
By the definition of $\tilde{v}_1^\alpha$ and (\ref{mnabla_w_i0}),
\begin{align*}
|\mathrm{II}_{\alpha\beta}^{3}|&\leq\left|\int_{\partial{D}_{1}\cap B_R}\partial_{x_d} (v_1^{\alpha})^{\beta} n_d\right|\\
&\leq\left|\int_{\partial{D}_{1}\cap B_R}(\partial_{x_d} (\tilde{v}_1^\alpha)^\beta)n_d\right|+\left|\int_{\partial{D}_{1}\cap B_R}|(\partial_{x_d}( v_{1}^{\alpha}-\tilde{v}_1^\alpha)^\beta)n_d\right|\nonumber\\
&\leq\int_{B'_R}\frac{Cdx'}{\sqrt{\varepsilon+|x'|^m}}
\leq\begin{cases}C,&m<2(d-1),\\
C|\log\varepsilon|,&m=2(d-1),\\
C\varepsilon^{\frac{d-1}{m}-\frac{1}{2}},&m>2(d-1).
\end{cases}
\end{align*}
Here we used the fact that $(\tilde{v}_1^\alpha)^{\beta}=0$ if $\alpha\neq\beta$. Hence,
\begin{align*}
|\mathrm{II}_{\alpha\beta}|
\leq\begin{cases}C,&m<2(d-1),\\
C|\log\varepsilon|,&m=2(d-1),\\
C\varepsilon^{\frac{d-1}{m}-\frac{1}{2}},&m>2(d-1).
\end{cases}
\end{align*}
This, together with (\ref{a-3'}), the boundedness of $|\nabla{v}_1^{\alpha}|$ on $\partial{D}_{1}\setminus B_R$, and the symmetry of $a_{11}^{\alpha\beta}=a_{11}^{\beta\alpha}$, implies that for $ \alpha,\beta=1, \cdots, d$ with $\alpha\neq\beta$,
\begin{align*}
|a_{11}^{\alpha\beta}|=|a_{11}^{\beta\alpha}|\leq |\lambda\|\mathrm{I}_{\alpha\beta}|+|\mu\|\mathrm{II}_{\alpha\beta}|+C\leq\begin{cases}C,&m<2(d-1),\\
C|\log\varepsilon|,&m=2(d-1),\\
C\varepsilon^{\frac{d-1}{m}-\frac{1}{2}},&m>2(d-1).
\end{cases}
\end{align*}

{\bf Step 3.} Estimate $a_{11}^{\alpha\beta}$, $\alpha=1,\cdots, d$, and $\beta=d+1, \cdots, \frac{d(d+1)}{2}$.

For $\alpha=1,\cdots, d$, and $\beta=d+1, \cdots, \frac{d(d+1)}{2}$, we take the case that $\alpha=1$, $\beta=d+1$ for instance. The other cases are the same.  Since $\psi_{d+1}=(-x_{2},x_{1},0£¬\cdots, 0)^{T}$, then using (\ref{mv1-x'}) and the boundedness of $|\nabla{v}_1^{\alpha}|$ on $\partial{D}_{1}\setminus B_R$, we have
\begin{align*}
a_{11}^{1d+1}=&-\int_{\partial{D}_{1}}\frac{\partial v_{1}^{1}}{\partial \nu_0}\large\Big|_{+}\cdot\psi_{d+1}\\
=&\int_{\partial{D}_{1}}\left(\lambda\big(\sum_{k=1}^d\partial_{x_k} (v_1^{1})^{k}\big)n_1
+\mu\sum_{l=1}^d\left(\partial_{x_1}(v_1^{1})^{l}+\partial_{x_l}(v_1^{1})^{1}\right)n_l\right)x_{2}\\
&-\int_{\partial{D}_{1}}\left(\lambda\big(\sum_{k=1}^d\partial_{x_k} (v_1^{1})^{k}\big)n_2
+\mu\sum_{l=1}^d\left(\partial_{x_2}(v_1^{1})^{l}+\partial_{x_l}(v_1^{1})^{2}\right)n_l\right)x_{1}.
\end{align*}
Denote
\begin{align*}
\mathrm{I}_{1d+1}:=&\lambda\int_{\partial{D}_{1}\cap B_R}\big(\sum_{k=1}^d\partial_{x_k} (v_1^{1})^{k}\big)n_1x_2+\mu\int_{\partial{D}_{1}\cap B_R}\sum_{l=1}^{d-1}\left(\partial_{x_1}(v_1^{1})^{l}+\partial_{x_l}(v_1^{1})^{1}\right)n_lx_{2}\\
&+\mu\int_{\partial{D}_{1}\cap B_R}\partial_{x_1}(v_1^{1})^{d}n_dx_2+\mu\int_{\partial{D}_{1}\cap B_R}\partial_{x_d}(v_1^{1})^{1}n_dx_2\\
=&:\,\lambda\mathrm{I}_{1d+1}^{1}+\mu\mathrm{I}_{1d+1}^{2}+\mu\mathrm{I}_{1d+1}^{3}+\mu\mathrm{I}_{1d+1}^{4},
\end{align*}
and
\begin{align*}
\mathrm{II}_{1d+1}:=&\lambda\int_{\partial{D}_{1}\cap B_R}\big(\sum_{k=1}^{d}\partial_{x_k} (v_1^{1})^{k}\big)n_2x_1+\mu\int_{\partial{D}_{1}\cap B_R}\sum_{l=1}^{d-1}\left(\partial_{x_2}(v_1^{1})^{l}+\partial_{x_l}(v_1^{1})^{2}\right)n_lx_{1}\\
&+\mu\int_{\partial{D}_{1}\cap B_R}\partial_{x_2}(v_1^{1})^{d}n_dx_1+\mu\int_{\partial{D}_{1}\cap B_R}\partial_{x_d}(v_1^{1})^{2}n_dx_1\\
=&:\,\lambda\mathrm{II}_{1d+1}^{1}+\mu\mathrm{II}_{1d+1}^{2}+\mu\mathrm{II}_{1d+1}^{3}+\mu\mathrm{II}_{1d+1}^{4}.
\end{align*}
From \eqref{mv1-bounded1} and \eqref{mn'}, we have
\begin{align*}
|\mathrm{I}_{1d+1}^1|=\left|\int_{\partial{D}_{1}\cap B_R}\big(\sum_{k=1}^d\partial_{x_k} (v_1^{1})^{k}\big)n_1x_2\right|\leq\int_{ B_R'}\frac{C|x'|^{m}dx'}{\varepsilon+|x'|^m}\leq C.
\end{align*}
$|\mathrm{I}_{1d+1}^2|$ is similar. As for $|\mathrm{I}_{1d+1}^3|$, from \eqref{mnabla_w_i0}, we can obtain
\begin{align*}
|\mathrm{I}_{1d+1}^{3}|&\leq \left|\int_{\partial{D}_{1}\cap B_R}\partial_{x_1} (\tilde{v}_1^1)^dn_dx_2\right|+\left|\int_{\partial{D}_{1}\cap B_R}\partial_{x_1}( v_{1}^{1}-\tilde{v}_1^1)^dn_dx_2\right|\nonumber\\
&\leq C\int_{B_R'}\frac{|x'|dx'}{\sqrt{\varepsilon+|x'|^m}}\leq C\int_0^R\frac{\rho^{d-1}}{\sqrt{\varepsilon+\rho^m}}d\rho\\
&\leq \begin{cases}C,&m<2d,\\
C|\log\varepsilon|,&m=2d,\\
C\varepsilon^{\frac{2d-m}{2m}},&m>2d.
\end{cases}
\end{align*}
Similarly, we can estimate $\mathrm{I}_{1d+1}^4$,
\begin{align*}
|\mathrm{I}_{1d+1}^{4}|&\leq \left|\int_{\partial{D}_{1}\cap B_R}\partial_{x_d} (\tilde{v}_1^1)^1n_dx_2\right|+\left|\int_{\partial{D}_{1}\cap B_R}\partial_{x_d}( v_{1}^{1}-\tilde{v}_1^1)^dn_dx_2\right|\nonumber\\
&\leq\,C\left|\int_{B_R'}\frac{x_{2}dx'}{\varepsilon+\kappa_{0}|x'|^m}\right|+\int_{B_R'}\frac{C|x'|dx'}{\sqrt{\varepsilon+|x'|^m}}\\
&\leq C\int_0^R\frac{\rho^{d-1}}{\sqrt{\varepsilon+\rho^m}}d\rho\leq\begin{cases}C,&m<2d,\\
C|\log\varepsilon|,&m=2d,\\
C\varepsilon^{\frac{2d-m}{2m}},&m>2d.
\end{cases}
\end{align*}
Therefore,
\begin{align*}
|\mathrm{I}_{1d+1}|\leq \begin{cases}C,&m<2d,\\
C|\log\varepsilon|,&m=2d,\\
C\varepsilon^{\frac{2d-m}{2m}},&m>2d.
\end{cases}
\end{align*}
The estimate of $\mathrm{II}_{1d+1}$  is similar to that of $\mathrm{I}_{1d+1}$.
Combining these estimates, we have
\begin{align*}
|a_{11}^{1~d+1}|\leq \begin{cases}C,&m<2d,\\
C|\log\varepsilon|,&m=2d,\\
C\varepsilon^{\frac{2d-m}{2m}},&m>2d.
\end{cases}
\end{align*}

{\bf Step 4.} Estimates of $a_{11}^{\alpha\beta}$, $\alpha, \beta=d+1,\cdots, \frac{d(d+1)}{2}$ with $\alpha\neq\beta$.

For $\alpha, \beta=d+1,\cdots, \frac{d(d+1)}{2}$ with $\alpha\neq\beta$, we  only   estimate $a_{11}^{d+1~d+2}$ with
$\tilde{v}_1^{d+1}=(-\bar{v}x_2, \bar{v}x_1, 0,\cdots,0)^T$ and $\psi_{d+2}=(-x_{3},0,x_{1}, 0, \cdots, 0)^{T}$ for instance. The other cases are the same.
\begin{align*}
&a_{11}^{d+1~d+2}\\
=&-\int_{\partial{D}_{1}}\frac{\partial v_1^{d+1}}{\partial \nu_0}\large\Big|_{+}\cdot\psi_{d+2}\\
=&\int_{\partial{D}_{1}\cap B_R}\left(\lambda\big(\sum_{k=1}^d\partial_{x_k} (v_1^{d+1})^k\big)n_1
+\mu\sum_{l=1}^d\left(\partial_{x_1}(v_1^{d+1})^l+\partial_{x_l}(v_1^{d+1})^1\right)n_l\right)x_{3}\\
&-\int_{\partial{D}_{1}\cap B_R}\left(\lambda\big(\sum_{k=1}^d\partial_{x_k} (v_1^{d+1})^k\big)n_3
+\mu\sum_{l=1}^d\left(\partial_{x_3}(v_1^{d+1})^l+\partial_{x_l}(v_1^{d+1})^3\right)n_l\right)x_{1}+O(1)\\
=:&~\mathrm{I}_{d+1~d+2}-\mathrm{II}_{d+1~d+2}+O(1).
\end{align*}

{\bf Step 4.1.}
For $d=3$, in view of that $|x_3|\leq C(\varepsilon+|x'|^m)$ on $\partial{D}_{1}$,  we have
\begin{align*}
|\mathrm{I}_{45}|\leq\int_{B_R'}\frac{C|x'\|x_3|}{\varepsilon+|x'|^m}dx'\leq C.
\end{align*}
Denote
\begin{align*}
\mathrm{II}_{45}:=&\lambda\int_{\partial{D}_{1}\cap B_R}\sum_{k=1}^{2}\partial_{x_k} (v_1^4)^kn_3x_1+(\lambda+2\mu)\int_{\partial{D}_{1}\cap B_R}\partial_{x_3}(v_1^4)^3n_3x_{1}\\
&+\mu\int_{\partial{D}_{1}\cap B_R}\sum_{l=1}^{2}(\partial_{x_3}(v_1^4)^l+\partial_{x_l}(v_1^4)^3)n_lx_1+O(1)\\
=&:\,\lambda\mathrm{II}_{45}^{1}+(\lambda+2\mu)\mathrm{II}_{45}^{2}+\mu\mathrm{II}_{45}^{3}+O(1),
\end{align*}
From \eqref{me2.4} and \eqref{mnabla_w_i00}, we have
\begin{align*}
|\mathrm{II}_{45}^{1}|&\leq\left|\int_{\partial{D}_{1}\cap B_R}\sum_{k=1}^{2}\partial_{x_k} (v_1^4)^kn_3x_1\right|\\
&\leq\left|\int_{\partial{D}_{1}\cap B_R}\sum_{k=1}^{2}\partial_{x_k} (\tilde{v}_1^4)^kn_3x_1\right|+\left|\int_{\partial{D}_{1}\cap B_R}\sum_{k=1}^{2}\partial_{x_k} (v_1^4-\tilde{v}_1^4)^kn_3x_1\right|\\
&\leq\int_{B_R'}\frac{C|x'|^{m+1}}{\varepsilon+|x'|^m}dx'+\int_{B_R'}\frac{C|x'|^2}{\sqrt{\varepsilon+|x'|^m}}dx'\\
&\leq\int_0^R\frac{\rho^3}{\sqrt{\varepsilon+\rho^m}}d\rho\leq C\begin{cases}C,&m<8,\\
C|\log\varepsilon|,&m=8,\\
C\varepsilon^{\frac{8-m}{2m}},&m>8.
\end{cases}
\end{align*}
Since $ (\tilde{v}_1^4)^3=0$, we have
\begin{align*}
|\mathrm{II}_{45}^{2}|&\leq\left|\int_{\partial{D}_{1}\cap B_R}\partial_{x_3} (\tilde{v}_1^4)^3n_3x_1\right|+\left|\int_{\partial{D}_{1}\cap B_R}\partial_{x_3} (v_1^4-\tilde{v}_1^4)^3n_3x_1\right|\\
&\leq\int_{B_R'}\frac{C|x'|^2}{\sqrt{\varepsilon+|x'|^m}}dx'\leq\int_0^R\frac{\rho^3}{\sqrt{\varepsilon+\rho^m}}d\rho\\
&\leq \begin{cases}C,&m<8,\\
C|\log\varepsilon|,&m=8,\\
C\varepsilon^{\frac{8-m}{2m}},&m>8.
\end{cases}
\end{align*}
In view of \eqref{mn'}, we have
\begin{align*}
|\mathrm{II}_{45}^{3}|&\leq\left|\int_{\partial{D}_{1}\cap B_R}\sum_{l=1}^{2}(\partial_{x_3}(v_1^4)^l+\partial_{x_l}(v_1^4)^3)n_lx_1\right|\leq\int_{ B_R'}\frac{C|x'|^{m+1}}{\varepsilon+|x'|^m}dx'\leq C.
\end{align*}
Combining these estimates yields
\begin{align*}
|a_{11}^{45}|\leq \begin{cases}C,&m<8,\\
C|\log\varepsilon|,&m=8,\\
C\varepsilon^{\frac{8-m}{2m}},&m>8.
\end{cases}
\end{align*}

{\bf Step 4.2.}
For $d>3$, denote
\begin{align*}
\mathrm{I}_{d+1~d+2}=&\int_{\partial{D}_{1}\cap B_R}\left[\lambda\sum_{k=1}^{d}\partial_{x_k} (v_1^{d+1})^kn_1
x_{3}+\mu\sum_{l=1}^{d-1}\left(\partial_{x_1}(v_1^{d+1})^l+\partial_{x_l}(v_1^{d+1})^1\right)n_lx_{3}\right]\\
&+\int_{\partial{D}_{1}\cap B_R}\mu\partial_{x_1}(v_1^{d+1})^dn_dx_{3}+\int_{\partial{D}_{1}\cap B_R}\mu\partial_{x_d}(v_1^{d+1})^1n_dx_{3}\\
:=&\,\mathrm{I}_{d+1~d+2}^1+\mathrm{I}_{d+1~d+2}^2+\mathrm{I}_{d+1~d+2}^3.
\end{align*}
Since $|n_l|\leq C|x'|^{m-1}$ for $l=1,\cdots,d-1$, we have
\begin{align*}
|\mathrm{I}_{d+1~d+2}^1|\leq\int_{B_R'}\frac{C|x'|^{m+1}}{\varepsilon+|x'|^m}dx'\leq C.
\end{align*}
In view of $(\tilde{v}_1^{d+1})^d=0$, we have
\begin{align*}
|\mathrm{I}_{d+1~d+2}^2|&\leq\left|\int_{\partial{D}_{1}\cap B_R}\partial_{x_2}(v_1^{d+1}-\tilde{v}_1^{d+1})^dn_dx_{3}\right|\\
&\leq\int_{B_R'}\frac{C|x'|^2}{\sqrt{\varepsilon+|x'|^m}}dx'\\
&=\int_0^R\frac{C\rho^d}{\sqrt{\varepsilon+\rho^m}}d\rho\leq \begin{cases}C,&m<2(d+1),\\
C|\log\varepsilon|,&m=2(d+1),\\
C\varepsilon^{\frac{2(d+1)-m}{2m}},&m>2(d+1).
\end{cases}
\end{align*}
Using the symmetry,
\begin{align*}
\left|\int_{\partial{D}_{1}\cap B_R}\partial_{x_d}(\tilde{v}_1^{d+1})^1n_dx_{3}\right|=\left|\int_{B_R'}\frac{x_2x_3}{\varepsilon+|x'|^m}dx'\right|=0,
\end{align*}
then
\begin{align*}
|\mathrm{I}_{d+1~d+2}^3|&\leq\left|\int_{\partial{D}_{1}\cap B_R}\partial_{x_d}(v_1^{d+1}-\tilde{v}_1^{d+1})^1n_dx_{3}\right|\\
&\leq\int_{B_R'}\frac{|x'|^2}{\sqrt{\varepsilon+|x'|^m}}dx'\leq \begin{cases}C,&m<2(d+1),\\
C|\log\varepsilon|,&m=2(d+1),\\
C\varepsilon^{\frac{2(d+1)-m}{2m}},&m>2(d+1).
\end{cases}
\end{align*}
Therefore,
\begin{align*}
|\mathrm{I}_{d+1~d+2}|\leq \begin{cases}C,&m<2(d+1),\\
C|\log\varepsilon|,&m=2(d+1),\\
C\varepsilon^{\frac{2(d+1)-m}{2m}},&m>2(d+1).
\end{cases}
\end{align*}
Similarly, we have
\begin{align*}
|\mathrm{II}_{d+1~d+2}|\leq \begin{cases}C,&m<2(d+1),\\
C|\log\varepsilon|,&m=2(d+1),\\
C\varepsilon^{\frac{2(d+1)-m}{2m}},&m>2(d+1).
\end{cases}
\end{align*}
As a conclusion, for $d\geq3$,
\begin{align*}
|a_{11}^{d+1~d+2}|\leq \begin{cases}C,&m<2(d+1),\\
C|\log\varepsilon|,&m=2(d+1),\\
C\varepsilon^{\frac{2(d+1)-m}{2m}},&m>2(d+1).
\end{cases}
\end{align*}
We finish the proof of Lemma \ref{lemma6.1}.
\end{proof}

\section{Appendix: The proof of Theorem \ref{thm2.1} and \ref{thm6.1}}

For the completeness, we give a sketch of the proof of Theorem \ref{thm2.1} in Subsection \ref{subsec_thm2.1} and that of Theorem \ref{thm6.1} in Subsection \ref{subsec_thm6.1}, although the idea, especially the iteration technique, is mainly from \cite{bll1,bll2}.

\subsection{Proof of theorem \ref{thm2.1}}\label{subsec_thm2.1}

Recalling that
\begin{equation*}
\delta(x'):=\varepsilon+h_{1}(x')-h_{2}(x'),\qquad\forall~(x',x_{d})\in\Omega_{R},
\end{equation*}
and
$$d(x'):=d_{\Sigma'}(x')=dist(x',\Sigma').$$
By (\ref{h1-h21}), we have
\begin{align}
\frac{1}{C}(\varepsilon+d^2(z'))\leq\delta(z')\leq C(\varepsilon+d^2(z')).\nonumber
\end{align}
By a direct calculation, we obtain that for $k, j=1,\cdots, d-1$, and $x\in\Omega_{R}$,
\begin{align}
|\partial_{x_{k}x_{j}}\bar{v}(x)|\leq\frac{C}{\varepsilon+d^2(x')}, \quad|\partial_{x_{k}x_{d}}\bar{v}(x)|\leq\frac{Cd(x')}{(\varepsilon+d^2(x'))^2},\quad \partial_{x_{d}x_{d}}\bar{v}(x)=0.\label{ee2.4}
\end{align}
Due to (\ref{e2.4}), and (\ref{ee2.4}), for $l=1,2,\cdots,d$, and $k,j=1,2,\cdots,d-1$, for $x\in\Omega_{R}$,
\begin{align}
&|\partial_{x_{k}x_{j}}\tilde{v}_{l}(x)|\nonumber\\
&\leq\frac{C|\psi^{l}(x',\varepsilon+h_{1}(x'))|}{\varepsilon+d^2(x')}
+C\left(\frac{d(x')}{\varepsilon+d^2(x')}+1\right)\|\nabla\psi^{l}\|_{L^{\infty}}+C\|\nabla^{2}\psi^{l}\|_{L^{\infty}},\label{eq1.8}\\
&|\partial_{x_{k}x_{d}}\tilde{v}_{l}(x)|
\leq\frac{Cd(x')}{(\varepsilon+d^2(x'))^2}|\psi^{l}(x',\varepsilon+h_{1}(x'))|+\frac{C}{\varepsilon+d^2(x')}\|\nabla\psi^{l}\|_{L^{\infty}},\label{eq1.9}\\
&\partial_{x_{d}x_{d}}\tilde{v}_{l}(x)=0.\label{eq1.10}
\end{align}
Here and throughout this section, for simplicity we use $\|\nabla\psi\|_{L^{\infty}}$ and $\|\nabla^{2}\psi\|_{L^{\infty}}$ to denote $\|\nabla\psi\|_{L^{\infty}(\partial{D}_{1})}$ and $\|\nabla^{2}\psi\|_{L^{\infty}(\partial{D_1})}$, respectively.

Set
$$\widehat{\Omega}_{s}(z'):=\left\{~x\in\Omega_{2R} ~\big|~ |x'-z'|<s~\right\}, \quad\forall~ 0\leq{s}\leq R.$$
It follows from \eqref{equ_tildeu}  and (\ref{eq1.8})--(\ref{eq1.10}) that for $x\in\Omega_{R}$, $l=1,2,\cdots,d$,
\begin{align}\label{f}
|\mathcal{L}_{\lambda,\mu}\tilde{v}_{l}|\leq\, C|\nabla^2\tilde{v}_{l}|
\leq&\left(\frac{C}{\varepsilon+d^2(x')}+\frac{Cd(x')}{(\varepsilon+d^2(x'))^2}\right)|\psi^{l}(x',\varepsilon+h_{1}(x'))|\nonumber\\
&+\frac{C}{\varepsilon+d^2(x')}\|\nabla\psi^{l}\|_{L^{\infty}}+C\|\nabla^{2}\psi^{l}\|_{L^{\infty}},
\end{align}
where $C$ is independent of $\varepsilon$.

Let
\begin{equation}\label{def_w}
w_l:=v_l-\tilde{v}_l,\qquad l=1,2,\cdots,d.
\end{equation}

\begin{proof}[Proof of Theorem \ref{thm2.1}]
\noindent{\bf Step 1.}
Let $v_l\in H^1(\Omega; \mathbb{R}^{d})$ be a weak solution of (\ref{eq_v2.1}).  We first prove that
\begin{align}\label{lem2.2equ}
\int_{\Omega}|\nabla w_l|^2dx\leq C\|\psi^{l}\|^{2}_{C^{2}(\partial{D}_{1})},\qquad\,l=1,2,\cdots,d.
\end{align}

For simplicity, we denote
$$w:=w_{l},\quad\mbox{and}\quad \tilde{v}:=\tilde{v}_{l}.$$
Thus, $w$ satisfies
\begin{align}\label{eq2.6}
\begin{cases}
  \mathcal{L}_{\lambda,\mu}w=-\mathcal{L}_{\lambda,\mu}\tilde{v},&
\hbox{in}\  \Omega,  \\
w=0, \quad&\hbox{on} \ \partial\Omega.
\end{cases}
\end{align}
Multiplying the equation in (\ref{eq2.6}) by $w$ and applying integration by parts, we get
\begin{align}\label{integrationbypart}
\int_{\Omega}\left(\mathbb{C}^0e(w),e(w)\right)dx=\int_{\Omega}w\left(\mathcal{L}_{\lambda,\mu}\tilde{v}\right)dx.
\end{align}

Since $w=0$ on $\partial\Omega$, by the Poincar\'{e} inequality,
\begin{align}\label{poincare_inequality}
\|w\|_{L^2(\Omega\setminus\Omega_R)}\leq\,C\|\nabla w\|_{L^2(\Omega\setminus\Omega_R)}.
\end{align}
Note that the above constant $C$ is independent of $\varepsilon$.
Using the Sobolev trace embedding theorem,
\begin{align}\label{trace}
\int\limits_{\scriptstyle |x'|={R},\atop\scriptstyle
h_2(x')<x_{d}<{\varepsilon}+h_1(x')\hfill}|w|ds\leq\ C \left(\int_{\Omega\setminus\Omega_{R}}|\nabla w|^2dx\right)^{\frac{1}{2}}.
\end{align}
According to (\ref{eq1.7}), we have
\begin{align}\label{nablax'tildeu}
&\int_{\Omega_{R}}|\nabla_{x'}\tilde{v}|^2dx\nonumber\\
&\leq C\int_{|x'|<R}(\varepsilon+h_1(x')-h_2(x'))\left(\frac{d^2(x')|\psi^{l}(x',\varepsilon+h_{1}(x'))|^{2}}{(\varepsilon+d^2(x'))^2}+\|\nabla\psi^{l}\|_{L^{\infty}}^{2}\right)dx'\nonumber\\
&\leq C\|\psi^{l}\|_{C^{1}(\partial{D}_{1})}^{2},
\end{align}
where $C$ depends only on $d$ and $\kappa_0$.

The first Korn's inequality together with \eqref{2.15}, \eqref{eu}, \eqref{nabla_vtilde_outside} and \eqref{poincare_inequality} implies
\begin{align}
\int_{\Omega}|\nabla w|^2dx\leq &\,2\int_{\Omega}|e(w)|^2dx\nonumber\\
\leq&\,C\left|\int_{\Omega_R}w(\mathcal{L}_{\lambda,\mu}\tilde{v})dx\right|+C\left|\int_{\Omega\setminus\Omega_R}w(\mathcal{L}_{\lambda,\mu}\tilde{v})dx\right|\nonumber\\
\leq&\,C\left|\int_{\Omega_R}w(\mathcal{L}_{\lambda,\mu}\tilde{v})dx\right|+C\|\psi^{l}\|_{C^2(\partial{D}_{1})}\int_{\Omega\setminus\Omega_R}|w|dx\nonumber\\
\leq&\,C\left|\int_{\Omega_R}w(\mathcal{L}_{\lambda,\mu}\tilde{v})dx\right|+C\|\psi^{l}\|_{C^2(\partial{D}_{1})}\left(\int_{\Omega\setminus\Omega_R}|\nabla w|^2\right)^{1/2},\nonumber
\end{align}
while, due to \eqref{eq1.10}, (\ref{trace}) and (\ref{nablax'tildeu}),
\begin{align}
&\left|\int_{\Omega_R}w(\mathcal{L}_{\lambda,\mu}\tilde{v})dx\right|
\leq C\sum_{k+l<2d}\left|\int_{\Omega_{R}}w\partial_{x_kx_l}\tilde{v}dx\right|\nonumber\\
\leq&\, C\int_{\Omega_{R}}|\nabla w\|\nabla_{x'}\tilde{v}|dx+\int\limits_{\scriptstyle |x'|={R},\atop\scriptstyle
h_2(x')<x_{d}<\varepsilon+h_1(x')\hfill}C|\nabla_{x'} \tilde{v}\|w|dx\nonumber \\
\leq&\,  C\left(\int_{\Omega_{R}}|\nabla w|^2dx\right)^{\frac{1}{2}}\left(\int_{\Omega_{R}}|\nabla_{x'}\tilde{v}|^2dx\right)^{\frac{1}{2}}+C\|\psi^{l}\|_{C^1(\partial{D}_{1})}\left(\int_{\Omega\setminus\Omega_R}|\nabla w|^2dx\right)^{\frac{1}{2}}\nonumber\\
\leq&\,  C\|\psi^{l}\|_{C^{1}(\partial{D}_{1})}\left(\int_{\Omega}|\nabla w|^2dx\right)^{\frac{1}{2}}.\nonumber
\end{align}
Therefore,
\begin{align*}
\int_{\Omega}|\nabla w|^2dx\leq  C\|\psi^{l}\|_{C^{2}(\partial{D}_{1})}\left(\int_{\Omega}|\nabla w|^2dx\right)^{\frac{1}{2}}.
\end{align*}

\noindent{\bf Step 2.} Proof of
\begin{align}\label{step2}
 \int_{\widehat{\Omega}_\delta(z')}|\nabla w|^2dx
 &\leq C\delta^{d-1}\left(|\psi^{l}(z',\varepsilon+h_{1}(z'))|^2+\delta(\|\psi^{l}\|_{C^2(\partial D_1)}^2+1)\right),
\end{align}
where $\delta=\delta(z')=\varepsilon+h_{1}(z')-h_{2}(z')$, and
\begin{equation*}
\widehat{\Omega}_{t}(z'):=\left\{x\in \mathbb{R}^{d}~\big|~h_{2}(x')<x_{d}<\varepsilon+h_{1}(x'),~|x'-z'|<{t}\right\}.
\end{equation*}

The following iteration scheme we used is similar in spirit to that used in  \cite{bll1,llby}. For $0<t<s<R$, let $\eta$ be a smooth cutoff function satisfying $\eta(x')=1$ if $|x'-z'|<t$, $\eta(x')=0$ if $|x'-z'|>s$, $0\leq\eta(x')\leq1$ if $t\leq|x'-z'|\leq\,s$, and $|\nabla_{x'}\eta(x')|\leq\frac{2}{s-t}$. Multiplying the equation in \eqref{eq2.6} by $w\eta^{2}$ and integrating by parts
leads  to
\begin{align}\label{integrationbypart}
\int_{\widehat{\Omega}_{s}(z')}\left(\mathbb{C}^0e(w),e(w\eta^{2})\right)dx=\int_{\widehat{\Omega}_{s}(z')}\left(w\eta^2\right)\mathcal{L}_{\lambda,\mu}\tilde{v}dx.
\end{align}
By the first Korn's inequality and the standard arguments, we have
\begin{align}\label{lemma2.2-2}
\int_{\widehat{\Omega}_{s}(z')}(\mathbb{C}^0e(w), e(\eta^2w))dx\geq\frac{1}{C}\int_{\widehat{\Omega}_{s}(z')}|\eta\nabla w|^2dx-C\int_{\widehat{\Omega}_{s}(z')}|\nabla\eta|^2|w|^2dx.
\end{align}
For the right hand side of (\ref{integrationbypart}), in view of H\"older inequality and Cauchy inequality,
\begin{align}
\left|\int_{\widehat{\Omega}_{s}(z')}(\eta^2w)\mathcal{L}_{\lambda,\mu}\tilde{v}dx\right|
&\leq\left(\int_{\widehat{\Omega}_{s}(z')}|w|^2dx\right)^{\frac{1}{2}}\left(\int_{\widehat{\Omega}_{s}(z')}|\mathcal{L}_{\lambda,\mu}\tilde{v}|^2dx\right)^{\frac{1}{2}}\nonumber\\
&\leq\frac{C}{(s-t)^2}\int_{\widehat{\Omega}_{s}(z')}|w|^2dx+C(s-t)^2\int_{\widehat{\Omega}_{s}(z')}|\mathcal{L}_{\lambda,\mu}\tilde{v}|^2dx.\nonumber
\end{align}
This, together with (\ref{integrationbypart}) and  (\ref{lemma2.2-2}), implies that
\begin{align}\label{ww}
\int_{\widehat{\Omega}_t(z)}|\nabla w|^2dx\leq\frac{C}{(s-t)^2}\int_{\widehat{\Omega}_{s}(z')}|w|^2dx +C(s-t)^2\int_{\widehat{\Omega}_{s}(z')}|\mathcal{L}_{\lambda,\mu}\tilde{v}|^2dx.
\end{align}

We know that $w=0$ on $\Gamma_{R}^-$, where $\Gamma_{R}^{-}:=\{x\in\mathbb{R}|~x_{d}=h_{2}(x'),|x'|<R\}$. By using (\ref{h1h2})--\eqref{h1h3} and H\"{o}lder inequality, we obtain\begin{align}\label{w}
\int_{\widehat{\Omega}_{s}(z')}|w|^2dx&=\int_{\widehat{\Omega}_{s}(z')}\left|\int_{h_2(x')}^{x_{d}}\partial_{x_{d}}w(x', \xi)d\xi\right|^2dx\nonumber\\
&\leq \int_{\widehat{\Omega}_{s}(z')}(\varepsilon+h_1(x')-h_2(x'))\int_{h_2(x')}^{\varepsilon+h_1(x')}|\nabla w(x',\xi)|^2d\xi\,dx\nonumber\\
&\leq\,C\delta^{2}(z')\int_{\widehat{\Omega}_{s}(z')}|\nabla w|^2dx.
\end{align}
It follows from (\ref{f}) and the mean value theorem that
\begin{align}
&\int_{\widehat{\Omega}_{s}(z')}|\mathcal{L}_{\lambda,\mu}\tilde{v}|^2dx\nonumber\\
\leq&\, |\psi^{l}(z',\varepsilon+h_{1}(z'))|^2\int_{\widehat{\Omega}_{s}(z')}\left(\frac{C}{\varepsilon+d^2(x')}+\frac{Cd(x')}{(\varepsilon+d^2(x'))^2}\right)^2dx\nonumber\\
&+\|\nabla\psi^{l}\|_{L^\infty}^2\int_{\widehat{\Omega}_{s}(z')}\left(\frac{C}{\varepsilon+d^2(x')}+\frac{Cd(x')}{(\varepsilon+d^2(x'))^2}\right)^2|x'-z'|^2dx\nonumber\\
&+\|\nabla\psi^{l}\|_{L^\infty}^2\int_{\widehat{\Omega}_{s}(z')}\left(\frac{C}{\varepsilon+d^2(x')}\right)^2dx+C\delta(z')s^{d-1}\|\nabla^2\psi^{l}\|_{L^\infty}^2\nonumber\\
\leq&\, C|\psi^{l}(z',\varepsilon+h_{1}(z'))|^2 \int_{|x'-z'|<s}\frac{dx'}{(\varepsilon+d^2(x'))^2}\nonumber\\
&+C\|\nabla\psi^{l}\|_{L^\infty}^2 \int_{|x'-z'|<s}\left(\frac{1}{\varepsilon+d^2(x')}+\frac{s^{2}}{(\varepsilon+d^2(x'))^2}\right)dx'+C\delta(z')s^{d-1}\|\nabla^2\psi^{l}\|_{L^\infty}^2.\label{ff}
\end{align}

We further divide into three cases to derive the iteration formula by using \eqref{ww}.

{\bf Case 1.} For $z'\in\Sigma'_{-\sqrt{\varepsilon}}:=\{x'\in\Sigma'~|~ d(x',\partial\Sigma')>\sqrt{\varepsilon}\}$ and $0<s<\sqrt{\varepsilon}$, where $\delta(z')=\varepsilon$.

We here assume that $B'_{\sqrt{\varepsilon}}\subset\Sigma'$ (otherwise, start from Case 2), then
\begin{align}\label{int_w1}
\int_{\widehat{\Omega}_{s}(z')}|w|^{2}
&\leq C\varepsilon^{2}\int_{\widehat{\Omega}_{s}(z')}|\nabla{w}|^{2}.
\end{align}
Denote
$$F(t):=\int_{\widehat{\Omega}_{t}(z')}|\nabla{w}|^{2}.$$
It follows from \eqref{ww} and \eqref{int_w1} that
\begin{equation}\label{energy_w1}
F(t)\leq\left(\frac{c_{1}\varepsilon}{s-t}\right)^{2}F(s),
\end{equation}
where $c_{1}$ is a {\it universal constant} but independent of $|\Sigma'|$.

Let $k=\left[\frac{1}{4c_{1}\sqrt{\varepsilon}}\right]$ and $t_{i}=\delta+2c_{1}i\varepsilon, i=1,2,\cdots,k$. Then by \eqref{energy_w1} with $s=t_{i+1}$ and $t=t_{i}$, we have
$$F(t_{i})\leq\frac{1}{4}F(t_{i+1}).$$
After $k$ iterations, using \eqref{lem2.2equ}, we have
$$F(t_{0})\leq(\frac{1}{4})^{k}F(t_{k})\leq C\varepsilon^{d-1}.$$
Here we take $\varepsilon^{d-1}$ to keep consistent with Case 2 below. Therefore, for some sufficiently small  $\varepsilon>0$,
\begin{equation*}
\int_{\widehat{\Omega}_{\delta}(z')}\left|\nabla{w}\right|^{2}dx\leq C\varepsilon^{d-1}.
\end{equation*}

{\bf Case 2.} For $z'\in\Sigma'_{\sqrt{\varepsilon}}\setminus\Sigma'_{-\sqrt{\varepsilon}}$, where $\Sigma'_{\sqrt{\varepsilon}}:=\{x'\in B'_{R}~| ~dist(x', \Sigma')<\sqrt{\varepsilon}\}$ and $0<s<\sqrt{\varepsilon}$, we have $\varepsilon\leq\delta(z')\leq C\varepsilon$.

By means of (\ref{w}) and (\ref{ff}), we have
\begin{align}
&\int_{\widehat{\Omega}_{s}(z')}|w|^2dx\leq C\varepsilon^2\int_{\widehat{\Omega}_{s}(z')}|\nabla w|^2dx,\label{w21}
\end{align}
and
\begin{align}
&\int_{\widehat{\Omega}_{s}(z')}|\mathcal{L}_{\lambda,\mu}\tilde{v}|^2dx\nonumber\\
&\leq C|\psi^{l}(z',\varepsilon+h_{1}(z'))|^2\frac{s^{d-1}}{\varepsilon^{2}}+C\|\nabla\psi^{l}\|_{L^\infty}^2 \frac{s^{d-1}}{\varepsilon}+C\varepsilon\,s^{d-1}\|\nabla^2\psi^{l}\|_{L^\infty}^2.\label{f21}
\end{align}

By (\ref{ww}), (\ref{w21}) and (\ref{f21}), for some universal constant $c_1>0$, we get for $0<t<s<\sqrt{\varepsilon}$,
\begin{align}\label{F}
F(t)\leq &\,\left(\frac{c_1\varepsilon}{s-t}\right)^2F(s)+C(s-t)^2s^{d-1} \cdot\nonumber\\
&\qquad\left(\frac{|\psi^{l}(z',\varepsilon+h_{1}(z'))|^2}{\varepsilon^2}+\frac{\|\nabla\psi^{l}\|_{L^\infty}^2}{\varepsilon}
+\varepsilon\|\nabla^2\psi^{l}\|_{L^\infty}^2\right).
\end{align}
Let  $t_i=\delta+2c_1i\varepsilon$, $i=0, 1, \cdots$ and $k=\left[\frac{1}{4c_1\sqrt{\varepsilon}}\right]+1$, then $$\frac{c_1\varepsilon}{t_{i+1}-t_i}=\frac{1}{2}.$$
Using (\ref{F}) with $s=t_{i+1}$ and $t=t_i$, we obtain
\begin{align*}
F(t_i)&\leq\frac{1}{4}F(t_{i+1})+C(i+1)^{d-1}\varepsilon^{d-1}\cdot\\
&\qquad\left(|\psi^{l}(z',\varepsilon+h_{1}(z'))|^2+\varepsilon(\|\nabla\psi^{l}\|_{L^\infty}^2+\|\nabla^2\psi^{l}\|_{L^\infty}^2)\right) , \quad i=0,1,2,\cdots, k.
\end{align*}
After $k$ iterations, making use of \eqref{lem2.2equ}, we have, for sufficiently small $\varepsilon$,
\begin{align*}
F(t_0)&\leq \big(\frac{1}{4}\big)^kF(t_k)+C\varepsilon^{d-1}\sum_{i=1}^k\big(\frac{1}{4}\big)^{i-1}(i+1)^{d-1}\cdot\\
&\qquad\left(|\psi^{l}(z',\varepsilon+h_{1}(z'))|^2+\varepsilon(\|\nabla\psi^{l}\|_{L^\infty}^2+\|\nabla^2\psi^{1}\|_{L^\infty}^2)\right)\\
&\leq \big(\frac{1}{4}\big)^kF(\sqrt{\varepsilon})+C\varepsilon^{d-1}\left(|\psi^{l}(z',\varepsilon+h_{1}(z'))|^2+\varepsilon(\|\nabla\psi^{l}\|_{L^\infty}^2+\|\nabla^2\psi^{l}\|_{L^\infty}^2)\right)\\
&\leq C\varepsilon^{d-1}\left(|\psi^{l}(z',\varepsilon+h_{1}(z'))|^2+\varepsilon\|\psi^{l}\|_{C^2(\partial{D}_{1})}^2\right),
\end{align*}
here we used the fact that $\big(\frac{1}{4}\big)^k\leq\big(\frac{1}{4}\big)^{\frac{1}{4c_{1}\sqrt{\varepsilon}}}\leq\,\varepsilon^{d-1}$ if $\varepsilon$ sufficiently small. This implies that
\begin{align*}
\|\nabla w\|_{L^2(\widehat{\Omega}_\delta(z'))}^{2}\leq  C\varepsilon^{d-1}\left(|\psi^{l}(z',\varepsilon+h_{1}(z'))|^2+\varepsilon\|\psi^{l}\|_{C^2(\partial{D}_{1})}^2\right).
\end{align*}

{\bf Case 3.} For $z'\in B'_{R}\setminus\Sigma'_{\sqrt{\varepsilon}}$, and $0<s<\frac{2}{3}d(z')$, we have $Cd^{2}(z')\leq\delta(z')\leq(C+1)d^{2}(z')$.

By using (\ref{w}) and (\ref{ff}) again, we have
\begin{align*}
&\int_{\widehat{\Omega}_{s}(z')}|w|^2dx\leq Cd^4(z')\int_{\widehat{\Omega}_{s}(z')}|\nabla w|^2dx,\\
&\int_{\widehat{\Omega}_{s}(z')}|\mathcal{L}_{\lambda,\mu}\tilde{v}|^2dx\leq C|\psi^{l}(z',\varepsilon+h_{1}(z'))|^2\frac{s^{d-1}}{d^4(z')}\\
& \hspace{4cm}+C\|\nabla\psi^{l}\|_{L^\infty}^2 \frac{s^{d-1}}{d^2(z')}+Cd^2(z')s^{d-1}\|\nabla^2\psi^{l}\|_{L^\infty}^2.
\end{align*}
Thus, for $0<t<s<\frac{2d(z')}{3}$,
 \begin{align}\label{F1}
F(t)\leq&\, \left(\frac{c_2d^2(z')}{s-t}\right)^2F(s)+C(s-t)^2s^{d-1}\cdot\nonumber\\
&\qquad\left(\frac{|\psi^{l}(z',\varepsilon+h_{1}(z'))|^2}{d^4(z')}+\frac{\|\nabla\psi^{l}\|_{L^\infty}^2}{d^2(z')}
+d^2(z')\|\nabla^2\psi^{l}\|_{L^\infty}^2\right),
\end{align}
where $c_{2}$ is another universal constant. Taking the same iteration procedure as in Case 1, setting  $t_i=\delta+2c_2id^2(z')$, $i=0, 1, \cdots$ and $k=\left[\frac{1}{4c_2d(z')}\right]+1$,
by (\ref{F1}) with $s=t_{i+1}$ and $t=t_i$, we have, for $i=0,1,2,\cdots, k$,
\begin{align*}
F(t_i)\leq&\,\frac{1}{4}F(t_{i+1})+C(i+1)^{d-1}d^{2(d-1)}(z')\cdot\\
&\qquad\left(|\psi^{l}(z',\varepsilon+h_{1}(z'))|^2+d^{2}(z')(\|\nabla\psi^{l}\|_{L^\infty}^2+\|\nabla^2\psi^{l}\|_{L^\infty}^2)\right).
\end{align*}
Similarly, after $k$ iterations, we have
\begin{align*}
F(t_0)\leq&\, \big(\frac{1}{4}\big)^kF(t_k)+C\sum_{i=1}^k\big(\frac{1}{4}\big)^{i-1}(i+1)^{d-1}d^{2(d-1)}(z')\cdot\\
&\qquad\left(|\psi^{l}(z',\varepsilon+h_{1}(z'))|^2+d^{2}(z')(\|\nabla\psi^{l}\|_{L^\infty}^2+\|\nabla^2\psi^{l}\|_{L^\infty}^2)\right)\\
\leq&\, \big(\frac{1}{4}\big)^kF(d(z'))\\
&+Cd^{2(d-1)}(z')\left(|\psi^{l}(z',\varepsilon+h_{1}(z'))|^2+d^{2}(z')(\|\nabla\psi^{l}\|_{L^\infty}^2+\|\nabla^2\psi^{1}\|_{L^\infty}^2)\right)\\
\leq&\, Cd^{2(d-1)}(z')\left(|\psi^{l}(z',\varepsilon+h_{1}(z'))|^2+d^{2}(z')\|\psi^{l}\|_{C^2(\partial{D}_{1})}^2\right),
\end{align*}
which implies that
\begin{align*}
\|\nabla w\|_{L^2(\widehat{\Omega}_\delta(z'))}^2
&\leq Cd^{2(d-1)}(z')\left(|\psi^{l}(z',\varepsilon+h_{1}(z'))|^2+d^{2}(z')\|\psi^{l}\|_{C^2(\partial{D}_{1})}^2\right).
\end{align*}
Therefore, \eqref{step2} is proved.

\noindent{\bf Step 3.}
Proof of that for $l=1,2,\cdots,d$,
\begin{align}\label{equa2.9v}
|\nabla w_l(x)|\leq \frac{C|\psi^{l}(x',\varepsilon+h_{1}(x'))|}{\sqrt{\delta(x')}}+C\|\psi^{l}\|_{C^2(\partial{D}_{1})},\quad\forall~x\in\Omega_{R}.
\end{align}

By the rescaling argument, Sobolev embedding theorem, $W^{2,p}$ estimate and bootstrap argument, the same as in \cite{bll1,bll2}, we have
\begin{align}\label{Linfty_estimate}
\|\nabla w\|_{L^\infty(\widehat{\Omega}_{\delta/2}(z))}\leq\frac{C}{\delta}\left(\delta^{1-\frac{d}{2}}\|\nabla w\|_{L^2(\widehat{\Omega}_\delta(z'))}+\delta^2\|\mathcal{L}_{\lambda,\mu}\tilde{v}\|_{L^\infty(\widehat{\Omega}_\delta(z'))}\right).
\end{align}
By  (\ref{f}) and (\ref{step2}), we have
\begin{align*}
\delta^{-\frac{d}{2}}\|\nabla w\|_{L^2(\widehat{\Omega}_\delta(z'))}
&\leq \frac{C}{\sqrt{\delta}}|\psi^{l}(z',\varepsilon+h_{1}(z'))|
+C\|\psi^{l}\|_{C^2(\partial D_1)},\end{align*}
and
\begin{align*}
\delta \|\mathcal{L}_{\lambda,\mu}\tilde{v}\|_{L^\infty(\widehat{\Omega}_\delta(z'))}\leq \frac{C}{\sqrt{\delta}}|\psi^{l}(z',\varepsilon+h_{1}(z'))|
+C(\|\nabla\psi^{l}\|_{L^\infty}+\|\nabla^2\psi^{l}\|_{L^\infty}).
\end{align*}
Plugging these estimates above into \eqref{Linfty_estimate} yields \eqref{equa2.9v}.

Consequently, by (\ref{eq1.7}), \eqref{eq1.7a} and \eqref{def_w}, we have for sufficiently small $\varepsilon$ and $x\in\Omega_{R}$,
\begin{align}\label{estimate_vl}
\frac{|\psi^{l}(x',\varepsilon+h_{1}(x'))|}{C(\varepsilon+d^2(x'))}&\leq|\nabla v_l(x',x_{d})|\leq \frac{C|\psi^{l}(x',\varepsilon+h_{1}(x'))|}{\varepsilon+d^2(x')}+C\|\psi^{l}\|_{C^2(\partial{D}_{1})}.
\end{align}

\noindent{\bf Step 4.} The completion of the proof of Theorem {thm2.1}.

By using \eqref{estimate_vl} and the decomposition of $\nabla{v}$, \eqref{equ_nablav},
\begin{align*}
|\nabla v(x)|&\leq\sum_{l=1}^{d}|\nabla{v}_{l}|\leq \frac{C|\psi(x',\varepsilon+h_{1}(x'))|}{\varepsilon+d^2(x')}+C\|\psi\|_{C^2(\partial D_1)},\quad\,x\in\Omega_R.
\end{align*}
Note that  for any $x\in\Omega\setminus\Omega_{R}$, by using the standard interior estimates and boundary estimates for elliptic systems \eqref{eq1.1} (see Agmon et al. \cite{AD1} and \cite{AD2}), we have
 $$\|\nabla v\|_{L^\infty(\Omega\setminus \Omega_{R})}\leq C\|\psi\|_{C^2(\partial D_1)}.$$  The proof of Theorem \ref{thm2.1} is completed.
\end{proof}

\subsection{Proof of theorem \ref{thm6.1}}\label{subsec_thm6.1}

The proof is similar to the Theorem \ref{thm2.1}, we only list the main differences. We define $\bar{v}$ by \eqref{vvd} as before.
By a direct calculation, we obtain that for $k, j=1,\cdots, d-1$, and $x\in\Omega_{R}$,
\begin{equation*}
|\partial_{x_{k}x_{j}}\bar{v}(x)|\leq\frac{C|x'|^{m-2}}{\varepsilon+|x'|^m}, \quad|\partial_{x_{k}x_{d}}\bar{v}(x)|\leq\frac{C|x'|^{m-1}}{(\varepsilon+|x'|^m)^2},\quad \partial_{x_{d}x_{d}}\bar{v}(x)=0.
\end{equation*}
Define $\tilde{v}_{l}$ by \eqref{equ_tildeu} as before, then we have
\begin{align}\label{meq1.7}
|\partial_{x_{k}}\tilde{v}_{l}(x)|\leq\frac{C|x'|^{m-1}|\psi^{l}(x',\varepsilon+h_{1}(x'))|}{\varepsilon+|x'|^m}+C\|\nabla\psi^{l}\|_{L^{\infty}},
\end{align}
\begin{align}\label{meq1.7a}
|\partial_{x_{d}}\tilde{v}_{l}(x)|\leq\frac{C|\psi^{l}(x',\varepsilon+h_{1}(x'))|}{\varepsilon+|x'|^m};
\end{align}
and
\begin{align*}
&|\partial_{x_{k}x_{j}}\tilde{v}_{l}(x)|\nonumber\\
&\leq\frac{C|x'|^{m-2}|\psi^{l}(x',\varepsilon+h_{1}(x'))|}{\varepsilon+|x'|^m}
+C\left(\frac{|x'|^{m-1}}{\varepsilon+|x'|^m}+1\right)\|\nabla\psi^{l}\|_{L^{\infty}}+C\|\nabla^{2}\psi^{l}\|_{L^{\infty}},\\
&|\partial_{x_{k}x_{d}}\tilde{v}_{l}(x)|
\leq\frac{C|x'|^{m-1}}{(\varepsilon+|x'|^m)^2}|\psi^{l}(x',\varepsilon+h_{1}(x'))|+\frac{C}{\varepsilon+|x'|^m}\|\nabla\psi^{l}\|_{L^{\infty}},\\
&\partial_{x_{d}x_{d}}\tilde{v}_{l}(x)=0.
\end{align*}
Therefore, for $x\in\Omega_{R}$, $l=1,2,\cdots,d$,
\begin{align}\label{mf}
|\mathcal{L}_{\lambda,\mu}\tilde{v}_{l}|\leq\, C|\nabla^2\tilde{v}_{l}|
\leq&\left(\frac{C|x'|^{m-2}}{\varepsilon+|x'|^m}+\frac{C|x'|^{m-1}}{(\varepsilon+|x'|^m)^2}\right)|\psi^{l}(x',\varepsilon+h_{1}(x'))|\nonumber\\
&+\frac{C}{\varepsilon+|x'|^m}\|\nabla\psi^{l}\|_{L^{\infty}}+C\|\nabla^{2}\psi^{l}\|_{L^{\infty}},
\end{align}

Similarly, let
\begin{equation}\label{mdef_w}
w_l:=v_l-\tilde{v}_l,\qquad l=1,2,\cdots,d.
\end{equation}
We only list the main differences.

\noindent{\bf Step 1. } Similar to \eqref{nablax'tildeu}, we have
\begin{align}\label{mnablax'tildeu}
&\int_{\Omega_{R}}|\nabla_{x'}\tilde{v}_l|^2dx\nonumber\\
&\leq C\int_{|x'|<R}(\varepsilon+|x'|^m)\left(\frac{|x'|^{2(m-1)}|\psi^{l}(x',\varepsilon+h_{1}(x'))|^{2}}{(\varepsilon+|x'|^m)^2}+\|\nabla\psi^{l}\|_{L^{\infty}}^{2}\right)dx'\nonumber\\
&\leq C\|\psi^{l}\|_{C^{1}(\partial{D}_{1})}^{2},
\end{align}
Similar to \eqref{lem2.2equ}, we can obtain
\begin{align}\label{mstep1}
\int_{\Omega}|\nabla w_l|^2dx\leq C\|\psi^{l}\|^{2}_{C^{2}(\partial{D}_{1})},\qquad\,l=1,2,\cdots,d.
\end{align}

\noindent{\bf Step 2. }Proof of
\begin{align}\label{mstep2}
 \int_{\widehat{\Omega}_\delta(z')}|\nabla w_l|^2dx
 &\leq C\delta^{d-1}\left(|\psi^{l}(z',\varepsilon+h_{1}(z'))|^2+\delta(\|\psi^{l}\|_{C^2(\partial D_1)}^2+1)\right).
\end{align}
For simplicity, we denote
$$w:=w_{l},\quad\mbox{and}\quad \tilde{v}:=\tilde{v}_{l}.$$
Similar to \eqref{w}, we have
\begin{align}\label{mw}
\int_{\widehat{\Omega}_{s}(z')}|w|^2dx&=\int_{\widehat{\Omega}_{s}(z')}\left|\int_{h_2(x')}^{x_{d}}\partial_{x_{d}}w(x', \xi)d\xi\right|^2dx\nonumber\\
&\leq \int_{\widehat{\Omega}_{s}(z')}(\varepsilon+h_1(x')-h(x'))\int_{h_2(x')}^{\varepsilon+h_1(x')}|\nabla w(x',\xi)|^2d\xi\,dx\nonumber\\
&\leq\,C(\varepsilon+(|z'|+s)^{m})^2\int_{\widehat{\Omega}_{s}(z')}|\nabla w|^2dx.
\end{align}
Similar to \eqref{ff}, we obtain
\begin{align}
&\int_{\widehat{\Omega}_{s}(z')}|\mathcal{L}_{\lambda,\mu}\tilde{v}|^2dx\nonumber\\
\leq&\, |\psi^{l}(z',\varepsilon+h_{1}(z'))|^2\int_{\widehat{\Omega}_{s}(z')}\left(\frac{C|x'|^{m-2}}{\varepsilon+|x'|^m}+\frac{C|x'|^{m-1}}{(\varepsilon+|x'|^m)^2}\right)^2dx\nonumber\\
&+\|\nabla\psi^{l}\|_{L^\infty}^2\int_{\widehat{\Omega}_{s}(z')}\left(\frac{C|x'|^{m-2}}{\varepsilon+|x'|^m}+\frac{C|x'|^{m-1}}{(\varepsilon+|x'|^m)^2}\right)^2|x'-z'|^2dx\nonumber\\
&+\|\nabla\psi^{l}\|_{L^\infty}^2\int_{\widehat{\Omega}_{s}(z')}\left(\frac{C}{\varepsilon+|x'|^m}\right)^2dx+C\delta(z')s^{d-1}\|\nabla^2\psi^{l}\|_{L^\infty}^2\nonumber\\
\leq&\, C|\psi^{l}(z',\varepsilon+h_{1}(z'))|^2 \int_{|x'-z'|<s}\left(\frac{|x'|^{2(m-2)}}{\varepsilon+|x'|^m}+\frac{|x'|^{2m-2}}{(\varepsilon+|x'|^m)^3}\right)dx'\nonumber\\
&+C\|\nabla\psi^{l}\|_{L^\infty}^2 \int_{|x'-z'|<s}\left(\frac{|x'|^{2(m-2)}s^2}{\varepsilon+|x'|^m}+\frac{|x'|^{2m-2}s^{2}}{(\varepsilon+|x'|^m)^3}+\frac{1}{\varepsilon+|x'|^m}\right)dx'\nonumber\\
&+C\delta(z')s^{d-1}\|\nabla^2\psi^{l}\|_{L^\infty}^2.\label{mff}
\end{align}
{\bf Case 1.} For $0\leq |z'| \leq \sqrt[m]{\varepsilon}$,  (i.e. $\varepsilon\leq\delta(z')\leq\,C\varepsilon$), and $0<t<s<\sqrt[m]{\varepsilon}$.

By means of (\ref{mw}) and (\ref{mff}), we have
\begin{align}
&\int_{\widehat{\Omega}_{s}(z')}|w|^2dx\leq C\varepsilon^2\int_{\widehat{\Omega}_{s}(z')}|\nabla w|^2dx,\label{mw21}
\end{align}
and
\begin{align}
&\int_{\widehat{\Omega}_{s}(z')}|\mathcal{L}_{\lambda,\mu}\tilde{v}|^2dx\nonumber\\
&\leq C|\psi^{l}(z',\varepsilon+h_{1}(z'))|^2\frac{\varepsilon^{1-\frac{2}{m}}}{\varepsilon^{2}}s^{d-1}+C\|\nabla\psi^{l}\|_{L^\infty}^2 \frac{s^{d-1}}{\varepsilon}+C\varepsilon\,s^{d-1}\|\nabla^2\psi^{l}\|_{L^\infty}^2.\label{mf21}
\end{align}

Denote $$F(t):=\int_{\widehat{\Omega}_t(z)}|\nabla w|^2dx.$$ By (\ref{ww}), (\ref{mw21}) and (\ref{mf21}), for some universal constant $c_1>0$, we get for $0<t<s<\sqrt[m]{\varepsilon}$,
\begin{align}\label{F2}
F(t)\leq &\,\left(\frac{c_1\varepsilon}{s-t}\right)^2F(s)+C(s-t)^2s^{d-1} \cdot\nonumber\\
&\qquad\left(\frac{\varepsilon^{1-\frac{2}{m}}|\psi^{l}(z',\varepsilon+h_{1}(z'))|^2}{\varepsilon^2}+\frac{\|\nabla\psi^{l}\|_{L^\infty}^2}{\varepsilon}
+\varepsilon\|\nabla^2\psi^{l}\|_{L^\infty}^2\right).
\end{align}
Let  $t_i=\delta+2c_1i\varepsilon$, $i=0, 1, \cdots$ and $k=\left[\frac{1}{4c_1\sqrt[m]{\varepsilon}}\right]+1$, then $$\frac{c_1\varepsilon}{t_{i+1}-t_i}=\frac{1}{2}.$$
Using (\ref{F2}) with $s=t_{i+1}$ and $t=t_i$, we obtain
\begin{align*}
F(t_i)&\leq\frac{1}{4}F(t_{i+1})+C(i+1)^{d-1}\varepsilon^{d-1}\cdot\\
&\qquad\left(|\psi^{l}(z',\varepsilon+h_{1}(z'))|^2\varepsilon^{1-\frac{2}{m}}+\varepsilon(\|\nabla\psi^{l}\|_{L^\infty}^2+\|\nabla^2\psi^{l}\|_{L^\infty}^2)\right) , \quad i=0,1,2,\cdots, k.
\end{align*}
After $k$ iterations, making use of \eqref{mstep1}, we have, for sufficiently small $\varepsilon$,
\begin{align*}
F(t_0)&\leq \big(\frac{1}{4}\big)^kF(t_k)+C\varepsilon^{d-1}\sum_{i=1}^k\big(\frac{1}{4}\big)^{i-1}(i+1)^{d-1}\cdot\\
&\qquad\left(|\psi^{l}(z',\varepsilon+h_{1}(z'))|^2\varepsilon^{1-\frac{2}{m}}+\varepsilon(\|\nabla\psi^{l}\|_{L^\infty}^2+\|\nabla^2\psi^{1}\|_{L^\infty}^2)\right)\\
&\leq \big(\frac{1}{4}\big)^kF(\sqrt[m]{\varepsilon})+C\varepsilon^{d-1}\left(|\psi^{l}(z',\varepsilon+h_{1}(z'))|^2\varepsilon^{1-\frac{2}{m}}+\varepsilon(\|\nabla\psi^{l}\|_{L^\infty}^2+\|\nabla^2\psi^{l}\|_{L^\infty}^2)\right)\\
&\leq C\varepsilon^{d-1}\left(|\psi^{l}(z',\varepsilon+h_{1}(z'))|^2\varepsilon^{1-\frac{2}{m}}+\varepsilon\|\psi^{l}\|_{C^2(\partial{D}_{1})}^2\right),
\end{align*}
here we used the fact that $\big(\frac{1}{4}\big)^k\leq\big(\frac{1}{4}\big)^{\frac{1}{4c_{1}\sqrt[m]{\varepsilon}}}\leq\,\varepsilon^{d-1}$ if $\varepsilon$ sufficiently small. This implies that for $0\leq |z'| \leq \sqrt[m]{\varepsilon}$,
\begin{align*}
\|\nabla w\|_{L^2(\widehat{\Omega}_\delta(z'))}^{2}\leq  C\varepsilon^{d-1}\left(|\psi^{l}(z',\varepsilon+h_{1}(z'))|^2+\varepsilon\|\psi^{l}\|_{C^2(\partial{D}_{1})}^2\right).
\end{align*}

{\bf Case 2.} For $\sqrt[m]{\varepsilon}\leq |z'|<R$,  (that is, $C|z'|^{m}\leq\delta(z')\leq(C+1)|z'|^{m}$), $0<t<s<\frac{2|z'|}{3}$.

By using (\ref{mw}) and (\ref{mff}) again, we have
$$\int_{\widehat{\Omega}_{s}(z')}|w|^2dx\leq C|z'|^{2m}\int_{\widehat{\Omega}_{s}(z')}|\nabla w|^2dx,
$$
and
\begin{align*}
&\int_{\widehat{\Omega}_{s}(z')}|\mathcal{L}_{\lambda,\mu}\tilde{v}|^2dx\\
& \leq C|\psi^{l}(z',\varepsilon+h_{1}(z'))|^2\frac{s^{d-1}}{|z'|^{m+2}}+C\|\nabla\psi^{l}\|_{L^\infty}^2 \frac{s^{d-1}}{|z'|^m}+C|z'|^{m}s^{d-1}\|\nabla^2\psi^{l}\|_{L^\infty}^2.
\end{align*}
Thus, for $0<t<s<\frac{2|z'|}{3}$,
 \begin{align}\label{F1}
F(t)\leq&\, \left(\frac{c_2|z'|^m}{s-t}\right)^2F(s)+C(s-t)^2s^{d-1}\cdot\nonumber\\
&\qquad\left(\frac{|\psi^{l}(z',\varepsilon+h_{1}(z'))|^2}{|z'|^{m+2}}+\frac{\|\nabla\psi^{l}\|_{L^\infty}^2}{|z'|^m}
+|z'|^{m}\|\nabla^2\psi^{l}\|_{L^\infty}^2\right),
\end{align}
where $c_{2}$ is another universal constant. Taking the same iteration procedure as in Case 1, setting  $t_i=\delta+2c_2i|z'|^m$, $i=0, 1, \cdots$ and $k=\left[\frac{1}{4c_2|z'|}\right]+1$,
by (\ref{F1}) with $s=t_{i+1}$ and $t=t_i$, we have, for $i=0,1,2,\cdots, k$,
\begin{align*}
F(t_i)\leq&\,\frac{1}{4}F(t_{i+1})+C(i+1)^{d-1}|z'|^{m(d-1)}\cdot\\
&\qquad\left(|z'|^{m-2}|\psi^{l}(z',\varepsilon+h_{1}(z'))|^2+|z'|^{m}(\|\nabla\psi^{l}\|_{L^\infty}^2+\|\nabla^2\psi^{l}\|_{L^\infty}^2)\right).
\end{align*}
Similarly, after $k$ iterations, we have
\begin{align*}
F(t_0)\leq&\, \big(\frac{1}{4}\big)^kF(t_k)+C\sum_{i=1}^k\big(\frac{1}{4}\big)^{i-1}(i+1)^{d-1}|z'|^{m(d-1)}\cdot\\
&\qquad\left(|z'|^{m-2}|\psi^{l}(z',\varepsilon+h_{1}(z'))|^2+|z'|^{m}(\|\nabla\psi^{l}\|_{L^\infty}^2+\|\nabla^2\psi^{l}\|_{L^\infty}^2)\right)\\
\leq&\, \big(\frac{1}{4}\big)^kF(|z'|)\\
&+C|z'|^{m(d-1)}\left(|z'|^{m-2}|\psi^{l}(z',\varepsilon+h_{1}(z'))|^2+|z'|^{m}(\|\nabla\psi^{l}\|_{L^\infty}^2+\|\nabla^2\psi^{1}\|_{L^\infty}^2)\right)\\
\leq&\, C|z'|^{m(d-1)}\left(|z'|^{m-2}|\psi^{l}(z',\varepsilon+h_{1}(z'))|^2+|z'|^{m}\|\psi^{l}\|_{C^2(\partial{D}_{1})}^2\right),
\end{align*}
which implies that, for $ \sqrt[m]{\varepsilon}\leq |z'|<R$,
\begin{align*}
\|\nabla w\|_{L^2(\widehat{\Omega}_\delta(z'))}^2
&\leq C|z'|^{m(d-1)}\left(|\psi^{l}(z',\varepsilon+h_{1}(z'))|^2+|z'|^{m}\|\psi^{l}\|_{C^2(\partial{D}_{1})}^2\right).
\end{align*}
Therefore, \eqref{mstep2} is proved.

\noindent{\bf Step 3. }
Proof of that for $l=1,2,\cdots,d$,
\begin{align}\label{mequa2.9}
|\nabla w_l(x)|\leq \frac{C|\psi^{l}(x',\varepsilon+h_{1}(x'))|}{\sqrt{\delta(x')}}+C\|\psi^{l}\|_{C^2(\partial{D}_{1})},\quad\forall~x\in\Omega_{R}.
\end{align}

Similar to \eqref{Linfty_estimate}, we have
\begin{align}\label{mLinfty_estimate}
\|\nabla w\|_{L^\infty(\widehat{\Omega}_{\delta/2}(z))}\leq\frac{C}{\delta}\left(\delta^{1-\frac{d}{2}}\|\nabla w\|_{L^2(\widehat{\Omega}_\delta(z'))}+\delta^2\|\mathcal{L}_{\lambda,\mu}\tilde{v}\|_{L^\infty(\widehat{\Omega}_\delta(z'))}\right).
\end{align}
By  (\ref{mf}) and (\ref{mstep2}), we have
\begin{align*}
\delta^{-\frac{d}{2}}\|\nabla w\|_{L^2(\widehat{\Omega}_\delta(z'))}
&\leq \frac{C}{\sqrt{\delta}}|\psi^{l}(z',\varepsilon+h_{1}(z'))|
+C\|\psi^{l}\|_{C^2(\partial D_1)},\end{align*}
and
\begin{align*}
\delta \|\mathcal{L}_{\lambda,\mu}\tilde{v}\|_{L^\infty(\widehat{\Omega}_\delta(z'))}\leq \frac{C}{\sqrt{\delta}}|\psi^{l}(z',\varepsilon+h_{1}(z'))|
+C\|\psi^{l}\|_{C^2(\partial D_1)}.
\end{align*}
Plugging these estimates above into \eqref{mLinfty_estimate} yields \eqref{mequa2.9}.

Consequently, by (\ref{meq1.7}), \eqref{meq1.7a} and \eqref{mdef_w}, we have for sufficiently small $\varepsilon$ and $x\in\Omega_{R}$,
\begin{align}\label{mestimate_vl}
|\nabla v_l(x',x_{d})|\leq \frac{C|\psi^{l}(x',\varepsilon+h_{1}(x'))|}{\varepsilon+|x'|^m}+C\|\psi^{l}\|_{C^2(\partial{D}_{1})}.
\end{align}

Now we prove Theorem \ref{thm6.1}. 
\begin{proof}[Proof of Theorem \ref{thm6.1}]
By using \eqref{mestimate_vl} and the decomposition of $\nabla{v}$, \eqref{equ_nablav},
\begin{align*}
|\nabla v(x)|&\leq\sum_{l=1}^{d}|\nabla{v}_{l}|\leq \frac{C|\psi(x',\varepsilon+h_{1}(x'))|}{\varepsilon+|x'|^m}+C\|\psi\|_{C^2(\partial D_1)},\quad\,x\in\Omega_R.
\end{align*}
Note that  for any $x\in\Omega\setminus\Omega_{R}$, by using the standard interior estimates and boundary estimates for elliptic systems \eqref{eq1.1} (see Agmon et al. \cite{AD1,AD2}), we have
 $$\|\nabla v\|_{L^\infty(\Omega\setminus \Omega_{R})}\leq C\|\psi\|_{C^2(\partial D_1)}.$$  The proof of Theorem \ref{thm6.1} is completed.
\end{proof}

\noindent{\bf{\large Acknowledgements.}} H.G. Li would like to thank Professor YanYan Li for his encouragements and constant supports. H.J. Ju was partially supported by NSFC (11471050). H.G. Li was partially supported by NSFC (11571042, 11631002), Fok Ying Tung Education Foundation (151003).

\noindent{\bf{\large Conflict of interest}} The authors declare that they have no conflict of interest.

\end{document}